\documentclass[a4paper,11pt,oneside]{amsart}

\usepackage[utf8]{inputenc}
\usepackage[T1]{fontenc}
\usepackage[english]{babel}

\usepackage{graphicx}
\usepackage{enumitem}
\usepackage{mathrsfs}
\usepackage{amsmath,amsthm,amssymb}
\usepackage{dsfont}
\usepackage{hyperref}

\usepackage[toc,page]{appendix}

\usepackage{subfigure}
\usepackage{subcaption}

\usepackage{nicefrac, xfrac}

\usepackage{nameref}

\newtheorem{innercustomgeneric}{\customgenericname}
\providecommand{\customgenericname}{}
\newcommand{\newcustomtheorem}[2]{%
  \newenvironment{#1}[1]
  {%
   \renewcommand\customgenericname{#2}%
   \renewcommand\theinnercustomgeneric{##1}%
   \innercustomgeneric
  }
  {\endinnercustomgeneric}
}

\newcustomtheorem{customassumptions}{Assumptions}
\newcustomtheorem{customassumption}{Assumption}

\usepackage[a4paper,top=3cm,bottom=3cm,left=2.5cm,right=2.5cm]{geometry}
\numberwithin{equation}{section}

\newcommand{\R}{\mathbb{R}}
\newcommand{\E}{\mathbb{E}}
\newcommand{\J}{\mathsf{J}}
\newcommand{\F}{\mathcal{F}}
\newcommand{\prob}{\mathbb{P}}
\newcommand{\A}{\mathcal{A}}

\newcommand{\1}{\mathds{1}}
\newcommand{\tonde}[1]{\left({#1}\right)}

\newcommand{\eps}{\varepsilon}

\let\liminf\relax
\DeclareMathOperator*{\liminf}{\underline{lim}}

\let\limsup\relax
\DeclareMathOperator*{\limsup}{\overline{lim}}

\DeclareMathOperator{\supp}{supp}

\theoremstyle{plain}
\newtheorem{thm}{Theorem}[section]
\newtheorem{lemma}[thm]{Lemma}
\newtheorem{prop}[thm]{Proposition}
\newtheorem{corollary}{Corollary}[thm]

\theoremstyle{definition}
\newtheorem{definition}{Definition}

\newtheorem{remark}{Remark}

\title[Cooperation, Correlation and Competition in Ergodic Singular Control Games]{Cooperation, Correlation and Competition in Ergodic \textit{N}-player Games and Mean-field Games of Singular Controls: A Case Study}
\date{\today}

\author[F.~Cannerozzi]{Federico~Cannerozzi\textsuperscript{\MakeLowercase{a},\MakeLowercase{b},1}}
\author[G.~Ferrari]{Giorgio~Ferrari\textsuperscript{\MakeLowercase{a},2}}
\thanks{\noindent \textsuperscript{a} Center for Mathematical Economics (IMW), Bielefeld University, Universit\"atsstrasse 25, 33615, Bielefeld, Germany.
\\
\noindent \textsuperscript{b} Department of Mathematics ``Federigo Enriques'', University of Milan, Via Saldini 50, 20133, Milan, Italy.}
\thanks{\noindent
\noindent \textsuperscript{1} E-mail: \href{mailto:federico.cannerozzi@uni-bielefeld.de}{federico.cannerozzi@uni-bielefeld.de}.
\\
\noindent \textsuperscript{2} E-mail: \href{mailto:giorgio.ferrari@uni-bielefeld.de}{giorgio.ferrari@uni-bielefeld.de}}

\begin{document}

\begin{abstract}
We consider a class of $N$-player games and mean-field games of singular controls with ergodic performance criterion, providing a benchmark case for irreversible investment games featuring mean-field interaction and strategic complementarities.
The state of each player follows a geometric Brownian motion, controlled additively through a nondecreasing process, while agents seek to maximize a long-term average reward functional with a power-type instantaneous profit, under strategic complementarity.
We explore three different notions of optimality, which, in the mean-field limit, correspond to the mean-field control solution, mean-field coarse correlated equilibria, and mean-field Nash equilibria.
We explicitly compute equilibria in the three cases and compare them numerically, in terms of yielded payoffs and existence conditions. 
Finally, we show that the mean-field control and mean-field equilibria can approximate the cooperative and competitive equilibria, respectively, in the corresponding $N$-player game when $N$ is sufficiently large.
Our analysis of the mean-field control problem features a novel Lagrange multiplier approach, which proves crucial in establishing the approximation result, while the treatment of mean-field coarse correlated equilibria necessitates a new, specifically tailored definition for the stationary setting.
\end{abstract}

\maketitle


\smallskip

{\textbf{MSC2020 subject classification}}: 91A11, 91A15, 91A16, 49N80.

\section{Introduction}

In this paper, we investigate ergodic stochastic games of singular control in both competitive and cooperative settings, considering scenarios with a finite number $N$ of players as well as in the mean-field limit. In the formulation with $N$ players, each symmetric player, indexed by $i=1,2,\dots,N$, seeks to maximize a long-term average reward functional. The instantaneous profit at time $t \geq 0$ is given by $\pi(X^i_t,\theta^N_t)=(X^i_t)^{\alpha}(\theta^N_t)^{\beta}$, where $\alpha, \beta \in (0,1)$. Here, $X^i_t$ represents the current level of the state variable for agent $i$, and $\theta^N_t=\frac{1}{N-1}\sum_{j\neq i}X^j_t$ denotes the empirical average of the state processes of all the other $N-1$ agents. Each agent $i$ can control the geometric dynamics of $X^i$ by increasing its level through a nondecreasing control process $\nu^i$, with the cost of control being proportional to the effort expended.

The game under study makes an important case study for ergodic MFGs of irreversible investment under strategic complementarities.
In this regard, the state variable of each agent can be the output or the goodwill stock, which is increased by irreversible investment or advertising, respectively.
The monopolistic case ($\beta = 0$) has a well-known role of benchmark example in literature (see, among the many others, \cite{chiarolla2005irreversible,pham_guo2005SPA,pham2006explicit_solution_irreverible,riedel2011irreversible,steg2012irreversible}).
We introduce competition among $N$ identical firms, in such a way that it shows strategic complementarity, in that each firm's profit function $\pi(x,\theta)$ is strictly increasing in both its own state $x$ and in the interaction term, resulting in the mixed derivative of $\pi$ being strictly positive.
Many contexts naturally give rise to this scenario, such as Cournot oligopoly with complementary products or advertising games (see, e.g., \cite{vives2005complementarities,vives2005games,vives2018strategic}).
The instantaneous profit can be derived through an isoelastic inverse demand function, depending on each firm's productivity and of the aggregate level of production or goodwill in the entire market, through a price index that features constant elasticity of the substitutes, in a similar fashion of \cite[p. 7, footnote 5]{achdou2014pdemodels} or \cite[Section 2.1]{calvia2024existenceuniquenessresultsmeanfield}.
Finally, the ergodic structure of the reward functional we consider is relevant in the context of investment into public goods, in which it might be important to take care of the payoffs received by successive generations.

\smallskip
We take into account both the cooperative and competitive behavior of the agents/players.
Regarding the first one, we introduce the problem of a central planner, who seeks to maximize the average of the rewards of all $N$ agents.
As for the competitive behavior, we introduce the notion of coarse correlated equilibria (CCE), which extends the more common one of Nash equilibria (NE), by introducing correlation between players' strategies without requiring them to cooperate.
CCEs can be interpreted as follows in an $N$-player setting:
A moderator, or correlation device, picks a strategy profile for the $N$ players randomly according to some publicly know distribution; then, she recommends it privately to the players.
Before the lottery is run, each player has to decide whether to commit to the moderator recommendation (whatever it will be), assuming that all other players commit, only knowing the lottery distribution.
If a player commits, then she is communicated in private her (and only her) selected strategy, and must follow it.
Instead, if a player deviates, she will do so without any information on the outcome of the lottery, assuming that all other players follow the private suggestion they receive.
A lottery is a CCE if every player prefers to commit rather than unilaterally deviate, assuming that all others do commit.

\smallskip
Given the complexity of constructing equilibria for $N$-player games in continuous-time and space, by relying on the theory of mean-field games (see the pioneer works \cite{huang_malhame_caines,lasry_lions}), we introduce the mean-field version of the previously described stochastic game.
We explicitly construct the solution to the mean-field central planner control problem, and we determine sufficient conditions for the existence of coarse correlated equilibria (based on suitable recommendations of the moderator), as well as the mean-field NEs.
In the mean-field game, the representative agent reacts to the long-term average of the distribution of the population, which is represented by a scalar parameter $\theta$. The stationary one-dimensional setting of the mean-field game and control problem allows for explicit characterizations of the equilibria (see also \cite{basei2022nonzero, cao2023stationary, cao2022mfgs,  christensen2021competition} and references therein in the context of singular/impulse control games).
Our definitions and analysis are justified by proving that mean-field equilibria define $\varepsilon$-equilibria for the related $N$-player games, with vanishing approximation error.

\vspace{0.25cm}

\textbf{Our contributions.} 
Our main results are as follows. First, to the best of our knowledge, this is the first paper that constructs the explicit solution to an ergodic mean-field singular stochastic control problem (Theorem \ref{mfc:thm:optimal_control}) and proves that its solution can approximate the solution to a central planner problem aiming to achieve Pareto efficiency in the game with $N$ players (see Theorem \ref{central_planner:thm:approximation}).
We construct the mean-field solution using a Lagrange-multiplier approach, which transforms the original McKean-Vlasov control problem into a two-stage optimization problem, in which one first optimizes over the admissible control variables and then over the mean-field parameter.
We stress that this methodology allows us to separate the two aspects of the MFC problem into distinct steps, to be carried out sequentially. This approach, which is typical of MFGs, is novel in the mean-field singular control framework.
Moreover, we characterize the Lagrange multiplier in terms of the optimal control for the mean-field central planner problem (see Lemma \ref{central_planner:lemma:envelope_theorem}).
In particular, we show that the Lagrange multiplier admits a probabilistic representation as the derivative of the mean-field control problem's value function with respect to the mean-field parameter (see Lemma \ref{central_planner:lemma:envelope_theorem} and Remark \ref{rmk:derivative_value_function} below).
Such a probabilistic representation of the Lagrange multiplier is the new key ingredient for suitably applying the Law of Large Numbers and completing the proof of the approximation result in Theorem \ref{central_planner:thm:approximation}.
As a byproduct of our analysis, we highlight the derivation of novel first-order conditions for optimality in ergodic singular stochastic control problems (see Lemma \ref{lemma:first_order_condition} below), which are of independent interest.

Secondly, in the competitive setting, we introduce the concept of mean-field coarse correlated equilibria in stationary MFGs.
Introduced implicitly in \cite{hannan1957} and explicitly by \cite{moulin_vial1978}, CCEs generalize both Nash equilibria, by allowing agents to adopt correlated strategies without any cooperation, and correlated equilibria (CE) (see Aumann's \cite{aumann1974,aumann1987}).
CCEs have been shown in game theory and computational literature to arise naturally from no-regret adaptive learning procedures (\cite{hannan1957},\cite{hart_mascolell_regret_based},\cite[Section 17.4]{roughgarden_2016}). On the contrary, agents are proved to actually behave according to a Nash equilibrium only under strong rationality assumptions (see e.g. \cite{gilboa1989nash}).
Moreover, incorporating a mediator has been proven useful in at least two ways.
First, it makes possible for CCEs to outperform NEs in terms of payoff: we refer to \cite{DokkaTrivikram2022Edia,Moulin2014Coarse} for contributions in emissions abatement games, and to \cite{neyman1997correlated,moulinraysengupta2014} for contributions regarding potential games. In particular, in this class of games the only CE is given by the Nash equilibrium itself, which makes important to consider the more general notion of CCEs.
Moreover, correlation between players is important to coordinate the agents between equilibria, in the case of multiple equilibria (see the famous ``Battle of Sexes'' and ``Games of Chicken'' games in \cite[Chapter 8, Examples 8.1 and 8.3]{maschler_solan_game_theory}, where this phenomenon can be appreciated already by taking into account CEs).
This feature can potentially bring an important benefit to the analysis, as the mean-field game under study presents strategic complementarity and strategic complementarities are intimately related to the existence of multiple equilibria (see the seminal \cite{vives1990nash} and also \cite{adlakha2013mean, alvarez2023price, alvarez2023strategic, dianetti2022strong, dianetti2024multiple,  dianetti2021submodular, dianetti2023unifying} for contributions on mean-field games).
We determine sufficient conditions for the existence of coarse correlated equilibria, with singular and regular (absolutely continuous) recommendations for the competitive mean-field game of singular controls (see Propositions \ref{cce:regular:prop_optimality} and \ref{cce:singular:prop_optimality} below).
The goodness of our definition is justified by showing that any mean-field CCE induces a sequence of approximate CCEs in the underlying $N$-playe game with underlying error.
Through a numerical analysis, we show that, in the case in which the correlation device has Gamma distribution, there may exist infinitely many mean-field CCEs which outperform the Nash equilibrium, both in the cases of the singular and the regular recommendation, and we compare it with the MFC solution reward.
We perform a comparative statics, in order to highlight the role of the parameters on the existence of mean-field CCEs which outperform the mean-field NEs.

Finally, we completely characterize the Nash equilibria in the ergodic mean-field game and prove that their existence and uniqueness depend on the strength of the strategic complementarity, measured by the parameter $\beta \in (0,1)$ (see Theorem \ref{mfg:thm:nash_eq} below). In particular, if $0 < \beta < 1-\alpha$ or $1-\alpha < \beta < 1$, then a unique Nash mean-field equilibrium exists, where the state process is reflected upwards \`a la Skorohod at an explicitly given barrier. On the other hand, if $\beta=1-\alpha$, either infinitely many equilibria exist, each of reflecting type, or none exist.
The non-existence of NE can be compensated by the existence of CCEs.
Indeed, we are able to numerically show that mean-field CCEs may exist even when mean-field NEs do not, thus highlighting the value of correlation in the game.
The occurrence of this phenomenon is due to the nature of our methodology for computing mean-field CCEs, which does not involve the usual optimization-fixed point scheme used to compute mean field NEs.
Since our results are still valid even when the mean-field NE fixed-point condition fails to hold, existence of mean-field CCEs may hold even when mean-field NEs fail to exist.

Despite the specific setting in which the game is formulated (geometric Brownian dynamics and profit function of power type), the analysis reveals a rich structure of the solution while also requiring technical results and arguments.
Given the benchmark nature of the model, it is important to provide useful methodology of explicitly computing equilibria, to identify the assumptions which yield existence or uniqueness, and to compare them.
This is made possible by the explicit dependence of the solutions on the parameters. 
In this regard, the numerical illustrations of Section \ref{sec:numerics} provide insights and operational contribution.

To the best of our knowledge, our paper is the first to provide a comprehensive analysis of a stationary mean-field singular stochastic game with strategic complementarities by explicitly characterizing its Pareto efficient outcome, as well as its Nash and coarse correlated equilibria.

\vspace{0.25cm}

\noindent \textbf{Related literature.} Our paper contributes to various streams of literature. Firstly, we contribute to the literature on one-dimensional singular stochastic control problems with ergodic performance criteria. Among others, we refer to \cite{alvarez2018stationary, benevs1980some, cohen2022optimal, jack2006singular, karatzas1983class, weerasinghe2002stationary} where explicit solutions to ergodic bounded-variation stochastic control problems have been derived.

Secondly, we add to the literature on stationary mean-field games with ergodic performance criterion.
Those games have nowadays received many contributions when considering regular controls, especially from the PDE community.
Considered already in the seminal paper \cite{lasry_lions}, various problems have been addressed, including existence results (e.g. \cite{araposthatis2017mfgs_ergodic_cost,cirant2016stationary,dragoni2018ergodic_mfgs}), forward-convergence results, i.e., the convergence of Nash equilibria in the $N$-player game to stationary MFG solutions (e.g. \cite{bardi2014sicon,feleqi2013derivation_ergodic_mfgs}), and the long-time convergence of finite-horizon MFGs to stationary MFGs (e.g. \cite{bardi2024SIMAN,cardaliaguet2012long_time_averages,cecchin2024turnpike}).
On the other hand, the field of mean-field games with singular stochastic controls presents a still limited but rapidly increasing number of contributions.
They have focused both on abstract results regarding the existence and uniqueness of equilibria (see \cite{cohen2024existence, denkert2024extended, dianetti2023unifying, fu2017mean, fu2023extended}), as well as on explicit characterizations of the Nash solution (see \cite{campi2022mean, cao2022mfgs, cao2023stationary, dianetti2023ergodic, guo2019stochastic}).
Particularly related to our paper are \cite{cao2023stationary} and \cite{cohen2024existence}. In \cite{cao2023stationary}, in the context of a mean-field game with singular controls, stationary discounted and ergodic Nash equilibria are explicitly constructed and related via the vanishing discount factor method. In the very recent \cite{cohen2024existence}, existence of optimal controls for stationary singular single-agent control problems as well as the existence of mean-field game equilibria for stationary singular MFGs are derived through a relaxed approach.
Contributions on stationary MFC problems are less numerous.
Here, we refer to \cite{albeverio2022sicon,bao_tang2023ergodic_ctrl_mkv,bayraktar2024discrete,christensen2021competition,fuhrman2025ergodicMKV}.
In \cite{bayraktar2024discrete}, an ergodic MFC control problem is addressed in a discrete framework.
In \cite{bao_tang2023ergodic_ctrl_mkv}, the authors consider the ergodic MFC problem, and prove the existence and uniqueness of the viscosity solution to a fully nonlinear PDE on the Hilbert space $L^2$ of square integrable random variables.
The long-time convergence of finite-horizon MFC problem value function to the value of the ergodic MFC problem is also addressed.
In \cite{fuhrman2025ergodicMKV} the authors provide interesting complementary to \cite{bao_tang2023ergodic_ctrl_mkv}, by studying the associated HJB equation on the Wasserstein space of square integrable probability measures directly, by relying on the Abelian limit of the discounted MKV control problems studied in \cite{rudà2025infinitetimehorizonoptimal}.
In \cite{albeverio2022sicon}, existence and uniqueness for a class of ergodic MFC problems is showed. In addition, they provide an $N$-agents Markovian control problem, for which the optimal control can be explicitly characterized, and show that its payoff, thus the value function, converges towards the value of the ergodic MFC problem.
Finally, \cite{christensen2021competition} contains an interesting study of MFGs of optimal harvesting, in which solutions to both the MFG and the MFC problem are computed and compared.

Thirdly, we contribute to the literature studying games with strategic complementarities (also known as submodular/supermodular games) and the potential emergence of multiple equilibria. This class of games has garnered significant attention in Economics. As Xavier Vives asserts in \cite{vives2018strategic}, ``Complementarities are intimately linked to multiple equilibria and have a deep connection with strategic situations, and the concept of strategic complementarity is at the center stage of game-theoretic analyses.'' Among the myriad contributions, we refer to the deterministic games considered in \cite{vives1990nash, vives2005complementarities, vives2005games}, as well as to the dynamic stochastic formulations presented in \cite{adlakha2013mean, alvarez2023price, alvarez2023strategic, carmona2017mean, dianetti2024multiple} and references therein.

Finally, we connect to recent works dealing with correlated and coarse correlated equilibria in mean-field games.
While well-known in game theory, correlated and coarse correlated equilibria have been considered in the mean-field games literature very recently.
In the works \cite{bonesini_thesis,bonesiniCE,campifischer2021} by Bonesini, Campi and Fischer, existence and convergence results are established for correlated equilibria in mean field games with discrete time and finite state and action spaces.
A second group of papers by Laurière et al. \cite{lauriere2024learningmeanfieldgames} and Muller et al. \cite{muller2022learningCE,muller2021learning} considers both CEs and CCEs in a similar setting.
\cite{lauriere2024learningmeanfieldgames,muller2022learningCE,muller2021learning} contain an extensive discussion of learning algorithms for approximating Nash equilibria, CEs and CCEs in the mean field limit.
Finally, \cite{campi2023LQ,campi_cannerozzi_fischer2023coarse,cannerozzi2025thesis} consider coarse correlated equilibria in continuous-time MFGs.
The latter papers are particularly relevant in our context, since we build on the definition and intuitions developed therein.

\vspace{0.25cm}

\noindent \textbf{Structure of the paper.}
The rest of the paper is organized as follows: Section \ref{sec:N_player_game} presents the $N$-player game, and Section \ref{sec:mfg} introduces the corresponding mean-field formulation.
Section \ref{sec:assumptions} details the assumptions and includes some auxiliary control-theoretic results.
In Section \ref{sec:cooperative}, the mean-field control problem is solved, and its connection with the central planner's optima are established.
Section \ref{sec:competitive} characterizes both coarse correlated equilibria and Nash equilibria in the mean-field game, and discusses their relationships with the equilibria in the $N$-player game.
Section \ref{sec:numerics} numerically illustrates the findings from previous sections. Finally, the \nameref{sec:appendix} contains technical proofs and lemmata.

\section{The \textit{N}-player Game}\label{sec:N_player_game}

Let $(\Omega, \F,\mathbb F = (\F_t)_{t \geq 0^{-}},\prob)$ be a filtered probability space satisfying the usual assumptions.
Let $(W^i)_{i \geq 1}$, $W$ be a sequence of independent $\mathbb{F}$-Brownian motions, and let $\xi$, $(\xi^i)_{i \geq 1}$ be a sequence of i.i.d. positive random variables with distribution $\mu_0 \in \mathcal{P}(\R_+)$.
We assume that they are independent from $W$ and $(W^i)_{i \geq 1}$, and they are $\F_{0^{-}}$-measurable.

\smallskip
We consider the following set of strategies, to be subject to further restrictions in the following:
\begin{align*}
    \A := & \{ \nu:\Omega \times \R_+ \to \R_+, \mbox{ $\mathbb{F}$-adapted and such that } t \mapsto \nu_t \mbox{ is a.s.} \\ 
	& \hspace{1cm} \mbox{nondecreasing, right-continuous, } \nu_{0^-} = 0 \mbox{ and } \E[\nu_T]<\infty \; \forall \, T > 0 \}. 
\end{align*}

Let $N \geq 2$.
We denote a vector of strategies $(\nu^1,\dots,\nu^N) \in \A^N$ by $\boldsymbol{\nu}^N$. We refer to $\boldsymbol{\nu}^{N} \in \A^N$ as a strategy profile.
We denote by $\boldsymbol{\nu}^{-i,N}=(\nu^{1},\dots,\nu^{i-1},\nu^{i+1},\dots,\nu^{N})$ the vector of strategies of players $j \neq i$, and we denote the vector of strategies $\boldsymbol{\nu}^N$ also by $(\nu^{i},\boldsymbol{\nu}^{-i,N})$.

\smallskip
Let $\delta$, $\sigma$ be in $\R_+ = (0,+\infty)$.
For any strategy profile $\boldsymbol{\nu} \in \A^N$, we consider the following dynamics:
\begin{equation}\label{N_player:dynamics}
    dX^{\nu^{i}}_{t}=-\delta X^{\nu^{i}}_{t}dt+\sigma X^{\nu^{i}}_{t} dW^i_t+d\nu^{i}_t, \quad  X_{0^-}^{\nu^{i}}=\xi^i,
\end{equation}
for any $i=1,\dots,N$.
Observe that, for any $\boldsymbol{\nu}^N \in \A^N$, there exists a unique strong solution to \eqref{N_player:dynamics} (see, e.g., \cite[Theorem 7, Chapter V]{protter2005stochastic}).
Actually, one has
\[
X^{\nu^i}_t = X^{i,0}_t\left( \xi^i + \int_0^t \frac{d\nu^i_s}{X^{i,0}_s} \right),
\]
where $X^0 = (X^{1,0},\dots,X^{N,0})$ denotes the uncontrolled solution of \eqref{N_player:dynamics}, i.e. the one associated to $\nu^i \equiv 0$, and so that $X^{i,0}_{0^{-}} \equiv 1$.
Moreover, for any $i=1,\dots,N$, we define the flow of empirical averages of players $j \neq i$ by
\begin{equation*}
	\theta^{N,\boldsymbol{\nu}^{-i,N}}_t := \frac{1}{N-1} \sum_{j \ne i} X_t^{\nu^{j}}, \quad t \geq 0^{-}.
\end{equation*}

\smallskip
Let $\alpha$, $\beta$ in $(0,1)$, $q > 0$.
Each player is associated with the following reward functional:
\begin{equation}\label{N_player:payoff_funtional}
    \J^N(\nu^i,\boldsymbol{\nu}^{-i,N}) = \liminf_{ T \uparrow \infty} \frac{1}{T}\E\left[ \int_0^T (X^{\nu^i}_t)^\alpha \big( \theta^{N,\boldsymbol{\nu}^{-i,N}}_t \big)^{\beta}dt - q \nu^i_T\right],
\end{equation}
which can possibly be infinite.
Occasionally, we use the notation $\pi(x,\theta) = x^\alpha \theta^\beta$, for any $(x,\theta) \in \R_+^2$, and we write $\pi_x(x,\theta) = \partial_x \pi(x,\theta)$ and analogously $\pi_\theta(x,\theta) = \partial_\theta \pi(x,\theta)$.

\smallskip
When dealing with $N$-player games, we consider open-loop strategies. Roughly speaking, we allow each player to observe the noises of all players, as well as their initial position.
To this extent, denote by $\mathbb{F}^{N} = (\F^N_t)_{t \geq 0^{-}}$ be the $\prob$-augmentation of the filtration generated by the Brownian motions $(W^i)_{i=1}^N$ and initial data $(\xi^i)_{i=1}^N$.
\begin{definition}[Open-loop strategies for the $N$-player game]\label{N_player:admissible_strategies:def}
We say that a process $\nu \in \A$ is an open-loop strategy for the $N$-player game if $\nu$ is $\mathbb{F}^{N}$-progressively measurable. We denote the set of open-loop strategies for the $N$-player game by $\A_N$.
We denote by $\A_N^N$ the set of strategy profiles $\boldsymbol{\nu}^N \in \A^N$ so that $\nu^i \in \A_N$ for every $i=1,\dots,N$. We refer to $\boldsymbol{\nu}^N \in \A_N^N$ as an open-loop strategy profile.
\end{definition}

\smallskip
We are interested in different kinds of equilibria in the $N$-player system.
We deal both with the cooperative case and the competitive framework.

\subsection*{Cooperative framework}
In the cooperative case, we consider the central planner's optimization problem.
In order to introduce such optimization problem, we associate to the dynamics \eqref{N_player:dynamics} and payoff functionals \eqref{N_player:payoff_funtional} an $N$-dimensional control problem.
We consider the following functional
\begin{equation}\label{central_planner:payoff}
    \Bar{\J}^{N} (\boldsymbol{\nu}^N) := \frac{1}{N}\sum_{i=1}^N \J^{N}(\nu^{i},\boldsymbol{\nu}^{-i,N}), \quad \boldsymbol{\nu}^N \in \mathcal{A}_N^N,
\end{equation}
which can be regarded as a welfare utility for the $N$-player system.

\begin{definition}\label{def:central_planner_optima}
Let $\eps \geq 0$, $\mathcal{C} \subseteq \A^N_N$.
A strategy profile $\boldsymbol{\hat{\nu}}^N = (\hat{\nu}^1,\dots,\hat{\nu}^N) \in \mathcal{C}$ is $\eps$-optimal for the central planner optimization problem within the set of strategy profiles $\mathcal{C}$ if 
\[ 
    \Bar{\J}^N(\boldsymbol{\hat{\nu}}^N) \geq \Bar{\J}^N(\boldsymbol{\nu}^N)-\varepsilon,\quad \forall \, \boldsymbol{\nu}^N \in \mathcal{C}.
\]
If $\eps = 0$, we say that the strategy profile $\boldsymbol{\hat{\nu}}^N$ is optimal for the central planner within the set of strategy profiles $\mathcal{C}$.
\end{definition}
When dealing with the central planner's optimization problem, the players are referred to as agents, since there is no competition between them: the central planner picks herself a strategy for each player in order to maximize the welfare utility functional $\Bar{\J}^N$.
As a consequence, agents are not allowed to unilaterally deviate from the central planner's strategy profile.
It can be easily shown that if a strategy profile is an optimum of the central planner maximization problem, it is Pareto efficient as well (see, e.g., \cite[Chapter 5, Definition 5.1]{carmona2016lectures} for a definition of Pareto optimality).

\subsection*{Competitive framework}
We consider the notion of coarse correlated equilibria in the $N$-player game.
Our definition of CCEs explictly models moderator's lottery, by introducing a correlation device (see Definition \ref{def:correlation_device}).
This random variable represents the extra-randomness used by the moderator to randomize players' strategies independently of the idiosyncratic shocks and initial data that determine players' states' dynamics.
We then carefully take into account the different information structures available to the committing players and the deviating players, as the former will depend on the correlation device, while the latter will not.

\smallskip
We assume the following structural condition on the $\sigma$-algebra $\F_{0^{-}}$:
\begin{customassumption}{\textbf{U}}\label{assumption_F_0}
The $\sigma$-algebra $\F_{0^{-}}$ is large enough to support a $\F_{0^{-}}$-measurable uniform random variable independent of the initial data $\xi$, $(\xi^i)_{i \geq 1}$ and the noises $W$, $(W^i)_{i \geq 1}$.
\end{customassumption}

Next, we introduce correlation between players' strategies.
\begin{definition}[Correlating device]\label{def:correlation_device}
A correlation device is any random variable $Z: (\Omega,\F,\prob) \to (\R,\mathcal{B}_{\R})$ so that $Z$ is $\F_{0^{-}}$-measurable and independent of $\xi$, $W$, $(\xi^i)_{i\geq 1}$ and $(W^i)_{i \geq 1}$.
\end{definition}
Our definition of correlation device $Z$ must be interpreted in the following way: A a lottery over open loop strategies is run by the mediator, by means of the correlation device $Z$.
The extraction of the strategy happens before the game starts and it is independent of the initial data and idiosyncratic shocks that determine the random evolution of players' states.
These features are captured by the fact that the random variable $Z$ is $\F_{0^{-}}$-measurable and it is independent of $(\xi^i)_{i \geq 1}$ and $(W^i)_{i \geq 1}$. 
We observe that the correlation device is not exogenous in the sense of a \textit{common noise} (see, e.g., \cite{delarue16commonnoise}), as it is picked by the moderator.
Finally, Assumption \ref{assumption_F_0} guarantees the existence of the correlation device $Z$.
Next, we define the result of moderator's recommendation to the $N$ players:
\begin{definition}[Correlated strategy profile]\label{N_player:correlated_strategy_profile}
We define a correlated strategy profile as a pair $(Z,\boldsymbol{\lambda})$ so that
\begin{enumerate}[label=(\roman*)]
    \item $Z$ is a correlation device;
    \item $\boldsymbol{\lambda} = (\lambda^{i})_{i=1}^N$ is an admissible recommendation to the $N$ players; that is, for each $i=1,\dots,N$, $\lambda^{i}=(\lambda^{i}_t)_{t \geq 0^{-}}$ belongs to $\A$ and it is progressively measurable with respect to the $\prob$-augmentation of the filtration $(\sigma(Z)\vee\F^N_t)_{t \geq 0^{-}}$.
\end{enumerate}
\end{definition}
We notice that the definition of correlated strategy profile (Definition \ref{N_player:correlated_strategy_profile}) is made of two parts: a correlation device, which is chosen by the moderator as part of her recommendation to the players, and a strategy profile. The latter is by definition dependent of both the correlation device $Z$ and the first $N$ Brownian motions and initial data $(\xi^i,W^i)_{i=1}^N$.
In particular, it carries at least some of the information the the moderator uses to randomize players' strategies.

\smallskip
We now assign dynamics and rewards of each player. To do so, we must distinguish two cases: Suppose that each player $i$ follows the recommendation $\lambda^{i}$. Then, players' state dynamics are given by \eqref{N_player:dynamics}, and each player gets the reward $\J^N(\lambda^{i},\boldsymbol{\lambda}^{-i})$, with $\J^N$ given by \eqref{N_player:payoff_funtional}.
Suppose player $i$ deviates, while other players stick to the correlated strategy profile $\boldsymbol{\lambda}^{-i}$.
The deviating player will pick instead an open loop strategy $\nu\in\A_N$.
Her dynamics are given by \eqref{N_player:dynamics} and her reward is given by $\J^N(\nu^i,\boldsymbol{\lambda}^{-i})$, with $\J^N$ given by \eqref{N_player:payoff_funtional}.

We interpret deviations in the following way: Each player must decide whether to commit to moderator's lottery before the extraction happens, only by relying on the information given by the law of the correlated strategy $(Z,\boldsymbol{\lambda})$, which is assumed to be common knowledge between the players.
If a player does not commit, she will pick a strategy without any information on the outcome of the extraction.
Notice that, since the deviating player has only knowledge of the law of the correlated strategy profile  $(Z,\boldsymbol{\lambda})$, she will use a strategy $\nu \in \A_N$, which is, in particular, independent of the correlation device $Z$; consequently, her state process is independent of $Z$.
This results in information asymmetry between the committing player and the deviating player, and models the fact that if a player deviates, she will do so without any knowledge about the outcome of moderator's lottery: she will not exploit any of the additional information the mediator would give away when communicating the recommended strategies to the players.
Finally, even if a deviation $\beta$ is independent of the correlation device $Z$, it is correlated to the strategy profile $\lambda$, as they both are dependent of the noises and initial data $(\xi^j,W^j)_{j=1}^N$.

We are now ready to define coarse correlated equilibria of the ergodic $N$-player game:
\begin{definition}[$\varepsilon$-coarse correlated equilibrium within the set of strategies $\mathcal{B}$]\label{def:epsilon_CCE}
Let $\varepsilon \geq 0$, $\mathcal{B} \subseteq \A_N$.
A correlated strategy profile $(Z,\boldsymbol{\lambda})$ is an $\varepsilon$-coarse correlated equilibrium ($\varepsilon$-CCE) of the ergodic $N$-player game within the set of strategies $\mathcal{B}$, if for any $i=1,\dots, N$, we have
\[
    \J^N(\lambda^{i},\boldsymbol{\lambda}^{-i}) \geq \J^N(\nu,\boldsymbol{\lambda}^{-i})-\varepsilon,\quad \forall \, \nu \in \mathcal{B}.
\]
If $\varepsilon = 0$, we say that the correlated strategy profile $(Z,\boldsymbol{\lambda})$ is a coarse correlated equilibrium (CCE) of the ergodic $N$-player game within the set of strategies $\mathcal{B}$.
\end{definition}

\medskip
We observe that the definition of $\eps$-coarse correlated equilibria for the $N$-player game extends the one of Nash equilibria, that we recall:
\begin{definition}[$\varepsilon$-Nash equilibrium within the set of strategies $\mathcal{B}$]\label{def:epsilon_Nash}
Let $\varepsilon \geq 0$, $\mathcal{B} \subseteq \A_N$.
A strategy profile $\boldsymbol{\nu}^{*}=(\nu^{1,{*}},\dots,\nu^{N,{*}}) \in {\mathcal{B}}^N$ is an $\varepsilon$-Nash equilibrium ($\varepsilon$-NE) of the ergodic $N$-player game within the set of strategies $\mathcal{B}$, if for any $i=1,\dots, N$, we have
\[ 
    \J^N(\nu^{i,*},\boldsymbol{\nu}^{-i,*}) \geq \J^N(\nu,\boldsymbol{\nu}^{-i,*})-\varepsilon,\quad \forall \, \nu \in \mathcal{B}.
\]
If $\varepsilon = 0$, we say that the strategy profile $\boldsymbol{\nu}^{*}$ is a Nash equilibrium (NE) of the ergodic $N$-player game within the set of strategies $\mathcal{B}$.
\end{definition}

Every $\eps$-CCE $(Z,\boldsymbol{\lambda})$ with deterministic correlation device $Z$ is an $\eps$-NE: It is enough to notice that, since $Z$ is deterministic, the correlated strategy profile $\boldsymbol{\lambda}$ reduces to an open-loop strategy profile $(\nu^1,\dots,\nu^N)$ in $\A_N^N$.
Conversely, any $\eps$-NE induces an $\eps$-CCE with deterministic correlation device.

\section{The Ergodic Mean-field Game}\label{sec:mfg}
In order to determine $\eps$-optimal solutions to the central planner problem and $\varepsilon$-equilibria in the competitive setting, we consider the mean-field counterparts of the optimization problem and game considered before.
We will then show in Sections \ref{sec:cooperative} and \ref{sec:competitive} the relation between mean-field solutions to the $N$-player cooperative and competitive problems respectively.

\smallskip
We work on the same probability space $(\Omega,\F,\mathbb{F},\prob)$ defined in the previous section.
Given a strategy $\nu \in \A$, we consider the following dynamics:
\begin{equation}\label{mf:dynamics}
    dX^{\nu}_t = -\delta X^{\nu}_tdt + \sigma X^{\nu}_t dW_t +d\nu_t, \quad X_{0^{-}}=\xi.
\end{equation}
Let $\theta$ be a $\F_{0^{-}}$-measurable non-negative random variable $\theta$.
We consider the following payoff functional, to be maximized:
\begin{equation}\label{mfg:payoff}
    \J(\nu,\theta) = \liminf_{T \uparrow \infty}\frac{1}{T}\E\left[ \int_0^T (X^{\nu}_t)^\alpha \theta^{\beta}dt -q \nu_T \right].
\end{equation}
In the following, $\theta$ will either be a positive scalar, i.e. a degenerate random variable, when dealing with the MFC problem and with mean-field NEs (see Definitions \ref{mfc:def:optimal_ctrl} and \ref{mf:def:NE}), or a proper random variable, when dealing with mean-field CCEs (see Definitions \ref{def:correlated_strategy} and \ref{def:CCE}), given by the average of the (possibly random) stationary measure of processes obeying \eqref{mf:dynamics}.

Let  $\mathbb{F}^{\xi,W} = (\F^{\xi,W}_t)_{t \geq 0^{-}}$ be the $\prob$-augmentation of the filtration generated by $\xi$ and $W$.
Analogously to Definition \ref{N_player:admissible_strategies:def}, we consider the following strategies:
\begin{definition}[Open-loop strategy for the ergodic MFG]\label{cce:admissible_strategies}
A process $\nu \in \A$ is an open-loop strategy if it is progressively measurable with respect to the filtration $\mathbb{F}^{\xi,W}$.
We denote the set of open-loop strategies by $\A_{mf}$.
\end{definition}

\subsection*{Cooperative framework}
We address the mean-field counterpart of the central planner's maximization problem, which is given by the mean-field control (MFC) maximization problem.
Roughly speaking, this problem consists in maximizing the reward \eqref{mfg:payoff} under the additional constraint $\theta=\lim_{t \to \infty}\E[X^{\nu}_t] =: \E[X^\nu_\infty]$, for every control choice $\nu$.

\smallskip
In order to properly define the reward, we need to restrict the class of admissible controls.
\begin{definition}
We say that a strategy $\nu$ is admissible for the mean-field control problem if $\nu \in \A_{mf}$ and the process $(X^{\nu}_t)_{t \geq 0^{-}}$ admits a unique stationary distribution $p^{\nu}_\infty \in \mathcal{P}(\R_+)$ with finite first moment, i.e. $\int_{\R_+} xp^\nu_\infty(dx) < \infty$.
We denote the set of admissible strategies for the stationary MFC problem by $\A_{MFC}$.
\end{definition}

We notice that the set $\A_{MFC}$ is not empty: As it will be evident by subsequent Lemma \ref{geometric_bm:reflected:lemma}, for any $a > 0$, the policy $\nu^a$ that reflects the process $X^{\nu^a}$ upwards à la Skorohod at the level $a$ is admissible for the MFC problem.

\smallskip
For any $\nu \in \A_{MFC}$, denote by $\E[X^\nu_\infty]$ the first moment of the corresponding limit measure $p^\nu_\infty$.
The payoff functional associated to a strategy $\nu \in \A_{MFC}$ is given by
\[
\J(\nu,\E[X^{\nu}_\infty]) = \liminf_{T \uparrow \infty}\frac{1}{T}\E\left[ \int_0^T (X^{\nu}_t)^\alpha (\E[X^\nu_\infty])^{\beta}dt - q \nu_T \right].
\]

\begin{definition}\label{mfc:def:optimal_ctrl}
An admissible control $\hat{\nu} \in \A_{MFC}$ is an optimal control for the mean-field control problem if 
\[
\J(\hat{\nu},\E[X^{\hat{\nu}}_\infty]) \geq \J(\nu,\E[X^{\nu}_\infty]), \quad \forall \, \nu \in \A_{MFC}.
\]
\end{definition}

The study of the central planner's optimization problem and its relation with the $N$-agent system is the content of Section \ref{sec:cooperative}: We show in Theorem \ref{mfc:thm:optimal_control} that it is possible to completely characterize solutions of the MFC problem.
Then, in Theorem \ref{central_planner:thm:approximation}, we use the solution of the MFC problem to build a sequence of approximate central planner's optima in the underlying $N$-agent system, with vanishing approximating error.

\subsection*{Competitive framework}
We consider the mean field analogues of CCEs and NEs in the $N$-player game.
We first define the mean-field analogue of correlated stationary profiles:
\begin{definition}[Correlated stationary strategy]\label{def:correlated_strategy}
We define a correlated stationary strategy as a triple $(Z,\theta_\infty,\lambda)$ so that the following holds:
\begin{enumerate}[label=(\roman*)]
    \item $Z$ is a correlation device;
    \item $\theta_\infty$ is a $\sigma(Z)$-measurable non-negative random variable;
    \item $\lambda=(\lambda_t)_{t \geq 0^{-}}$ belongs to $\A$ and it is progressively measurable with respect to the $\prob$-augmentation of the filtration $(\sigma(Z)\vee\F_t^{\xi,W})_{t \geq 0^{-}}$.
\end{enumerate}
\end{definition}
In the following, we will denote the law of $\theta_\infty$ by $\rho \in \mathcal P (\R_+)$.
We interpret a correlated stationary strategy $(Z,\lambda,\theta_\infty)$ as follows: Before the game starts, the moderator picks a correlation device and uses it to generate a recommendation the representative player $\lambda$ and a random stationary mean $\theta_\infty$.
The stationary mean is expected to become stochastic, as a result of moderator's lottery for generating representative player's recommendation, and only as a result of it.
For this reason, we request $\theta_\infty$ to be measurable with respect to $\sigma(Z)$: in this sense, the correlation device picks both the representative player's strategy and the stationary mean $\theta_\infty$, as the latter captures the behavior of the rest of the population.
Finally, we notice that, even if we are considering the long-time averages, moderator's lottery has the important impact of keeping the stationary mean $\theta_\infty$ stochastic.

\smallskip
We now assign dynamics and payoff functional.
We distinguish the following two cases:
If the representative player decides to trust the mediator and so to follow her recommendation $\lambda$, then the dynamics is given by \eqref{mf:dynamics} with $\lambda$ instead of $\nu$ and the payoff is given by $\J(\lambda,\theta_\infty)$, with $\J$ defined by \eqref{mfg:payoff}.
If instead the representative player chooses to deviate, she uses a strategy $\nu \in \A_{mf}$, her dynamics is given by \eqref{mf:dynamics}, and she gets the reward $\J(\nu,\theta_\infty)$.
Observe that, when the representative player deviates, her strategy $\nu$ is $\mathbb F^{\xi,W}$-progressively measurable and therefore independent of $\theta_\infty$, since she has no information on the outcome of moderator's lottery.
As in the $N$-player game, the deviating player can only use her knowledge of the law of the correlated stationary strategy $(Z,\lambda,\theta_\infty)$, which is assumed to be publicly known.
Nevertheless, $\theta_\infty$ still appears in her payoff.

\begin{definition}[Coarse correlated Equilibrium for the ergodic MFG]\label{def:CCE}
A correlated stationary triple $(Z,\lambda,\theta_\infty)$ is a coarse correlated equilibrium (CCE) for the ergodic MFG if the following holds:
\begin{enumerate}[label=(\arabic*)]
    \item\label{cce:def:optimality} $\J(\lambda,\theta_\infty) \geq \J(\nu,\theta_\infty)$ for any $\nu \in \A_{mf}$,
    \item \label{cce:def:consistency} The process $X^{\lambda}$ admits a stationary distribution and it holds
    \begin{equation}\label{cce:cons_eq}
        \theta_\infty = \int_{\R_+} x p_\infty(dx,\theta_\infty),
    \end{equation}
    where $p_\infty$ is the stochastic kernel so that $\mu_\infty(dx,d\theta)=p_\infty(dx,\theta)\rho(d\theta)$ with $\rho = \prob\circ \theta_\infty^{-1}$ and $\mu_\infty = \lim_{t \to \infty}\prob\circ(X^{\lambda}_t,\theta_\infty)^{-1}$ in the weak sense.
\end{enumerate}
\end{definition}
We will refer to CCEs for the ergodic MFG as mean-field CCEs as well.

\begin{remark}
Property \ref{cce:def:consistency} in Definition \ref{def:CCE} is equivalent to 
\begin{equation}\label{cce_cons:conditional_form}
\theta_\infty \sim w-\lim_{t \to \infty }\E[X^{\lambda}_t \vert \theta_\infty],
\end{equation}
where $w$ denotes the weak convergence of random variables.
\end{remark}

The consistency condition \ref{cce:def:consistency} in Definition \ref{def:CCE} should be read in the following way: the mediator imagines what the stationary mean $\theta_\infty$ will be, before the game starts, and gives a recommendation to each player according to her idea.
If all players commit to the mediator's lottery for generating recommendations, then the long-time average should be consistent with what imagined by the mediator.

\medskip
The notion of CCE for the ergodic MFG extends the notion of notion of Nash equilibrium for the ergodic MFG, that we borrow from \cite{cao2023stationary}:
\begin{definition}[Nash equilibrium of the ergodic MFG]\label{mf:def:NE} 
A pair $(\nu^{*}, \theta^*) \in \A_{mf} \times \mathbb{R}_+$ is said to be a Nash equilibrium of the ergodic MFG if 
\begin{enumerate}
    \item\label{mfg:def:optimality} $\J(\nu^{*}, \theta^* ) \geq \J(\nu, \theta^* ) $, for any $\nu \in \A_{mf}$;
    \item\label{mfg:def:consistency} The optimally controlled process $X^{\nu^{*}}$ admits a limiting distribution $p^*_{\infty} \in \mathcal{P}(\R_+)$ satisfying
    \[
    \theta^* = \int_{\mathbb{R}_+} x p^*_\infty(dx).
    \]
    \end{enumerate}
\end{definition}
We will refer to NEs for the ergodic MFG as mean-field NEs as well.
We stress that, differently from mean-field CCEs, when looking for Nash equilibria, $\theta^*$ is assumed to be deterministic.

\smallskip
As in the $N$-player game, and actually by exactly the same reasoning, every CCE for the ergodic MFG $(Z,\theta_\infty,\lambda)$ with deterministic correlation device $Z$ is an Nash equilibrium for the ergodic MFG as well.
Conversely, any Nash equilibrium for the ergodic MFG $(\nu^*,\theta^*)$ induces a mean-field CCE with deterministic correlation device.

\medskip
The study of the existence of CCEs in the ergodic MFG and its relation with the $N$-player game is the content of Section \ref{sec:competitive}: We identify specific classes of correlated stationary strategies for which it is possible to state a sufficient condition to be a CCE.
Then we use CCEs in those classes to build a sequence of approximate CCEs the underlying $N$-player system, with vanishing approximating error.
As a byproduct, we get the same approximation result for NEs as well.

\section{Assumptions and Preliminary Results}\label{sec:assumptions}

On top of Assumption \ref{assumption_F_0}, we assume the following structural condition on the coefficients of the SDE:
\begin{customassumption}{\textbf{D}}\label{assumption:dissipativity}
The parameters $\delta$ and $\sigma$ satisfy the following condition:
\[
    2\delta - \sigma^2 > 0.
\]
Moreover, $\mu_0$ admits a finite second moment. 
\end{customassumption}

\begin{remark}
Let $X^0$ be the solution of \eqref{mf:dynamics} when the policy $\nu$ is identically equal to $0$.
Assumption \ref{assumption:dissipativity} is a dissipativity assumption on the square of $X^0$: indeed, by It\^{o}'s formula, we have
\[
(X^0_t)^2 = \xi^2 - (2\delta - \sigma^2)\int_0^t (X^0_s)^2ds + 2\sigma \int_0^t(X^0_s)^2 dW_s,
\]
which is then square-integrable and dissipative. Notice that the same assumption is assumed in \cite[Section 6]{cao2023stationary}.
\end{remark}

Let $(X^{\nu^a},\nu^a)$ be a solution to the Skorohod reflection problem at the barrier $a$, i.e. a pair of processes so that $\nu^a \in \A$, equation \eqref{mf:dynamics} holds true, $X^{\nu^a}_t \geq a, \, \forall \, t \geq 0$ $\prob$-a.s. and $\int_0^\infty \1_{ \{ X^{\nu^a}_s > a \} }d\nu^a_s = 0$, $\prob$-a.s. (see \cite[Section 1.2]{pilipenko2014introduction}, and in particular Definition 1.2.1 therein).
We refer to $X^{\nu^a}$ also as the diffusion reflected upward at the level $a$.
We state some important properties of $(X^{\nu^a},\nu^a)$, which will be used through the whole manuscript.

\begin{lemma}\label{geometric_bm:reflected:lemma}
\begin{enumerate}[wide,label=\roman*)]
\item \label{geometric_bm:reflected:lemma:finite_moments} For any $a > 0$, let $p_a \in \mathcal{P}(\R_+)$ be given by
\begin{equation}\label{gemetric_bm:reflected:density}
   p_a(dx) = \frac{2\delta + \sigma^2}{\sigma^2}  a^{\frac{2\delta}{\sigma^2} + 1}x^{-\frac{2\delta}{\sigma^2} -2}\1_{\{ x \geq a\} } dx.
\end{equation}
For any $0 \leq k \leq 2$, the measure $p_a$ admits finite $k$-moment.
In particular, it holds
\begin{equation}\label{geometric_bm:reflected:barrier_mean}
    \int_{\R_+} x p_a(dx) = \frac{2\delta + \sigma^2}{2\delta}a.
\end{equation}
Moreover, the map $\R_+ \times \mathcal{B}_{\R_+} \ni (a,B) \mapsto p_a(B)$ defines a stochastic kernel from $\R_+$ to $\mathcal{B}_{\R_+}$.

\item \label{geometric_bm:reflected:lemma:reflection} For any $a > 0$, there exists a unique strong solution $(X^{\nu^a}_t,\nu^a_t)_{t \geq 0^{-}}$ of the Skorohod reflection problem at the barrier $a$.
The process $\nu^a$ is given by
\[
\nu^a_t = \int_0^t X^0_s \: d \left( \sup_{0 \leq u \leq s}\left( \frac{a-X^0_u}{X^0_u}\right)^+\right), \quad \nu^a_{0^{-}} = 0.
\]
Moreover, the process $X^{\nu^a}$ is positively recurrent with stationary measure given by $p_a$.

\item \label{geometric_bm:reflected:lemma:integrability} There exists a positive constant $c$ so that
\begin{equation*}
    \sup_{t \geq 0} \E[(X^{\nu^a}_t)^2]  \leq c (1 + a^2), \quad   \sup_{T > 0} \E\left[\left( \frac{1}{T}\nu^a_T \right)^2\right]  \leq c(1 + a^2).
\end{equation*}
\end{enumerate}
\end{lemma}
The proof is postponed to the \nameref{sec:appendix}.
For later use, we introduce the real function $C(a,p)$, for $(a,p) \in \R^2_+$, given by 
\begin{equation}\label{alvarez:ergodic_functional}
\begin{aligned}
    C(a,p) = & (2\delta + \sigma^2)\left( \frac{p}{2\delta + \sigma^2(1-\alpha) }a^{\alpha} - \frac{q}{2}a  \right).
\end{aligned}
\end{equation}
Let $\nu^a$ is the policy that reflects the process $X^{\nu^a}$ solution to \eqref{mf:dynamics} upwards à la Skorohod at the level $a > 0$.
By \cite[Lemma 2.1]{alvarez2018stationary},
it holds
\[
    C(a,p) = \liminf_{T \uparrow \infty}\frac{1}{T}\E\left[ \int_0^T p\cdot(X^{\nu^a}_t)^\alpha dt - q \nu^a_T \right] = \lim_{T \uparrow \infty}\frac{1}{T}\E\left[ \int_0^T p\cdot(X^{\nu^a}_t)^\alpha dt - q \nu^a_T \right].
\]

In the following, we will need to solve several ergodic optimization problems of singular controls.
The existence of a unique solution to such problems is ensured by following Lemma:
\begin{lemma}\label{lemma:ergodic_control}
Let $\lambda \in \R$ so that $q \delta - \lambda > 0$ and $p > 0$.
Define the following function
\begin{equation}\label{control:auxiliary:running_cost}
    g(x,p,\lambda) := x^\alpha p + \lambda x,
\end{equation}
and consider the reward functional
\begin{equation}\label{lemma:ergodic_control:payoff}
    \Tilde{\J}(\nu,p,\lambda) := \liminf_{T \uparrow \infty} \frac{1}{T}\E\left[ \int_0^T g(X^{\nu}_t,p,\lambda) dt - q \nu_T \right],
\end{equation}
where $X^{\nu}=(X^{\nu}_t)_{t \geq 0^{-}}$ evolves accordingly to \eqref{mf:dynamics}.
Then, there exists an optimal control $\nu^* \in \A$, given by the policy which reflects the state process upwards à la Skorohod at the barrier $ a^*(p,\lambda)$, where
\begin{equation}\label{lemma:ergodic_control:barrier}
    a^*(p,\lambda) = \left( \frac{2\alpha\delta }{2\delta + \sigma^2(1-\alpha)} \frac{p}{q\delta - \lambda} \right)^{\frac{1}{1-\alpha}}.
\end{equation}
\end{lemma}
The proof is postponed to the \nameref{sec:appendix}.
The extra term $\lambda x$ in the definition of $g(x,p,\lambda)$ comes from  the Lagrange multiplier approach that we use to solve the MFC problem (see Theorem \ref{mfc:thm:optimal_control}), where $\lambda$ is the Lagrange multiplier itself.
$p$ will either be $\E[\theta_\infty^\beta]$ (see Proposition \ref{dev_player:prop:optimal_deviation}) or a constant $\theta^\beta$ (see again Theorem \ref{mfc:thm:optimal_control} and Proposition \ref{mfg:thm:nash_eq}).

\medskip
With respect to a smaller class of strategies, we state a first order optimality condition for a control $\hat{\nu}$, inspired by \cite{federico2021infinite} (see also \cite{bank_riedel2001AAP}).
Although only the necessary part will be needed, for the sake of completeness, we also show that it is sufficient under additional assumptions.

\begin{lemma}[First order optimality condition]\label{lemma:first_order_condition}
Let $p > 0$, $\lambda > q \delta$.
Let $1 \leq q, q' \leq \infty$ be Young conjugates, and define the set $\A_{2q} $ as the set of controls $\nu \in \A$ so that
\begin{equation}\label{lemma:first_order_condition:admissible_controls}
\sup_{T > 0}\frac{1}{T}\E\left[ \int_0^T \vert X^{\nu}_t \vert^{2q} dt  \right] < \infty.
\end{equation}
Let $\hat{\nu} \in \A_{2q}$ so that 
\begin{equation}\label{lemma:first_order_condition:integrability_condition}
\sup_{T > 0}\frac{1}{T}\E\left[\int_0^T \vert (X^{\hat{\nu}}_t)^{\alpha - 2} \vert^{q'} \right]<\infty,
\end{equation}
if $q'<\infty$, and so that
\begin{equation}\label{lemma:first_order_condition:integrability_condition_esssup}
\sup_{T > 0} \inf \left\{C \geq 0: \: (X^{\hat{\nu}}_t)^{\alpha - 2} \leq C \; \text{for $dt \otimes d\prob$-a.e. $(t,\omega) \in [0,T] \times \Omega$} \right\}  < \infty,
\end{equation}
if $q' = \infty$.
\begin{enumerate}[label=(\alph*)]
    \item\label{lemma:first_order_condition:necessary} Suppose that $\hat{\nu}$ is optimal within the set $\A_{2q}$ for the control problem with dynamics \eqref{mf:dynamics} and reward \eqref{lemma:ergodic_control:payoff}. Then, for every $\nu \in \A_{2q}$, it holds
    \begin{equation}\label{lemma:first_order_condition:ineq_limsup}
    \limsup_{T \uparrow \infty}\frac{1}{T} \E \left[ \int_0^T  g_x(X^{\hat{\nu}}_t,p,\lambda)( X^{\hat{\nu}}_t - X^{\nu}_t ) dt - q(\hat{\nu}_T - \nu_T) \right] \geq 0.
    \end{equation}
    \item\label{lemma:first_order_condition:sufficient} Suppose that either $\hat{\nu}$ satisfies \eqref{lemma:first_order_condition:ineq_limsup} and 
    \begin{equation}\label{lemma:first_order_condition:limit}
     \liminf_{T \uparrow \infty} \frac{1}{T}\E\left[ \int_0^T g(X^{\hat{\nu}}_t,p,\lambda) dt - q \hat{\nu}_T \right] = \lim_{T \uparrow \infty} \frac{1}{T}\E\left[ \int_0^T g(X^{\hat{\nu}}_t,p,\lambda) dt - q \hat{\nu}_T \right],
    \end{equation}
    or that $\hat{\nu}$ satisfies
    \begin{equation}\label{lemma:first_order_condition:ineq_liminf}
    \liminf_{T \uparrow \infty}\frac{1}{T} \E \left[ \int_0^T  g_x(X^{\hat{\nu}}_t,p,\lambda) ( X^{\hat{\nu}}_t - X^{\nu}_t ) dt - q(\hat{\nu}_T - \nu_T) \right] \geq 0.
    \end{equation}
    Then $\hat{\nu}$ is optimal within the set $\A_{2q}$.
\end{enumerate}
\end{lemma}
The proof is postponed to the \nameref{sec:appendix}.

\begin{remark}\label{rmk:first_order}
Let $a>0$ and let $\nu^a$ be the policy that reflects the process $X^{\hat{\nu}}$ upwards à la Skorohod at the barrier $a$.
Then, the control $\nu^a$ is so that $\prob( (X^{\nu^a}_t)^{\alpha - 2} \leq c \: \forall t \geq 0) = 1$, for a constant $c > 0$, since the reflected process is so that $X^{\nu^a}_t \geq a$ for any $t$, and $X^{\hat{\nu}}$ belongs to $\A_{2}$ by Lemma \ref{geometric_bm:reflected:lemma}.
Therefore, \eqref{lemma:first_order_condition:integrability_condition_esssup} is satisfied, and we take $q=1$ in \eqref{lemma:first_order_condition:admissible_controls}.
Assumption \ref{assumption:dissipativity} and Lemma \ref{geometric_bm:reflected:lemma} imply that $X^{\nu^a}$ satisfies  \eqref{lemma:first_order_condition:admissible_controls}. 
By Lemma \ref{geometric_bm:reflected:lemma}, \eqref{lemma:first_order_condition:limit} is satisfied as well.
\end{remark}

\begin{remark}
We can restate Lemma \ref{lemma:first_order_condition} in terms of linear conditions involving optional projections.
Consider the probability measure $\Tilde{\prob}$ equivalent to $\prob$ defined via the Radon-Nykodim derivative
\[
\frac{d\Tilde{\prob}}{d\prob} \Big\vert_{\F_t} = e^{\sigma W_t - \frac{\sigma^2}{2}t}, \quad t \geq 0.
\]
It can be shown that, for every $\nu \in \A$, the state process $X^{\nu}$ can be represented as
\[
X^{\nu}_t = e^{-\delta t} M_t (\xi + \bar \nu_t ),
\]
where $\bar{\nu} = (\bar{\nu}_t)_{t \geq 0}$ is defined via the identity $\nu_t =\int_0^t e^{-\delta s}M_s d\bar\nu_s$.
Denote by $\Tilde{\E}$ the expectation with respect to $\Tilde{\prob}$.
By taking advantage of Fubini's theorem and optional projections as in \cite[Theorem 57, Chapter VI, p. 122]{dellacherie1982probabilities_potential_B}, we can restate necessary condition \eqref{lemma:first_order_condition:ineq_limsup} as
\begin{equation*}
\limsup_{T \uparrow \infty}\frac{1}{T} \Tilde{\E} \left[ \int_0^T  \left( \Tilde{\E}\left[ \int_s^T e^{-\delta t} g_x(X^{\hat{\nu}}_t,p,\lambda)dt  \Big \vert \F_s \right] -q e^{-\delta s} \right) d(\hat{\bar{\nu}}_s - \bar{\nu}_s) \right] \geq 0,  \quad \forall \nu \in \A_{2q},
\end{equation*}
and analogously holds for \eqref{lemma:first_order_condition:ineq_liminf} and \eqref{lemma:first_order_condition:limit}.
While first order conditions for optimality are well known for both finite and infinite time horizon singular control problems (see, e.g., \cite{bank2005sicon_fuel_constaint,bank_riedel2001AAP,federico2021infinite,ferrari_salminen2016irreversible} among others), to the best of our knowledge, this is the first time in the literature that first order conditions for singular control problems with ergodic reward functionals have been derived.
\end{remark}

\section{Cooperative Case: Mean-field Solution and Approximation}\label{sec:cooperative}

In this section, we tackle the mean-field control solution of the central planner's optimization problem defined in Definition \ref{mfc:def:optimal_ctrl}.

\subsection{Mean-field solution}
The first result, Theorem \ref{mfc:thm:optimal_control}, regards existence and uniqueness of optimal solutions of the MFC problem.
In order to compute the optimal control, we use a Lagrangian multiplier type approach, to take care of the constraint on the stationary first order moment.
We first restrict to strategies so that the corresponding stationary mean is equal to some prescribed level $ \theta \geq 0$, we compute the optimal strategy within this smaller constrained set and finally we optimize over all possible values of the stationary mean.
A somehow similar approach has been used in \cite{christensen2021competition}, for a MFC problem of impulse control.

\begin{thm}\label{mfc:thm:optimal_control}
If $\alpha + \beta < 1$, there exists an optimal control $\hat{\nu}$ for the MFC problem.
The process $\hat{\nu}$ reflects the state process $X^{\hat{\nu}}$ upwards à la Skorohod at the barrier $\hat{a}$ given by
\begin{equation}\label{mfc:optimal_reflection}
\hat{a} = \left[ \frac{2(\alpha + \beta)}{q(2\delta + \sigma^2(1-\alpha))} \left( \frac{2\delta+\sigma^2}{2\delta} \right)^{\beta} \right] ^{\frac{1}{1 - \alpha -\beta}},
\end{equation}
and the corresponding stationary mean is given by 
\begin{equation}\label{mfc:stationary_moment}
    \hat{\theta} = \left( \frac{2\delta+\sigma^2}{2\delta} \right)^{\frac{1-\alpha}{1 - \alpha -\beta}} \left[ \frac{2(\alpha + \beta)}{q(2\delta + \sigma^2(1-\alpha))} \right]^{\frac{1}{1-\alpha -\beta}}.
\end{equation}
If instead $\alpha + \beta > 1$, the problem is ill-posed, in the sense that 
\[
\sup_{\nu \in \A_{MFC}}\J(\nu,\E[X^{\nu}_\infty]) = +\infty.
\]
Finally, if $\alpha + \beta = 1$ and 
\begin{equation}\label{eq:mfc:condition_zero_infinity}
    \frac{2\delta + \sigma^2}{2\delta + \sigma^2(1-\alpha)}\left( \frac{2\delta}{2\delta+\sigma^2}\right)^\alpha  < q\delta,
\end{equation}
the null control $\nu \equiv 0$ is optimal; otherwise, the problem is ill-posed.
\end{thm}
\begin{proof}
We note that
\begin{equation}\label{mfc:lagrangian:1st_expression}
\sup_{\nu \in \A_{MFC}} \J(\nu,\E[X^{\nu}_\infty]) = \sup_{\theta > 0} \sup_{\substack{\nu \in \A_{MFC} \\ \E[X^{\nu}_\infty] = \theta}} \J(\nu,\theta) = \sup_{\theta > 0} \sup_{\substack{\nu \in \A_{MFC} \\ \E[X^{\nu}_\infty] = \theta}} \tonde{\J(\nu,\theta) +\lambda(\theta)(\E[X^{\nu}_\infty] -\theta)},
\end{equation}
where $\lambda:\R_+ \to \R$ is any real function of the mean $\theta$.
Moreover, for $\nu \in \A_{MFC}$ so that $\E[X^\nu_\infty] = \theta$, we rewrite the right-hand side term of \eqref{mfc:lagrangian:1st_expression} using ergodicity:
\begin{equation}\label{mfc:lagrangian:2nd_expression}
\begin{aligned}
&\J(\nu,\theta) +\lambda(\theta)( \E[X^{\nu}_\infty] -\theta) = \J(\nu,\theta) + \lim_{T \uparrow \infty} \frac{1}{T}\int_0^T\E[\lambda(\theta) X^{\nu}_t -\lambda(\theta) \theta]dt \\
& = -\lambda(\theta) \theta + \liminf_{T \uparrow \infty} \frac{1}{T}\E\left[ \int_0^T \tonde{(X^{\nu}_t)^\alpha \theta^{\beta} + \lambda(\theta) X^{\nu}_t}dt - q\nu_T \right] = -\lambda(\theta)\theta + \Tilde{\J}(\nu,\theta^\beta,\lambda(\theta)),
\end{aligned}
\end{equation}
where $\Tilde{\J}$ is defined in \eqref{lemma:ergodic_control:payoff}.
We split the problem in the following three steps:
\begin{enumerate}[label=\arabic*),wide]
    \item \label{mfc:thm_opt_ctrl:step1} For fixed $\theta$ and $\lambda$, show that there exists a unique optimal control of barrier type which maximizes $\Tilde{J}(\nu,\theta^\beta,\lambda)$ over $\A_{mf}$.
    Denote by $\hat{a}(\theta,\lambda)$ the optimal reflection barrier.
    \item \label{mfc:thm_opt_ctrl:step2} Show that for any $\theta>0$ there exists a real value $\lambda(\theta)$ so that $ \theta=\E[X^{\nu^{\hat{a}(\theta,\lambda(\theta))}}_\infty]$ and deduce that
    \[
    \sup_{\substack{\nu \in \A_{MFC} \\ \E[X^{\nu}_\infty] = \theta}} \liminf_{T \uparrow \infty} \frac{1}{T}\E\left[ \int_0^T \tonde{(X^{\nu})^\alpha \theta^{\beta} + \lambda(\theta) X^{\nu}_t}dt - q\nu_T \right] = \Tilde{\J}(\nu^{\hat{a}(\theta,\lambda(\theta))},\theta^\beta,\lambda(\theta)).
    \]
    \item \label{mfc:thm_opt_ctrl:step3} Perform the optimization over $\theta \in \R_+$.
\end{enumerate}

As for Step \ref{mfc:thm_opt_ctrl:step1}, we restrict to $\lambda < q \delta$, so that, by applying Lemma \ref{lemma:ergodic_control} with $p = \theta^\beta$, the reflection policy at the level $\hat{a}(\theta,\lambda)=a^*(\theta^\beta,\lambda)$ given by \eqref{lemma:ergodic_control:barrier} is optimal.

\smallskip
As for Step \ref{mfc:thm_opt_ctrl:step2}, we look for $\lambda(\theta)$ so that the stationary distribution satisfies $\E[X^{\nu^{\hat{a}(\theta,\lambda(\theta))}}_\infty] = \theta $ and the condition $q \delta - \lambda(\theta) > 0$ holds.
In view of \eqref{geometric_bm:reflected:barrier_mean}, this is equivalent to imposing
\[
\frac{2\delta + \sigma^2}{2\delta} \tonde{ \frac{2\delta + \sigma^2(1-\alpha) }{ 2\alpha\delta } \frac{q\delta -\lambda}{\theta^{\beta}} }^{\frac{1}{\alpha - 1}} = \theta, \quad q \delta - \lambda(\theta) > 0
\]
which are both satisfied by 
\begin{equation}\label{mfc:lagrangian_multiplier}
    \lambda(\theta) = q\delta - \tonde{\frac{2\delta+\sigma^2}{2\delta}  }^{1- \alpha} \frac{2\delta\alpha}{2\delta + \sigma^2(1-\alpha)} \theta^{\alpha + \beta - 1}.
\end{equation}
Therefore, by choosing $\lambda(\theta)$ as the Lagrangian multiplier in \eqref{mfc:lagrangian:1st_expression}, we have that 
\begin{equation*}
\begin{aligned}
    & \sup_{\nu \in \A_{MFC} } \J(\nu,\E[X^{\nu}_\infty]) = \sup_{\theta>0} \bigg( -\lambda(\theta) \theta + \sup_{\substack{\nu \in \A_{MFC} \\ \E[X^{\nu}_\infty] = \theta}}  \Tilde{\J}(\nu;\theta,\lambda(\theta) ) \bigg) \\
    & = \sup_{\theta>0} \left( -\lambda(\theta) \theta +\Tilde{\J}(\nu^{\hat{a}(\theta,\lambda(\theta))};\theta,\lambda(\theta)) \right) = \sup_{\theta>0} \J(\nu^{\hat{a}(\theta,\lambda(\theta))},\theta) = \sup_{\theta>0} \J(\nu^{\hat{a}(\theta,\lambda(\theta))},\E[X^{\nu^{\hat{a}(\theta,\lambda(\theta))}}_\infty]).
\end{aligned}
\end{equation*}

\medskip
We are left with performing the optimization over $\theta \in \R_+$.
By exploiting \eqref{geometric_bm:reflected:barrier_mean} and \eqref{alvarez:ergodic_functional}, for every $\theta>0$ we have 
\[
\J(\nu^{\hat{a}(\theta,\lambda(\theta))},\E[X^{\nu^{\hat{a}(\theta,\lambda(\theta))}}_\infty]) = C(\hat{a}(\theta,\lambda(\theta)),\theta)=\frac{2\delta + \sigma^2}{2\delta + \sigma^2(1-\alpha)}\left( \frac{2\delta}{2\delta+\sigma^2}\right)^\alpha \theta^{\alpha+\beta} - q\delta \theta.
\]
Set
\begin{equation}\label{mfc:mean_dependence_function}
    f(\theta) := \frac{2\delta + \sigma^2}{2\delta + \sigma^2(1-\alpha)}\left( \frac{2\delta}{2\delta+\sigma^2}\right)^\alpha \theta^{\alpha+\beta} - q\delta \theta.
\end{equation}
If $\alpha + \beta < 1$, we have
\begin{align*}
    & f'(\theta) = (\alpha + \beta ) \tonde{\frac{2\delta}{ 2\delta+\sigma^2 } }^{\alpha}\frac{2\delta + \sigma^2}{2\delta + \sigma^2(1-\alpha)}\theta^{\alpha + \beta -1} -  q\delta,
\end{align*}
so that $f''(\theta) < 0$ for every $\theta>0$, i.e. $f$ is strictly concave in $\R_+$.
This implies that there exists a unique maximizer $\hat{\theta}$.
By imposing $f'(\theta) = 0$, we find the expression of $\hat{\theta}$ given by \eqref{mfc:stationary_moment}, and by \eqref{geometric_bm:reflected:barrier_mean} we find the expression of $\hat{a}$ in \eqref{mfc:optimal_reflection}.

\smallskip
If $\alpha + \beta >1$, the function $f$ defined in \eqref{mfc:mean_dependence_function} is unbounded, and therefore the MFC problem does not admit a maximizer.
Finally, suppose $\alpha + \beta = 1$.
Then, the function $f$ is just given by 
\[
f(\theta) = \left( \frac{2\delta + \sigma^2}{2\delta + \sigma^2(1-\alpha)}\left( \frac{2\delta}{2\delta+\sigma^2}\right)^\alpha  - q\delta \right) \theta.
\]
Since $f$ is linear in $\theta$, we either have $\sup_{\theta>0}f(\theta)$ equal to $0$ or $+\infty$ depending on the sign of the coefficient.
\end{proof}

\begin{remark}\label{mfc:rmk_assumptions}
Assumptions \ref{assumption_F_0} and \ref{assumption:dissipativity} are not needed to prove Proposition \ref{mfc:thm:optimal_control}.
On the other hand, according to the previous result, restrictions on the parameters are needed in order to have existence of the optimal control.
Notice that, if an optimal control exists, it is always unique, by strict concavity of the reward functional.
\end{remark}

For the sake of completeness and for later use, we derive the relationship between the Lagrangian multiplier $\lambda(\theta)$ and the constraint parameter $\theta$.

\begin{lemma}\label{central_planner:lemma:envelope_theorem}
Let $\nu = \nu^{a(\theta)}$ be the strategy which reflect the process $X^{\nu^{a(\theta)}}$ upwards at the barrier $a(\theta) = \frac{2\delta}{2\delta + \sigma^2}\theta$.
Let $\lambda(\theta)$ be given by \eqref{mfc:lagrangian_multiplier} and $f$ given by \eqref{mfc:mean_dependence_function}.
Then, for any $\theta > 0$, it holds
\begin{equation}\label{central_planner:envelope_theorem}
f'(\theta) + \lambda(\theta) = \liminf_{T \uparrow \infty} \frac{1}{T}\E\left[\int_0^T \pi_\theta(X^{\nu^{a(\theta)}}_t,\theta)dt \right].
\end{equation}
\end{lemma}
The proof is postponed to the \nameref{sec:appendix}.

\begin{remark}\label{rmk:derivative_value_function}
The same calculations of Lemma \ref{central_planner:lemma:envelope_theorem} show that the following relation holds:
\[
\lambda(\hat{\theta}) =  \partial_\theta \J(\hat{\nu},\hat{\theta}) = \partial_\theta \left( \sup_{\nu \in \A_{MFC}} \J(\nu,\E[X^\nu_\infty]) \right).
\]
In this sense, the Lagrange multiplier $\lambda(\hat{\theta})$ can be regarded as the derivative of the value function with respect to the  mean-field term.
\end{remark}

\subsection{Approximation}
We show that any solution to the MFC problem as given by Theorem \ref{mfc:thm:optimal_control} induces a sequence $(\boldsymbol{\hat{\nu}}^{N})_{N \geq 1}$ of approximate optimal strategy profiles for the central planner in the smaller class of strategy profiles $\A^{cp,\kappa}_N \subseteq \A^N_N$, defined in subsequent Definition \ref{central_planner:admissible_strategies}, with vanishing error.
Our approach is mainly inspired by \cite[Section 6]{carmona2015AP}, although it uses different techniques due to the nature of our dynamics and payoff. Notice that the exchangeability assumption is not required.

\medskip
Let $N \geq 2$.
We consider the following set $\mathcal{C} \subseteq \A_N^N$ of strategies for the central planner.
\begin{definition}[Admissible strategies for the central planner optimization problem]\label{central_planner:admissible_strategies}
Let $\kappa>0$ be a fixed constant.
A strategy profile $\boldsymbol{\beta}^{N} = (\beta^{1,N},\dots,\beta^{N,N})$ is an admissible strategy profile for the  central planner optimization problem if $\beta^{i,N} \in \A_{N}$ for any $i=1,\dots,N$ and if they are ergodic, in the sense that the resulting $N$-dimensional process $(X_t^{\beta^{1,N}},\dots, X^{\beta^{N,N}}_t)_{t \geq 0}$ is ergodic, where $X^{\beta^{i,N}}$ is given by \eqref{N_player:dynamics} for every $i=1,\dots,N$.
Moreover, every strategy $\beta^j$ is such that
\begin{equation}\label{central_planner:admisible:bound_momenti}
\sup_{T > 0}\frac{1}{T}\E\left[ \int_0^T \vert X^{\beta^j}_t \vert^2 dt  \right] \leq \kappa, \quad \forall \; j=1,\dots,N.
\end{equation}
We denote the set of strategy profiles for the central planner by $\A^{cp,\kappa}_N$. 
\end{definition}
Observe that the inclusion $\A^{cp,\kappa}_N \subseteq \A^N_N$ holds strictly.
In particular, the ergodicity request is needed in order to freely exchange sums and inferior limits throughout the proof, while the restriction to strategy profiles satisfying the bound \eqref{central_planner:admisible:bound_momenti} allows to get uniform estimates in the aforementioned class of admissible strategies.

\begin{thm}\label{central_planner:thm:approximation}
Let $\boldsymbol{\hat{\nu}}^{N} = (\hat{\nu}^{1},\dots,\hat{\nu}^{N})$ be the strategy profile that reflects each process $X^{\hat{\nu}^i}$ upwards à la Skorohod at the level $\hat{a}$ given by \eqref{mfc:optimal_reflection}.
It holds
\begin{equation}
    \lim_{N \to \infty} \sup_{\boldsymbol{\beta}^{(N)} \in \A^{cp,\kappa}_N} \Bar{\J}^N(\boldsymbol{\beta}^{(N)}) = \lim_{N \to \infty} \Bar{\J}^N(\boldsymbol{\hat{\nu}}^{N})= \J(\hat{\nu},\E[X^{\hat{\nu}}_\infty]).
\end{equation}
\end{thm}

\begin{proof}
Notice that, for every $N \geq 2$ and $\kappa$ large enough, $\boldsymbol{\hat{\nu}}^{N}$ belongs to $\A^{cp,\kappa}_N$.
By \cite[Lemmata 23.17-19]{kallenberg_foundations}, the $N$-dimensional process $(X^{\hat{\nu}^i})_{i = 1}^N$ is a positively recurrent regular diffusion with ergodic measure $\bigotimes_{i = 1}^N \hat{p}_\infty(dx_i)$, where we set $\hat{p}_\infty = p_{\hat{a}}$, with $p_a$ given by \eqref{gemetric_bm:reflected:density} and $\hat{a}$ given by \eqref{mfc:optimal_reflection}.
Finally, up to choosing $\kappa$ large enough, the processes $(X^{\hat{\nu}^i})_{i = 1}^N$ satisfy \eqref{central_planner:admisible:bound_momenti} by point \ref{geometric_bm:reflected:lemma:integrability} of Lemma \ref{geometric_bm:reflected:lemma}.

\smallskip
We show that
\begin{align*}
    & \lim_{N \to \infty} \inf_{\boldsymbol{\beta}^{(N)} \in \A^{cp,\kappa}_N } \left( \J(\hat{\nu},\E[X^{\hat{\nu}}_\infty]) - \Bar{\J}^N(\boldsymbol{\beta}^{(N)})\right) \geq 0, && \lim_{N \to \infty} \left( \J(\hat{\nu},\E[X^{\hat{\nu}}_\infty]) - \Bar{\J}^N(\boldsymbol{\nu}^{(N)})\right) = 0.
\end{align*}
Observe that, by Lemma \ref{geometric_bm:reflected:lemma}, the inferior limit in the definition of $\J(\hat{\nu}^i,\E[X^{\hat{\nu}^i}_\infty])$ is actually a limit.
Then, by using the inequalities $\limsup_{n} z_n -\liminf_n x_n \geq \limsup_{n}(z_n - y_n)$ and $\limsup_{n}(z_n + y_n) \geq \limsup_n z_n + \liminf_n (y_n)$ and concavity of $\pi(x,\theta) = x^{\alpha}\theta^\beta$ jointly in $(x,\theta)$, it holds
\begin{equation}\label{N_agents:thm:inequalities1}
\begin{aligned}
    & \J(\hat{\nu}^i,\E[X^{\hat{\nu}^i}_\infty]) - \J^{N}(\beta^{i,N},\boldsymbol{\beta}^{-i,N}) \geq \limsup_{T \uparrow \infty} \frac{1}{T} \E\left[ \int_0^T \left( \pi(X^{\hat{\nu}^i}_t,\hat{\theta}) - \pi(X^{\beta^{i,N}}_t,\theta^{N,\boldsymbol{\beta}^{-i,N}}_t) \right)dt - q(\hat{\nu}^i_T - \beta^{i,N}_T) \right] \\
    & \geq \limsup_{T \uparrow \infty}\frac{1}{T} \E \left[ \int_0^T  \left( \pi_x(X^{\hat{\nu}^i}_t,\hat{\theta}) + \lambda(\hat{\theta})\right)( X^{\hat{\nu}^i}_t - X^{\beta^{i,N}}_t ) dt - q(\hat{\nu}^i_T - \beta^{i,N}_T) \right]  \\
    & \; + \liminf_{T \uparrow \infty}\frac{1}{T} \E \left[ \int_0^T  \left( -\lambda(\hat{\theta})( X^{\hat{\nu}^i}_t - X^{\beta^{i,N}}_t ) + \pi_\theta(X^{\hat{\nu}^i}_t,\hat{\theta})(\hat{\theta} - \theta^{N,\boldsymbol{\beta}^{-i,N}}_t) \right) dt \right] \\
    & \geq \liminf_{T \uparrow \infty}\frac{1}{T} \E \left[ \int_0^T  -\lambda(\hat{\theta})( X^{\hat{\nu}^i}_t - X^{\beta^{i,N}}_t ) dt \right]  + \liminf_{T \uparrow \infty}\frac{1}{T} \E \left[\int_0^T \pi_\theta(X^{\hat{\nu}^i}_t,\hat{\theta})(\hat{\theta} - \theta^{N,\boldsymbol{\beta}^{-i,N}}_t) dt \right]
\end{aligned}
\end{equation}
where we added and subtracted $\lambda(\hat{\theta})(X^{\hat{\nu}^i}_t - X^{\beta^i}_t)$ inside the time integral.
Last inequality follows from sublinearity of the $\liminf$ and from Lemma \ref{lemma:first_order_condition}, as the pair $(X^{\hat{\nu}^i},\hat{\nu}^i)$ has the same distribution as $(X^{\hat{\nu}},\hat{\nu})$.
Moreover, since $(X^{\hat{\nu}^i})_{i \geq 1}$ are i.i.d. copies of $X^{\hat{\nu}}$, by Lemma \ref{central_planner:lemma:envelope_theorem}, it holds
\begin{equation}\label{eq:probabilistic_representation}
\lambda(\hat{\theta}) = \lim_{T \uparrow \infty} \frac{1}{T}\E\left[\int_0^T \pi_\theta(X^{\hat{\nu}^i}_t,\hat{\theta})dt \right] \quad \forall i \geq 1.
\end{equation}
By using the identity $\frac{1}{N-1}\sum_{j \neq i} y_j = \frac{N}{N-1}(\frac{1}{N}\sum_{i = 1}^N y_j - \frac{1}{N}y_i)$, the ergodicity of all involved processes and equation \eqref{eq:probabilistic_representation}, we get
\begin{align*}
    \liminf_{T \uparrow \infty}\frac{1}{T} \E \bigg[\int_0^T \pi_\theta(X^{\hat{\nu}^i}_t,\hat{\theta})(\hat{\theta} - \theta^{N,\boldsymbol{\beta}^{-i,N}}_t) dt \bigg] & = \frac{N}{N-1}\liminf_{T \uparrow \infty}\frac{1}{T} \E \bigg[\int_0^T \pi_\theta(X^{\hat{\nu}^i}_t,\hat{\theta})\Big(\frac{1}{N}\sum_{j = 1}^N (\hat{\theta} - X_t^{\beta^{j,N}}) \Big) dt \bigg] \\
    & - \frac{1}{N-1}\liminf_{T \uparrow \infty}\frac{1}{T} \E \bigg[\int_0^T \pi_\theta(X^{\hat{\nu}^i}_t,\hat{\theta})(\hat{\theta} - X_t^{\beta^{-i,N}}) dt \bigg], \\
    \liminf_{T \uparrow \infty}\frac{1}{T} \E \bigg[ \int_0^T  -\lambda(\hat{\theta})( X^{\hat{\nu}^i}_t - X^{\beta^{i,N}}_t ) dt \bigg] &  = \frac{N}{N-1}\liminf_{T \uparrow \infty}\frac{1}{T} \E \bigg[ \int_0^T  -\lambda(\hat{\theta})( \hat{\theta} - X^{\beta^{i,N}}_t ) dt \bigg] \\
    & -\frac{1}{N-1}\liminf_{T \uparrow \infty}\frac{1}{T} \E \bigg[ \int_0^T  -\lambda(\hat{\theta})( \hat{\theta} - X^{\beta^{i,N}}_t ) dt \bigg].
\end{align*}
Thus, by taking the average over $i=1,\dots,N$, recalling that $(X^{\hat{\nu}^i})_{i=1}^N$ are identically distributed as $X^{\hat{\nu}}$, we get
\begin{align*}
    & \J(\hat{\nu},\E[X^{\hat{\nu}}_\infty]) - \Bar{\J}^N(\boldsymbol{\beta}^{(N)}) \geq \frac{N}{N-1} \liminf_{T \uparrow \infty} \frac{1}{T} \int_0^T \E \bigg[ \Big( \frac{1}{N}\sum_{i=1}^N ( \pi_\theta(X^{\hat{\nu}^i}_t,\hat{\theta}) -\lambda(\hat{\theta}) )   \Big) \Big( \frac{1}{N}\sum_{j=1}^N (\hat{\theta}-X^{\beta^{j,N}}_t)  \Big) \\
    & \quad - \frac{1}{N-1}\frac{1}{N} \sum_{i=1}^N (\pi_\theta(X^{\hat{\nu}^i}_t,\hat{\theta}) -\lambda(\hat{\theta}))( \hat{\theta}-X^{\beta^{i,N}}_t ) \bigg]dt \\
    & \geq - c\Bigg( \limsup_{T \uparrow \infty} \frac{1}{T} \E \int_0^T \bigg[ \Big\vert \frac{1}{N}\sum_{i=1}^N ( \pi_\theta(X^{\hat{\nu}^i}_t,\hat{\theta}) -\lambda(\hat{\theta}) )   \Big\vert \Big\vert \frac{1}{N}\sum_{j=1}^N (\hat{\theta}-X^{\beta^{j,N}}_t)  \Big\vert \bigg]dt \\
    & \quad + \frac{1}{N-1} \limsup_{T \uparrow \infty} \frac{1}{T}  \int_0^T \E\bigg[ \Big\vert \frac{1}{N}\sum_{i=1}^N (\pi_\theta(X^{\hat{\nu}^i}_t,\hat{\theta}) -\lambda(\hat{\theta}))( \hat{\theta}-X^{\beta^{i,N}}_t ) \Big\vert \bigg]dt \Bigg).
\end{align*}
We study separately the two $\limsup$.
As for the the second one, by taking advantage the exchangeability of $(X^{\hat{\nu}^i})_{i \geq 1}$ and the uniform bound \eqref{central_planner:admisible:bound_momenti} to handle $(X^{\beta^{i,N}})_{i=1}^N$, we get
\begin{align*}
    & \frac{1}{N} \limsup_{T \uparrow \infty} \frac{1}{T}  \int_0^T \E\bigg[ \Big\vert \frac{1}{N}\sum_{i=1}^N (\pi_\theta(X^{\hat{\nu}^i}_t,\hat{\theta}) -\lambda(\hat{\theta}))( \hat{\theta}-X^{\beta^{i,N}}_t ) \Big\vert \bigg]dt \\
    & \leq \frac{1}{N} \limsup_{T \uparrow \infty} \bigg(\frac{1}{N}\sum_{i=1}^N  \frac{1}{T} \int_0^T \E\bigg[ \Big\vert \pi_\theta(X^{\hat{\nu}^i}_t,\hat{\theta}) -\lambda(\hat{\theta}) \Big\vert^2 \bigg] dt \bigg)^{\frac{1}{2}} \bigg(\frac{1}{N}\sum_{i=1}^N \frac{1}{T} \int_0^T \E\bigg[ \Big\vert \hat{\theta} - X^{\beta^{i,N}}_t \Big\vert^2 \bigg] dt \bigg)^{\frac{1}{2}} \\
    & \leq \frac{c}{N} \limsup_{T \uparrow \infty} \bigg( \frac{1}{T} \int_0^T \E\bigg[ \Big\vert \pi_\theta(X^{\hat{\nu}}_t,\hat{\theta}) -\lambda(\hat{\theta}) \Big\vert^2 \bigg] dt \bigg)^{\frac{1}{2}} \leq \frac{c}{N},
\end{align*}
and then it goes to $0$ as $N \to \infty$.
As for the first term, by analogous computations, we get
\begin{align*}
    & \limsup_{T \uparrow \infty} \frac{1}{T} \E \int_0^T \bigg[ \Big\vert \frac{1}{N}\sum_{i=1}^N ( \pi_\theta(X^{\hat{\nu}^i}_t,\hat{\theta}) -\lambda(\hat{\theta}) )   \Big\vert \Big\vert \frac{1}{N}\sum_{j=1}^N (\hat{\theta}-X^{\beta^{j,N}})  \Big\vert \bigg]dt \\
    & \leq \limsup_{T \uparrow \infty} \bigg( \frac{1}{T} \E \int_0^T \bigg[ \Big \vert \frac{1}{N}\sum_{i=1}^N ( \pi_\theta(X^{\hat{\nu}^i}_t,\hat{\theta}) -\lambda(\hat{\theta}) ) \Big\vert^2  \bigg]dt \bigg)^\frac{1}{2} \bigg( \frac{1}{T}\int_0^T  \frac{1}{N}\sum_{j=1}^N \E\bigg[ \Big\vert\hat{\theta}-X_t^{\beta^{j,N}} \Big\vert^{2} \bigg ]dt\bigg)^\frac{1}{2} \\
    & \leq c \limsup_{T \uparrow \infty} \bigg( \frac{1}{T} \E \int_0^T \bigg[ \Big \vert \frac{1}{N}\sum_{i=1}^N ( \pi_\theta(X^{\hat{\nu}^i}_t,\hat{\theta}) -\lambda(\hat{\theta}) ) \Big\vert^2  \bigg]dt \bigg)^\frac{1}{2},
\end{align*}
for a constant $c$ independent of $N$, by using the bound \eqref{central_planner:admisible:bound_momenti} and exchangeability of $ \boldsymbol{\beta}^{(N)} \in \A^{cp,\kappa}_N$.
By the ergodic ratio theorem, it holds 
\begin{equation}\label{central_planner:ergodic_ratio_pointwise}
    \lim_{T \uparrow \infty} \frac{1}{T} \int_0^T  \Big\vert \frac{1}{N}\sum_{i=1}^N \left( \pi_\theta(X^{\hat{\nu}^i}_t,\hat{\theta}) -\lambda(\hat{\theta}) \right) \Big\vert^2  dt = \int_{\R_+^{N}} \Big\vert \frac{1}{N}\sum_{i=1}^N \left( \pi_\theta(x_i,\hat{\theta}) -\lambda(\hat{\theta})  \right) \Big\vert^2 \bigotimes_{i=1}^N \hat{p}_\infty(dx_i),
\end{equation}
$\prob$-a.s. We show that the right hand-side is uniformly integrable.
Take $r = \sfrac{1}{\alpha} > 1$.
By Jensen inequality and identical distribution of $(X^{\hat{\nu}^i})_{i \geq 1}$, we have
\begin{align*}
    & \E\left[\left( \frac{1}{T} \int_0^T  \Big\vert \frac{1}{N}\sum_{i=1}^N \left( \pi_\theta(X^{\hat{\nu}^i}_t,\hat{\theta}) -\lambda(\hat{\theta}) \right) \Big\vert^2 dt \right)^r \right] \leq \frac{C}{T} \E\left[ \int_0^T  \Big\vert \Big( \frac{1}{N}\sum_{i=1}^N \pi_\theta(X^{\hat{\nu}^i}_t,\hat{\theta}) -\lambda(\hat{\theta}) \Big) \Big\vert^{2r} dt \right] \\
    & \leq C \left( \vert \lambda(\hat{\theta})\vert^{2r} + \frac{1}{N} \sum_{i=1}^N  \frac{1}{T} \int_0^T  \E\left[ \vert \pi_\theta(X^{\hat{\nu}^i}_t,\hat{\theta}) \vert^{2r} \right] dt  \right)  \leq  C \left( 1 + \frac{1}{T} \int_0^T  \E\left[ \vert \pi_\theta(X^{\hat{\nu}}_t,\hat{\theta}) \vert^{2r} \right] dt \right)  \\
    & \leq C \left( 1 + \frac{1}{T} \int_0^T \E\left[\vert (X^{\hat{\nu}}_t)^\alpha  \vert^{2r} \right] dt \right) = C \left( 1 + \frac{1}{T} \int_0^T \E\left[  \vert (X^{\hat{\nu}}_t) \vert^{2}\right] dt \right) \leq C(1 + \hat{a}^2),
\end{align*}
where we used Lemma \ref{geometric_bm:reflected:lemma} in the last estimate.
Since last estimate holds for any $T>0$ and $r>1$, we deduce that the right hand-side of \eqref{central_planner:ergodic_ratio_pointwise} is uniformly integrable.
By, e.g., \cite[Lemma 4.12]{kallenberg_foundations}) this yields
\begin{align*}
    & \lim_{T \uparrow \infty} \frac{1}{T} \E \bigg[ \int_0^T  \Big\vert \frac{1}{N}\sum_{i=1}^N \left( \pi_\theta(X^{\hat{\nu}^i}_t,\hat{\theta}) -\lambda(\hat{\theta}) \right) \Big\vert^2 dt \bigg] = \int_{\R_+^{N}} \Big\vert \frac{1}{N}\sum_{i=1}^N \left( \pi_\theta(x_i,\hat{\theta}) -\lambda(\hat{\theta})  \right) \Big\vert^2 \bigotimes_{i=1}^N \hat{p}_\infty(dx_i).
\end{align*}
We conclude by invoking the law of large numbers:
let $(X_i)_{i \geq 1}$ be a sequence of i.i.d. random variables with law $\hat{p}_\infty(dx)$.
Therefore,  by Lemmata \ref{geometric_bm:reflected:lemma} and \ref{central_planner:lemma:envelope_theorem}, the sequence $(Y_i)_{i \geq 1}$ defined by $Y_i = \pi_\theta(X_i,\hat{\theta}) -\lambda(\hat{\theta})$ is i.i.d., square integrable and centered, which yields
\begin{align*}
    & \int_{\R_+^{N-1}} \Big\vert \frac{1}{N}\sum_{i=1}^N \left( \pi_\theta(x_i,\hat{\theta}) -\lambda(\hat{\theta}) \right) \Big\vert^2 \bigotimes_{i= 1}^N \hat{p}_\infty(dx_i) = \E\left[ \Big\vert \frac{1}{N}\sum_{i = 1}^{N} Y_i \Big\vert^2 \right] = \frac{\E[\vert Y_1] \vert ^2}{N} \to 0.
\end{align*}

\smallskip
Finally, $\lim_{N \to \infty} \Bar{\J}^N(\boldsymbol{\nu}^{(N)}) = \J(\hat{\nu},\E[X^{\hat{\nu}}_\infty])$ follows form ergodicity of the processes $(X^{\hat{\nu}^i})_{i=1}^N$ and the law of large numbers, analogously as before.
\end{proof}

\section{Competitive Case: Mean-field Equilibria and Approximation}\label{sec:competitive}
In this section, we focus on coarse correlated equilibria for the MFG, as defined by Definition \ref{def:CCE}.
Since the set of CCEs is typically very wide and it is difficult to characterize in a continuous time setting, following the procedure outlined in \cite{campi2023LQ}, we restrict our analysis to specific classes of correlated stationary strategies for which we are able to state a sufficient condition for being a CCE for the ergodic MFG.
With respect to \cite{campi2023LQ}, we move a step forward, and show that every CCE in these classes induces a sequence of approximate CCEs in the underlying $N$-player game with vanishing error.
More specifically, we establish the following:
\begin{itemize}
    \item We fix a correlated stationary strategy $(Z,\lambda,\theta_\infty)$.
    We suppose that the representative player decides to ignore the moderator's recommendation, and compute her best deviating strategy, i.e.
    \[
    U^* = \mathrm{arg}\max_{\nu \in \A_{mf}}\J(\nu,\theta_\infty).
    \]
    This is the content of Proposition \ref{dev_player:prop:optimal_deviation}.
    
    \item We define specific classes of correlated flows $(Z,\lambda,\theta_\infty)$ so that the consistency condition \ref{cce:def:consistency} is satisfied.
    This is established in Propositions \ref{cce:regular:prop_corr_strategy} and \ref{cce:singular:prop_consistency}.
    
    \item For $(Z,\lambda,\theta_\infty)$ in such classes, we express the optimality condition \ref{cce:def:optimality} as an inequality involving the law of $\theta_\infty$ only, thus deriving a sufficient condition for the existence of CCEs. This is the content of Propositions \ref{cce:regular:prop_optimality} and \ref{cce:singular:prop_optimality}.

    \item We show that every mean-field CCEs in each class induces a sequence of aprroximate CCEs in the underlying $N$-player game with vanishing error.
    This is the content of Theorems \ref{N_player:regular:thm_approximation} and \ref{N_player:singular:thm_approximation}.

\end{itemize}
We consider two classes of correlated stationary strategies: while in both classes the correlating device is the random mean $\theta_\infty$ itself, in the first class the recommendation $\lambda^r$ is a $\sigma(\theta_\infty)$-measurable regular control, while in the second one, the recommendation $\lambda^s$ is a policy of reflection type at a random barrier $a(\theta_\infty)$.
Surprisingly, the sufficient condition of the two classes differ only by a constant.
Moreover, in Theorem \ref{mfg:thm:nash_eq}, we explicitly characterize existence and uniqueness of Nash equilibria. We find that they belong to class of CCEs with recommendation of reflection type, so that the same approximation result applies.

\subsection{The Deviating Player Problem}
Suppose that the representative player decides to ignore the moderator's recommendation.
By definition of CCE for the ergodic MFG, the deviating player must choose her strategy $\nu \in \A_{mf}$ only by knowing the joint law of the correlated stationary strategy $(Z,\lambda,\theta_\infty)$, which is assumed to be publicly known, and not by observing its realizations.
Since $\nu \in \A_{mf}$, it follows that $X^{\nu}$ is $\mathbb{F}^{\xi,W}$-adapted as well and thus independent of the random variable $\theta_\infty$, which implies that deviating player's payoff can be written as:
\begin{equation}\label{dev_player:payoff}
\begin{aligned}
\J & (\nu,\theta_\infty) = \liminf_{T \uparrow \infty} \frac{1}{T} \left( \int_0^T \E\left[ \E[ (X^{\nu}_t)^\alpha \theta_\infty^{\beta} \vert \F^{\xi,W}_t ] \right]dt - q\E\left[ \E[\nu_T \vert \F^{\xi,W}_T] \right] \right) \\
& = \liminf_{T \uparrow \infty}\frac{1}{T}\E\left[ \int_0^T (X^{\nu}_t)^\alpha \E[\theta_\infty^{\beta}]dt - q \nu_T \right].
\end{aligned}
\end{equation}
Observe that deviating player's payoff functional depends on $\theta_\infty$ only through its expectation.

\begin{prop}\label{dev_player:prop:optimal_deviation}
There exists a unique optimal strategy for the deviating player $U^* \in \A_{mf}$ which reflects the process $X^{U^*}$ upwards \'{a} la Skorohod at the level $a^*$, where $a^*$ is given by
\begin{equation}\label{dev_player:opt_reflection_boundary}
    a^* = \left( \frac{2\alpha}{q(2\delta + \sigma^2(1 - \alpha))} \right)^{\frac{1}{1-\alpha}}(\E[\theta_\infty^{\beta}])^{\frac{1}{1- \alpha}}.
\end{equation}
\end{prop}
\begin{proof}
Since the payoff functional of the deviating player is given by \eqref{dev_player:payoff}, it is enough to apply Lemma \ref{lemma:ergodic_control} with $\lambda = 0$ and $p = \E[\theta_\infty^{\beta}]$.
\end{proof}

In the following, we set
\begin{equation}\label{dev_player:constant_K}
K :=  \tonde{ \frac{2\alpha}{q(2\delta + \sigma^2(1 - \alpha))}}^{\frac{1}{1-\alpha}},
\end{equation}
so that the optimal policy of the deviating player is the reflection at the level $a^* = K (\E[\theta_\infty^\beta])^{\sfrac{1}{1-\alpha}}$.

\subsection{Regular recommendation}
Let $\mathcal{G}^r$ be the set of correlated stationary strategies $(Z,\lambda^r,\theta_\infty)$ so that $\theta_\infty \in L^2(\F_{0^{-}})$ independent of $\xi$ and $W$, $Z=\theta_\infty$ and 
\begin{equation}\label{cce:regular:control}
d\lambda^r_t = \delta \theta_\infty dt, \quad t \geq 0,
\end{equation}
so that, in particular, the recommended strategy $\lambda^r$ is a $\sigma(\theta_\infty)$-measurable regular control.

\begin{prop}\label{cce:regular:prop_corr_strategy}
Let $(Z,\theta_\infty,\lambda^r) \in \mathcal{G}^r$.
Define
\begin{equation}\label{inv_gamma:speed_measure}
    m'_{\infty,\theta}(x) = \frac{2}{\sigma^2}\exp\tonde{\frac{2\delta \theta}{\sigma^2}}x^{-\frac{2\delta}{\sigma^2} - 2}\exp\tonde{-\frac{2\delta}{\sigma^2} \frac{\theta}{x}},
\end{equation}
and let $p^r_\infty(dx,\theta)$ be the stochastic kernel from $\R_+$ to $\mathcal{B}_{\R_+}$ defined by
\begin{equation}\label{cce:regular:limit_kernel}
p^r_\infty(dx,\theta)= \frac{x^{ -\frac{2\delta}{\sigma^2}  - 2} \exp\tonde{ -\frac{2\delta}{\sigma^2} \frac{\theta}{x} } }{\int_0^\infty y^{ -\frac{2\delta}{\sigma^2} -2  } \exp\tonde{ -\frac{2\delta}{\sigma^2} \frac{\theta}{y} } dy } dx.
\end{equation}
Then, the triple $(Z,\theta_\infty,\lambda^r)$ is a correlated stationary strategy so that the consistency condition \eqref{cce:cons_eq} is satisfied.
In particular, it holds $\mu^r_\infty(dx,d\theta) = \lim_{t \to \infty}\prob\circ(X^{\lambda^r}_t,\theta_\infty)^{-1} = p^r_\infty(dx,\theta)\rho(d\theta)$.
\end{prop}
\begin{proof}
The measurability requirements on the triple $(Z,\theta_\infty,\lambda^r)$ are clearly satisfied.
Let $X^{\lambda^r}$ the state process controlled by $\lambda^r$, which satisfies 
\begin{equation}\label{cce:regular:controlled_process}
dX^{\lambda^r}_t = \delta(\theta_\infty - X^{\lambda^r}_t)dt + \sigma X^{\lambda^r}_t dW_t, \quad X^{\lambda^r}_{0} = \xi.
\end{equation}
We show that the joint law of $(X^{\lambda^r}_t,\theta_\infty)$ converges weakly to $\mu^r_\infty$ as $t \to \infty$.
To this extent, it is enough to verify that the regular conditional probability of $X^{\lambda^r}_t$ with respect to $\theta_\infty = \theta$, which we denote by $p^r_t(dx,\theta)$, converges weakly to $p^r_\infty(dx,\theta)$ as $t \to \infty$, for $\rho$-a.e. $\theta \in \R_+$.
Conditionally to $\theta_\infty = \theta$, $X^{\lambda^r}_t$ satisfies the following equation:
\[
dX^{\lambda^r,\theta}_t = \delta(\theta - X^{\lambda^r,\theta}_t )dt + \sigma X^{\lambda^r,\theta}_t dW_t, \quad X^{\lambda^r,\theta}_{0} = \xi.
\]
By Lemma \ref{lemma:inv_gamma}, we have $\int_0^\infty m'_{\infty,\theta}(x)dx < \infty$, which implies that the diffusion $X^{\lambda^r,\theta}$ is positively recurrent for every $\theta > 0$.
Thus, the measure $p^r_\infty(dx,\theta)$ is the unique stationary distribution and $p^r_t(dx,\theta) \to p^r_\infty(dx,\theta)$ in total variation norm (see, e.g., \cite[Paragraph 36]{borodin2002handbook}).
As for equality \eqref{cce:cons_eq}, define $\varphi=(\varphi_t(\theta))_{t \geq 0, \theta>0}$ as
\begin{equation}\label{cce:regular:conditional_mean}
\varphi_t(\theta) = e^{-\delta t}\E[\xi] +\theta(1 - e^{\delta t}).
\end{equation}
By It\^{o}'s formula, it follows that $\varphi_t(\theta_\infty) = \E[X^{\lambda^r}_t \vert \theta_\infty]$ for every $t\geq 0$, $\prob$-a.s., which implies that $\theta_\infty= \lim_{t \to \infty}\E[X_t \vert \theta_\infty]$ $\prob$-a.s., and therefore condition \eqref{cce_cons:conditional_form} is satisfied, and so \eqref{cce:cons_eq}.
\end{proof}

\begin{prop}\label{cce:regular:prop_optimality}
A correlated stationary strategy $(Z,\theta_\infty,\lambda^r)$ in the class $\mathcal{G}^r$
is a mean-field CCE if and only if the following inequality is satisfied:
\begin{equation}\label{cce:regular:optimality:condition_moments}
c_{\beta} (\E[\theta_\infty^{\beta}])^{\frac{1}{ 1-\alpha }} + c_1 \E[\theta_\infty] \leq  c_{\alpha+\beta} \E[\theta_\infty^{\alpha+\beta }],
\end{equation}
where $c_{\beta}$, $c_{\alpha+\beta}$ and $c_1$ are positive constants defined by
\begin{equation}\label{cce:reg_ctrl:constants}
\begin{aligned}
& c_{\beta} := \frac{(2\delta+ \sigma^2)q}{2} \tonde{ \frac{2\alpha}{q(2\delta + \sigma^2(1 - \alpha))}}^{\frac{1}{1-\alpha}}\frac{ 1 -\alpha }{\alpha},  && c_{1} := \delta q, &&& c_{\alpha+\beta} := \tonde{\frac{2\delta}{\sigma^2}}^{ \alpha }	\frac{\Gamma( \frac{2\delta}{\sigma^2} + 1 - \alpha)}{\Gamma( \frac{2\delta}{\sigma^2} + 1 )} .
\end{aligned}
\end{equation}
\end{prop}
\begin{proof}
By Proposition \ref{cce:regular:prop_corr_strategy}, $(Z,\theta_\infty,\lambda^r)$ satisfies the consistency condition \eqref{cce:cons_eq}.
Let $U^*$ be the optimal control for the deviating player given by Proposition \ref{dev_player:prop:optimal_deviation}.
Since $\J(U^*,\theta_\infty) = \max_{\nu \in \A_{mf}} \J(\nu,\theta_\infty)$, we just need to verify that the inequality $\J(\lambda^r,\theta_\infty) \geq \J(U^*,\theta_\infty)$ is equivalent to \eqref{cce:regular:optimality:condition_moments}.
Since $U^*$ is a reflection policy at the level $a^*$ given by \eqref{dev_player:opt_reflection_boundary}, formulae \eqref{alvarez:ergodic_functional} and \eqref{dev_player:payoff} yield
\begin{equation}\label{dev_player:payoff:moment_m_infty}
\begin{aligned}
\J & (U^*,\theta_\infty) = C(a^*,\E[\theta_\infty^{\beta}]) \\
& = \frac{2\delta + \sigma^2 }{2\delta + \sigma^2(1 - \alpha)}\E[\theta_\infty^{\beta}]\left( K(\E[\theta_\infty^{\beta}])^{\frac{1}{1-\alpha}}\right)^\alpha - q \frac{2\delta + \sigma^2}{2} K(\E[\theta_\infty^{\beta}])^{\frac{1}{1-\alpha}} \\
& = (2\delta+ \sigma^2)\tonde{ \frac{1}{2\delta + \sigma^2(1-\alpha)}K^\alpha - q\frac{1}{2}K }(\E[\theta_\infty^{\beta}])^{\frac{1}{1-\alpha}} = c_{\beta}(\E[\theta_\infty^{\beta}])^{\frac{1}{1-\alpha}},
\end{aligned}
\end{equation}
by noticing that, with $K$ given by  \eqref{dev_player:constant_K}, it holds
\[
(2\delta+ \sigma^2)\tonde{ \frac{1}{2\delta + \sigma^2(1-\alpha)}K^{\alpha} -\frac{q}{2}K } = \frac{2\delta+ \sigma^2}{2} q K\tonde{ \frac{1 -\alpha}{\alpha} } = c_{\beta}.
\]

\medskip
As for the payoff associated to the representative player, conditionally to $\theta_\infty = \theta$ and exploiting ergodicity, it holds
\begin{equation*}
\begin{aligned}
\lim_{T \uparrow \infty} & \frac{1}{T}\E\left[ \int_0^T (X^{\lambda^r}_t)^{\alpha} \theta^\beta dt -q\lambda^r_T\big\vert \theta_\infty = \theta \right]  = \int_0^\infty \theta^{\beta} x^\alpha p^r_\infty(dx,\theta)- \delta q \theta  \\
& = \tonde{\frac{2\delta}{\sigma^2}}^{ \alpha } \frac{\Gamma( \frac{2\delta}{\sigma^2} - \alpha  + 1)}{\Gamma( \frac{2\delta}{\sigma^2} + 1 )} \theta^{\alpha +\beta } - \delta q \theta ,
\end{aligned}
\end{equation*}
where last equality follows from the definition of $p^r_\infty(dx,\theta)$ and Lemma \ref{lemma:inv_gamma}.
Moreover, by \eqref{cce:regular:conditional_mean}, we have the bound
\begin{align*}
    \frac{1}{T} & \E\left[ \int_0^T (X^{\lambda^r}_t)^{\alpha} \theta^\beta dt -q\lambda^r_T\big\vert \theta_\infty \right] \leq \frac{\theta^\beta}{T}\E\left[\int_0^T(1 + X^{\lambda^r}_t)dt \Big\vert  \theta_\infty \right] + \delta q \theta_\infty \\
    & \leq \theta^\beta \left(1 + \sup_{t \geq 0}\E\left[X^{\lambda^r}_t\Big\vert  \theta_\infty \right] \right)+ \delta q \theta_\infty \leq C(1 + \theta_\infty^2),
\end{align*}
which is integrable by assumption.
Therefore, by dominated convergence theorem, we can exchange limit and expectation to conclude
\begin{equation*}
\begin{aligned}
\J & (\lambda^r,\theta_\infty) = \lim_{T \uparrow \infty} \frac{1}{T} \E\left[ \E\left[ \int_0^T (X^{\lambda^r}_t)^{\alpha} \theta^\beta dt -q\lambda^r_T\big\vert \theta_\infty \right] \right] = \E\left[ \lim_{T \uparrow \infty}\frac{1}{T}\E\left[ \int_0^T (X^{\lambda^r}_t)^{\alpha} \theta^\beta dt -q\lambda^r_T\big\vert \theta_\infty \right] \right] \\
& = \tonde{\frac{2\delta}{\sigma^2}}^{ \alpha } \frac{\Gamma( \frac{2\delta}{\sigma^2} - \alpha  + 1)}{\Gamma( \frac{2\delta}{\sigma^2} + 1 )} \int_0^\infty \theta^{\alpha +\beta }\rho(d\theta) - \delta q \int_0^\infty \theta \rho(d\theta),
\end{aligned}
\end{equation*}
By comparing this equation with equation \eqref{dev_player:payoff:moment_m_infty}, we get equation \eqref{cce:regular:optimality:condition_moments}.
\end{proof}

\begin{remark}\label{regular:rmk_integrability}
Observe that, if $\alpha + \beta \leq 1$, for this result to hold true it is enough to require $\theta_\infty$ to be in $L^1$; Assumption \ref{assumption:dissipativity} is not needed as well.
\end{remark}

Finally, we show how to use a CCE in the class $\mathcal{G}^r$ to build a sequence $\boldsymbol{\lambda}^N$ of $\eps_N$-CCE in the $N$-player game with $\eps_N \to 0$ as $N \to \infty$.
To this extent, we consider the following set of strategies $\mathcal{B} \subseteq \A_N$:
\begin{definition}[$c$-admissible strategies]
Let $c > 0$. A strategy $\nu$ be in $\A_{N}$ is $c$-admissible if 
\begin{equation}\label{N_player:admissible_strategies:integrability}
     \limsup_{T \uparrow \infty} \frac{1}{T} \E\left[\int_0^T \vert X^{\nu}_t \vert^2 dt \right] \leq c.
\end{equation}
We denote by $\A_{N,c}$ the set of $c$-admissible strategies for the $N$-player game.
\end{definition}

Starting from a correlated stationary strategy $(Z,\theta_\infty,\lambda^r)$ in the class $\mathcal{G}^r$, we define the following correlated strategy profiles for the $N$-player game: take $Z = \theta_\infty$ as correlation device, and set, for any $i \geq 1$,
\begin{equation}\label{N_player:regular:recommendation}
    d\lambda^{i,r}_t = \delta \theta_\infty dt.
\end{equation}
Then, when player $i$ plays accordingly to moderator's suggestion, her dynamics are hold by the following equation:
\begin{equation}\label{N_player:dynamics:following_player}
    dX^{i,r}_t = \delta(\theta_\infty - X^{i,r}_t)dt + \sigma X^{i,r}_t dW^i_t, \quad X^{i,r}_0 = \xi^i.
\end{equation}
Observe that, for each $i \geq 1$, the triple $(X^{i,r},\lambda^{i,r},\theta_\infty)$ has the same law as $(X,\lambda,\theta_\infty)$.
Moreover, while not independent, the processes $(X^{i,r})_{i \geq 1}$ are conditionally independent given $\theta_\infty$.

\begin{thm}[Approximation of CCEs - regular case]\label{N_player:regular:thm_approximation}
Let $(Z,\theta_\infty,\lambda^r)$ be a CCE in the class $\mathcal{G}^r$.
Let $\boldsymbol{\lambda}^N = (\lambda^{i,r})_{i=1}^N$, with $\lambda^{i,r}$ defined by \eqref{N_player:regular:recommendation}.
Then, for any $c > 0$, the correlated strategy profile $(\theta_\infty, \boldsymbol{\lambda}^N)$ defines an $\varepsilon_N$-CCE for the $N$-player game within the set of strategies $\A_{N,c}$, with $\varepsilon_N \to 0$ as $N\to\infty$.
\end{thm}

\begin{proof}
For every $N \geq 2$, set 
\[
\varepsilon_N := \sup_{\nu \in \A_{N,c}} \left( \J_N(\nu,\boldsymbol{\lambda}^{-i,N}) - \J_N(\lambda^{i,r},\boldsymbol{\lambda}^{-i,N}) \right).
\]
Notice that, by symmetry, $\varepsilon_N$ is independent of $i=1,\dots,N$.
Clearly $(\theta_\infty,\boldsymbol{\lambda}^N)$ is an $\eps_N$-CCE for the $N$-player game within the set of strategies $\A_{N,c}$.
We show that $\varepsilon_N$ vanishes as $N$ goes to $\infty$.
Note that, for any $\nu$ in $\A_{N,c}$, we have
\begin{equation}\label{N_player:regular:difference_decomposition}
\begin{aligned}
\J_N(\nu,\boldsymbol{\lambda}^{-i,N}) - & \J_N(\lambda^{i,r},\boldsymbol{\lambda}^{-i,N}) = \left( \J_N(\nu,\boldsymbol{\lambda}^{-i,N})  - \J(\nu,\theta_\infty) \right) \\
& + \left( \J(\nu,\theta_\infty) - \J(\lambda^{i,r},\theta_\infty) \right) + \left( \J(\lambda^{i,r},\theta_\infty) - \J_N(\lambda^{i,r},\boldsymbol{\lambda}^{-i,N}) \right)
\end{aligned}
\end{equation}
where $\J$ is defined by \eqref{mfg:payoff}.
We treat separately each of the three terms in the right-hand side of \eqref{N_player:regular:difference_decomposition}.

\smallskip
As for the first term, by Cauchy-Schwartz inequality and using the inequality $\liminf_{n \to \infty}{\alpha_n} - \liminf_{n \to \infty}{\beta_n} \leq \limsup_{n \to \infty}(\alpha_n - \beta_n)$, we have the following estimates:
\begin{equation}\label{N_player:regular:thm:estimates1}
\begin{aligned}
    \vert & \J (\nu,\theta_\infty) - \J_N(\nu,\boldsymbol{\lambda}^{-i,N}) \vert \\
    & \leq  \limsup_{ T \uparrow \infty} \left( \frac{1}{T} \int_0^T \E\left[ (X^{i,\nu}_t)^{2\alpha} \right]dt \right)^{\frac{1}{2}} \left(\frac{1}{T}\int_0^T \E\left[ \left \vert \big( \theta^{N,\boldsymbol{\lambda}^{-i,N}}_t \big)^{\beta} - \theta_\infty^{\beta} \right\vert^2 \right]dt \right)^\frac{1}{2} \\
    & \leq  \left( 1 + \limsup_{ T \uparrow \infty} \frac{1}{T} \int_0^T \E\left[ \vert X^{i,\nu}_t \vert^2 \right]dt \right)^{\frac{1}{2}} \left( \limsup_{ T \uparrow \infty} \frac{1}{T}\int_0^T \E\left[ \E\left[ \left \vert \big( \theta^{N,\boldsymbol{\lambda}^{-i,N}}_t \big)^{\beta} - \theta_\infty^{\beta} \right\vert^2 \Big \vert \theta_\infty \right] \right]dt \right)^\frac{1}{2} \\
    & \leq   (1+c)^{\frac{1}{2}} \left( \E\left[ \limsup_{ T \uparrow \infty} \frac{1}{T}\int_0^T  \E \left[ \left \vert \big( \theta^{N,\boldsymbol{\lambda}^{-i,N}}_t \big)^{\beta} - \theta_\infty^{\beta} \right\vert^2 \Big \vert \theta_\infty\right] dt \right] \right)^\frac{1}{2},
\end{aligned}
\end{equation}
where in the last inequality we exchanged limsup and expectation by reverse Fatou's lemma with the integrable upper bound $C(1 + \theta_\infty^2)$.
Indeed, by recalling that $(X^j)_{j \neq i}$ are i.i.d. as $X$ conditionally to $\theta_\infty$, we have the following estimates
\begin{align*}
\E & \left[ \left \vert \big( \theta^{N,\boldsymbol{\lambda}^{-i,N}}_t \big)^{\beta} - \theta_\infty^{\beta} \right\vert^2 \Big \vert \theta_\infty\right] \leq C\left( 1 + \theta_\infty^2 + \frac{1}{N-1} \sum_{j \neq i} \E \left[ \vert X^{j}_t \vert^{2}  \Big\vert \theta_\infty \right] \right) \leq C(1 + \theta_\infty^2),
\end{align*}
where last inequality holds thanks to Lemma \ref{N_player:regular:lemma_integrability}.
By Lemma \ref{N_player:lemma:ergodic_limit_q_power}, we then have
\begin{align*}
    \limsup_{ T \uparrow \infty} & \frac{1}{T} \int_0^T \E \left[ \left \vert \big( \theta^{N,\boldsymbol{\lambda}^{-i,N}}_t \big)^{\beta} - \theta_\infty^\beta \right \vert^2 \Big \vert \theta_\infty \right] dt \\
    & = \int_{\R_+^{N-1}} \Big \vert \big( \frac{1}{N-1}\sum_{j \neq i}x_j \big)^\beta - \theta_\infty^\beta \Big \vert^2 \bigotimes_{j \neq i}p^r_\infty(dx_j,\theta_\infty), \quad \prob\text{-a.s.}
\end{align*}
which converges to $0$ in expectation as $N$ goes to infinity by Lemma \ref{N_player:regular:lemma:convergence_empirical_measures} and dominated convergence theorem, with $C(1+\theta_\infty^2)$ as the integrable upper bound.

\smallskip
As for the second term, we claim that $\J(\nu,\theta_\infty) - \J(\lambda^{i,r},\theta_\infty) \leq 0$ for any $\nu \in \A_{N,c}$.
Indeed, observe that, since $\nu \in \A_N$ is independent of $\theta_\infty$ by definition of admissible strategies for the $N$-player game, the proof of Proposition \ref{dev_player:prop:optimal_deviation} shows that 
\[
\sup_{\nu \in \A_N} \J(\nu,\theta_\infty) = \J(U^{i,*},\theta_\infty),
\]
where $U^{i,*}$ is the policy that reflects the process $X^{i,U^{i,*}}$ upward à la Skorohod at the level $a^*$ given by \eqref{dev_player:opt_reflection_boundary}.
In particular, $X^{i,U^{i,*}}$ as the same distribution of the process $X^{U^*}$.
Therefore, we have
\[
\sup_{\nu \in \A_{N,c}}\left( \J(\nu,\theta_\infty) - \J(\lambda^{i,r},\theta_\infty) \right) \leq \sup_{\nu \in \A_{N}}\J(\nu,\theta_\infty) - \J(\lambda,\theta_\infty) \leq \J(U^*,\theta_\infty) - \J(\lambda,\theta_\infty) \leq 0
\]
where we used the inclusion $\A_{N,c} \subseteq \A_N$, the fact that $(X^{i,r},\lambda^{i,r},\theta_\infty)$ has the same distribution as $(X,\lambda,\theta_\infty)$ and the optimality property \ref{cce:def:optimality} of CCE of the ergodic MFG.

\smallskip
As for the third term, taking advantage of the conditional independence and identical distribution of $(X^{\lambda^{i,r}})_{i \geq 1}$ and by analogous estimates as in \eqref{N_player:regular:thm:estimates1}, we have
\begin{align*}
    \vert \J & (\lambda^{i,r},\theta_\infty) - \J_N(\lambda^{i,r},\boldsymbol{\lambda}^{-i,N}) \vert \\
    & \leq C (1 + \E[\theta_\infty^2])^\frac{1}{2} \left( \E\left[ \limsup_{ T \uparrow \infty} \frac{1}{T} \int_0^T \E\left[ \left \vert \big( \theta^{N,\boldsymbol{\lambda}^{-i,N}}_t \big)^{\beta} - \theta_\infty^{\beta} \right \vert^2 \big\vert \theta_\infty \right] dt \right]\right)^\frac{1}{2}
\end{align*}
and we conclude the proof by applying Lemmata \ref{N_player:lemma:ergodic_limit_q_power} and \ref{N_player:regular:lemma:convergence_empirical_measures} as before.
\end{proof}

\begin{remark}[On the integrability condition]\label{N_player:rmk:integrability}
It is worth to notice that the integrability condition \eqref{N_player:admissible_strategies:integrability} that defines $c$-admissible strategies can be weakened  at the price of more integrability requirements on $\theta_\infty$, $\xi$ and the diffusion $X^0$.
Indeed, let $q=\sfrac{\beta}{1-\alpha}$, $k=1+\lceil q \rceil$, and suppose that $2\delta-(k-1)\sigma>0$, $\E[\xi^k]<\infty$ and $\E[\theta_\infty^k]<\infty$, which imply that the estimates in point \ref{geometric_bm:reflected:lemma:integrability} of Lemma \ref{geometric_bm:reflected:lemma} holds up to the $k$-th moment.
Then, up to little modification of the proof, one could consider strategies $\nu \in \A_N$ so that
\[
\limsup_{T \uparrow \infty} \frac{1}{T}\E\left[\int_0^T X^\nu_s ds\right] \leq c.
\] 
\end{remark}

\subsection{Singular Recommendation}
We now look for a policy $\lambda^s$ of reflection type at a random barrier $a(\theta_\infty)$.
Let $\mathcal{G}^s$ be the set of correlated stationary strategies $(Z,\lambda^s,\theta_\infty)$ so that $Z=\theta_\infty$, $\theta_\infty \in L^2(\F_{0^{-}})$ independent of $\xi$ and $W$, and $\lambda^s$ is the control that reflects the process $X^{\lambda^s}$ upwards à la Skorohod at the random level
\begin{equation}\label{cce:singular:barrier_mean}
    a(\theta_\infty) = \frac{2\delta}{2\delta + \sigma^2}\theta_\infty.
\end{equation}

\begin{prop}\label{cce:singular:prop_consistency}
Let $(Z,\lambda^s,\theta_\infty)$ in $\mathcal{G}^s$.
Let $p^s_\infty(dx,\theta)$ be the stochastic kernel from $\R_+$ to $\mathcal{B}_{\R_+}$ defined by
\begin{equation}\label{cce:singular:limit_kernel}
p^s_\infty(dx,\theta)= p_{a(\theta)}(dx),
\end{equation}
where $p_a$ is the family of measures defined by \eqref{gemetric_bm:reflected:density}.
Then, the triple $(Z,\theta_\infty,\lambda^s)$ is a correlated stationary strategy so that the consistency condition \eqref{cce:cons_eq} is satisfied.
In particular, it holds $\mu^s_\infty(dx,d\theta) = \lim_{t \to \infty}\prob\circ(X^{\lambda^s}_t,\theta_\infty)^{-1} = p^s_\infty(dx,\theta)\rho(d\theta)$.
\end{prop}
\begin{proof}
Since the map $\R_+ \times \mathcal{B}_{\R_+} \ni (a,B) \mapsto p_a(B)$ defines a stochastic kernel from $\R_+$ to $\mathcal{B}_{\R_+}$, the kernel \eqref{cce:singular:limit_kernel} is well defined.
As in the proof of Proposition \ref{cce:regular:prop_corr_strategy}, we show that the joint law of $X^{\lambda^s}_t$ and $\theta_\infty$ converges weakly to $\mu^s_\infty$ as $t \to \infty$.
Indeed, conditionally to $\theta_\infty = \theta$, $X^{\lambda^s}_t$ satisfies the equation \eqref{mf:dynamics} with $\nu$ replaced by $\lambda^{s,\theta}$, where $\lambda^{s,\theta}$ reflects the process $X^{\lambda^s,\theta}$ upwards at the level $a(\theta)$, for $\rho$-a.e. $\theta \in \R_+$.
By Lemma \ref{geometric_bm:reflected:lemma}, the reflected process $X^{\lambda^s,\theta}$ admits $p_{a(\theta)}$ given by \eqref{gemetric_bm:reflected:density} as the unique invariant distribution.
This implies that the regular conditional probability of $X^{\lambda^s}_t$ with respect to $\theta_\infty$, that we denote by $p^s_t(dx,\theta)$, converges weakly to $p^s_\infty(dx,\theta)$ as $t \to \infty$ for $\rho$-a.e. $\theta > 0$.
Consistency condition \eqref{cce:cons_eq} follows from the definition of $a(\theta_\infty)$ and \eqref{geometric_bm:reflected:barrier_mean}.
\end{proof}

\begin{prop}\label{cce:singular:prop_optimality}
A correlated stationary strategy $(Z,\theta_\infty,\lambda^s)$ in $\mathcal{G}^s$ is a mean-field CCE for the ergodic MFG if and only if the following inequality is satisfied:
\begin{equation}\label{cce:singular:optimality_moments}
    c_{\beta} (\E[\theta_\infty^{\beta}])^{\frac{1}{1-\alpha}} + c_1\E[\theta_\infty] \leq \Tilde{c}_{\alpha+\beta}\E[\theta_\infty^{\alpha+\beta}],
\end{equation}
where $c_1$ and $c_{\beta}$ are given by \eqref{cce:reg_ctrl:constants} and $\Tilde{c}_{\alpha+\beta}$ is given by 
\begin{equation}\label{cce:singular:constants}
    \Tilde{c}_{\alpha+\beta} := \frac{2\delta + \sigma^2}{2\delta + \sigma^2(1-\alpha)}\left( \frac{2\delta}{2\delta + \sigma^2} \right)^{\alpha}.
\end{equation}
\end{prop}
\begin{proof}
As in the proof of Proposition \ref{cce:regular:prop_optimality}, it is enough to verify that the inequality $\J(\lambda^s,\theta_\infty) \geq \J(U^*,\theta_\infty)$ is equivalent to \eqref{cce:singular:optimality_moments}.
By \eqref{dev_player:payoff:moment_m_infty}, the payoff of the deviating player is equal to
\[
\J(U^*,\theta_\infty) = c_{\beta}( \E[\theta_\infty^{\beta}] )^{\frac{1}{1-\alpha}}.
\]
We turn our attention to $\J(\lambda^s,\theta_\infty)$.
We note that, conditionally to $\theta_\infty = \theta$, it holds
\begin{equation*}
\begin{aligned}
    & \lim_{T \uparrow \infty} \frac{1}{T} \E \left[ \int_0^T (X^{\lambda^s}_t)^\alpha \theta^\beta dt - q \lambda^s_T \Big \vert \theta_\infty = \theta \right]  = C(a(\theta),\theta^{\beta}), \quad \rho\text{-a.e. $\theta \in \R_+$}, 
\end{aligned}
\end{equation*}
where $C(a,p)$ is given by \eqref{alvarez:ergodic_functional}.
To see this, it is enough to recall that, by Proposition \ref{cce:singular:prop_consistency}, for $\rho$-a.e. $\theta$ in $\R_+$, we have $p^s_t(dx,\theta) \to p^s_\infty(dx,\theta)$ weakly as $t \to \infty$, with $p^s_\infty(dx,\theta)$ given by \eqref{cce:singular:limit_kernel}.
Since, conditionally to $\theta_\infty = \theta$, the control $\lambda^s$ is a reflection at the barrier $a(\theta)$, we apply \cite[Lemma 2.1]{alvarez2018stationary} to get the equality above.
By applying Lemma \ref{geometric_bm:reflected:lemma} with $a = a(\theta_\infty)$, and exploiting square-integrability of $\theta_\infty$, at any time $T > 0$ we can bound the left hand-side with $C(1 + \theta_\infty^2)$, for some positive constant $C$ independent of $\theta_\infty$.
Therefore, by dominated convergence theorem, we can exchange limit and expectation, to get
\begin{equation*}
\begin{aligned}
    \J & (\lambda^s,\theta_\infty) = \lim_{T \uparrow \infty} \frac{1}{T}\E\left[ \E \left[ \int_0^T (X^{\lambda^s}_t)^\alpha \theta_\infty^\beta dt - q \lambda^s_T \Big \vert \theta_\infty \right]  \right]  = \E[C(a(\theta_\infty),\theta_\infty^{\beta})] \\
    & = \frac{2\delta + \sigma^2}{2\delta + \sigma^2(1-\alpha)}\left( \frac{2\delta}{2\delta + \sigma^2} \right)^{\alpha}\E[\theta_\infty^{\alpha+\beta}] - q\delta\E[\theta_\infty].
\end{aligned}
\end{equation*}
By rearranging the terms, we have that $\J(U^*,\theta_\infty) \leq \J(\lambda,\theta_\infty)$ if and only if equation \eqref{cce:singular:optimality_moments} is satisfied.
\end{proof}

Finally, consider a correlated stationary strategy $(Z,\theta_\infty,\lambda^s)$ in the family $\mathcal{G}^s$ and define a correlated strategy profile for the $N$-player game starting from it: we take $Z = \theta_\infty$ as correlation device, and, for any $i \geq 1$, we consider the policy $\lambda^{i,s} = (\lambda^{i,s}_t)_{t \geq 0^{-}}$ according to which the state $X^{i,s}$ is reflected upward at the random barrier $a(\theta_\infty)$ given by \eqref{cce:singular:barrier_mean}.
As in the case of a regular recommendation, for each $i \geq 1$, the triple $(X^{i,s},\lambda^{i,s},\theta_\infty)$ has the same law as $(X,\lambda^s,\theta_\infty)$, and the processes $(X^{i,s})_{i \geq 1}$ are conditionally independent given $\theta_\infty$.

\begin{thm}[Approximation of CCEs - singular case]\label{N_player:singular:thm_approximation}
Let $(Z,\theta_\infty,\lambda^r)$ be a CCE in the class $\mathcal{G}^s$.
Let $\boldsymbol{\lambda}^N = (\lambda^{i,s})_{i=1}^N$, with $\lambda^{i,s}$ the policy according to which the state is reflected upward at the random barrier $a(\theta_\infty)$ given by \eqref{cce:singular:barrier_mean}.
Then, for any $c > 0$, the correlated strategy profile $(\theta_\infty, \boldsymbol{\lambda}^N)$ defines an $\varepsilon_N$-CCE for the $N$-player game within the set of strategies $\A_{N,c}$, with $\varepsilon_N \to 0$ as $N\to\infty$.
\end{thm}
We omit the proof since it is completely analogous to the proof of Theorem \ref{N_player:regular:thm_approximation}: it is enough to repeat the proof of Theorem \ref{N_player:regular:thm_approximation}, invoking Lemma \ref{geometric_bm:reflected:lemma} instead of Lemma \ref{N_player:regular:lemma_integrability}.

\subsection*{Nash Equilibrium for the Ergodic Mean-field Game}

Since the MFG considered here does not satisfy the assumptions of \cite[Theorem 4.4]{cao2023stationary}, we cannot directly deduce existence and uniqueness of NE for the MFG can not be applied.
Nevertheless, we have the following result:
\begin{prop}\label{mfg:thm:nash_eq}
If $\alpha + \beta \neq 1$, there exists a unique Nash equilibrium $(\nu^{*}, \theta^{*})$ of the ergodic MFG.
Moreover, the process $\nu^{*}$ reflects the state process at a barrier $ a^*$, and the pair $(a^*, {\theta}^*)$ is given by
\begin{equation}\label{mfg:nash_eq:expression}
\begin{aligned}
    & {a}^* = \left( \frac{2\delta + \sigma^2}{2\delta} \right)^{\frac{ \beta }{1-\alpha - \beta}}\tonde{ \frac{2\alpha}{q(2\delta + \sigma^2(1 - \alpha))}}^{\frac{1}{1-\alpha-\beta}}, \\
    & \theta^* = \left( \frac{2\delta + \sigma^2}{2\delta} \right)^{\frac{1-\alpha}{1-\alpha - \beta}} \tonde{ \frac{2\alpha}{q(2\delta + \sigma^2(1 - \alpha))}}^{\frac{1}{1-\alpha-\beta}}.
\end{aligned}
\end{equation}
If $\alpha + \beta = 1$ and the relation
\begin{equation}\label{mfg:condition:infinitely_many}
     1 + \frac{\sigma^2}{2\delta} = \left( \frac{q\delta}{\alpha} \right)^{\frac{1}{1-\alpha}} \left( 1 +(1-\alpha) \frac{\sigma^2}{2\delta} \right)^{\frac{1}{1-\alpha}}
\end{equation}
holds, then there exist infinitely many mean-field Nash equilibria given by the pair $(\nu^{a(\theta)},\theta)$, with $\nu$ a reflection at the level $a(\theta) = K\theta^{\beta}$ and $\theta>0$; otherwise, it does not exist any Nash equilibrium for the MFG.
\end{prop}
\begin{proof}
Fix $\theta > 0$.
By applying Lemma \ref{lemma:ergodic_control} with $p=\theta^\beta$ and $\lambda = 0$, the payoff functional $\J(\nu,\theta)$ is maximized by the strategy $\nu^{\theta}$ which reflect the process $X^{\nu^{\theta}}$ upwards à la Skorohod at the point $a(\theta)$ given by 
\[
\hat{a}(\theta) = \tonde{ \frac{2\alpha}{q(2\delta + \sigma^2(1 - \alpha))}}^{\frac{1}{1-\alpha}} \theta^{\frac{\beta}{1-\alpha}}.
\]
In view of \eqref{geometric_bm:reflected:barrier_mean}, in order to get the consistency condition \eqref{mfg:def:consistency} satisfied, we impose
\begin{equation}\label{mfg:solution:consistency}
\theta^* =   \frac{2\delta + \sigma^2}{2\delta} K (\theta^*)^{\frac{\beta}{1-\alpha}}.
\end{equation}
If $\alpha + \beta \neq 1$, this is equivalent to
\[
(\theta^*)^{\frac{1 -\alpha -\beta}{1-\alpha}} = \frac{2\delta + \sigma^2}{2\delta}K.
\]
If $\alpha + \beta \neq 1$, the function $g(\theta):=\theta^{\frac{1-\alpha -\beta}{1-\alpha}}$ is always non negative, strictly monotone and with image equal to $\R_+$, which implies that there exists a unique $\theta^*$ so that the above equality is verified; by direct computation, it can be verified that $\theta^*$ is given by \eqref{mfg:nash_eq:expression}.
Finally, if $\alpha + \beta = 1$, condition \eqref{mfg:solution:consistency} becomes
\[
\theta^* = \frac{2\delta + \sigma^2}{2\delta} K \theta^*.
\]
Thus, the MFG admits infinitely many Nash equilibria if $\frac{2\delta + \sigma^2}{2\delta} K = 1$, and none otherwise.
By explicit calculations, this is equivalent to 
\begin{align*}
    \frac{2\delta + \sigma^2}{2\delta} = \left( \frac{q\delta}{\alpha} \right)^{\frac{1}{1-\alpha}} \left( \frac{2\delta+\sigma^2}{2\delta} - \alpha\frac{\sigma^2}{2\delta} \right)^{\frac{1}{1-\alpha}}.
\end{align*}
By rearranging the terms, we get to \eqref{mfg:condition:infinitely_many}.
\end{proof}

\begin{remark}\label{mfg:rmk_assumptions}
Analogous considerations as in Remark \ref{mfc:rmk_assumptions} hold for the MFC problem as well.
On top of those considerations, we note that, when $\alpha + \beta = 1$ and condition \eqref{mfg:condition:infinitely_many} is not satisfied, we do not have existence of a Nash equilibrium for the ergodic MFG, while the optimality conditions for both classes $\mathcal{G}^r$ and $\mathcal{G}^s$ are still valid.
The ultimate reason is that the procedure outlined in Section \ref{sec:competitive} does not involve the usual two steps scheme used to compute mean-field NEs: first, optimize with a fixed flow of moments and, second, perform a fixed point argument to determine the flow. Actually, we first impose the consistency condition and then we restate the optimality condition.
\end{remark}

We notice that the pair $(\nu^*,\theta^*)$ is a correlated stationary strategy in $\mathcal{G}^s$ with deterministic correlation device.
In particular, it satisfies the optimality condition \eqref{cce:singular:optimality_moments}.
As a consequence of Theorem \ref{N_player:singular:thm_approximation}, we also deduce that every NE for the ergodic MFG induces a sequence of approximate Nash equilibria with vanishing error in the $N$-player game:
\begin{corollary}
Let $(\nu^*,\theta^*)$ be a Nash equilibrium for the MFG, as given by Proposition \ref{mfg:thm:nash_eq}.
For any $i \geq 1$, let $\nu^{i,*}$ be the policy according to which the state is reflected upward at the random barrier $a^*$ given by \eqref{mfg:nash_eq:expression}.
Then, for any $c > 0$, the open-loop strategy profile $\boldsymbol{\nu}^{*,N} = (\nu^{i,*})_{i = 1}^N$ defines an $\varepsilon_N$-NE for the $N$-player game within the set of strategies $\A_{N,c}$, with $\varepsilon_N \to 0$ as $N\to\infty$.
\end{corollary}
\begin{proof}
It is enough to notice that, when starting from the NE $(\nu^*,\theta^*)$ the recommendation $\boldsymbol{\lambda}^N$ defined in Theorem \ref{N_player:singular:thm_approximation} is actually an open-loop strategy profile for the $N$-player game.
\end{proof}

\section{Numerical Illustrations}\label{sec:numerics}
In this Section, we numerically illustrate our previous findings.
In particular, we exhibit possible choices of distribution of $\theta_\infty$ so that the correlated stationary strategies $(Z,\theta_\infty,\lambda^r)$ in $\mathcal{G}^r$ and $(Z,\theta_\infty,\lambda^s)$ in $\mathcal{G}^s$ are mean-field CCEs, i.e., according to Propositions \ref{cce:regular:prop_optimality} and \ref{cce:singular:prop_optimality}, the inequalities \eqref{cce:regular:optimality:condition_moments} and \eqref{cce:singular:optimality_moments} respectively are verified.

\smallskip
For the sake of illustrations, throughout the whole section, we assume $\theta_\infty$ is distributed as a Gamma with $u>0$ and scale parameter $v > 0$.
Then, for any $k \geq 0$ the $k$-th moment of $\theta_\infty \sim \Gamma(u,v)$ is given by
\[
\E[\theta_\infty^k] = \frac{1}{\Gamma(u) v^u}\int_0^\infty x^{k}x^{u-1}e^{-\frac{x}{v}}dx = \frac{\Gamma(u + k)}{\Gamma(u)} v^{k}.
\]
By assuming Gamma distribution on $\theta_\infty$, the optimality conditions \eqref{cce:regular:optimality:condition_moments} and \eqref{cce:singular:optimality_moments} for regular and singular recommendation become, respectively,
\begin{align}
        & c_{\beta} \left( \frac{\Gamma(\beta+u)}{\Gamma(u)} \right)^{\frac{1}{1-\alpha}} v^{\frac{\beta}{1-\alpha}} \leq c_{\alpha+\beta}  \frac{\Gamma(\alpha + \beta+u)}{\Gamma(u)} v^{\alpha + \beta} - c_1 u v,  \label{cce:regular:optimality_gamma} \\
        & c_{\beta} \left( \frac{\Gamma(\beta+u)}{\Gamma(u)} \right)^{\frac{1}{1-\alpha}} v^{\frac{\beta}{1-\alpha}} \leq \Tilde{c}_{\alpha+\beta}  \frac{\Gamma(\alpha + \beta+u)}{\Gamma(u)} v^{\alpha + \beta} - c_1 u v,  \label{cce:singular:optimality_gamma} 
\end{align}

In the case where a unique mean‐field Nash equilibrium exists (i.e., when $\alpha + \beta \neq 1$, see Theorem \ref{mfg:thm:nash_eq}), we also aim at selecting mean‐field CCEs that yield a higher payoff than the Nash equilibrium.
To implement this refinement, we then pair the existence conditions for the regular recommendation \eqref{cce:regular:optimality_gamma} and for the singular recommendation \eqref{cce:singular:optimality_gamma} with, respectively,
\begin{align}
    &  c_{\alpha+\beta} \frac{\Gamma(\alpha + \beta+u)}{\Gamma(u)} v^{\alpha + \beta} - c_1 u v \geq \Tilde{c}_{\alpha+\beta}(\theta^*)^{\alpha + \beta} - c_1 \theta^*, \label{cce:regular:out_performance_gamma} \\
    &  \Tilde{c}_{\alpha+\beta} \frac{\Gamma(\alpha + \beta+u)}{\Gamma(u)} v^{\alpha + \beta} - c_1 u v \geq \Tilde{c}_{\alpha+\beta}(\theta^*)^{\alpha + \beta} - c_1 \theta^*. \label{cce:singular:out_performance_gamma}
\end{align}
Figure \ref{fig:cce_region} shows that there exist infinitely many mean-field CCEs both in $\mathcal{G}^r$ and $\mathcal{G}^s$ which yield an higher reward than the Nash equilibrium $(\nu^*,\theta^*)$.

\begin{figure}[!ht]
    \centering
    \includegraphics[trim= 3cm 0.5cm 3cm 0.5cm, clip, width = .55\textwidth,angle=-90]{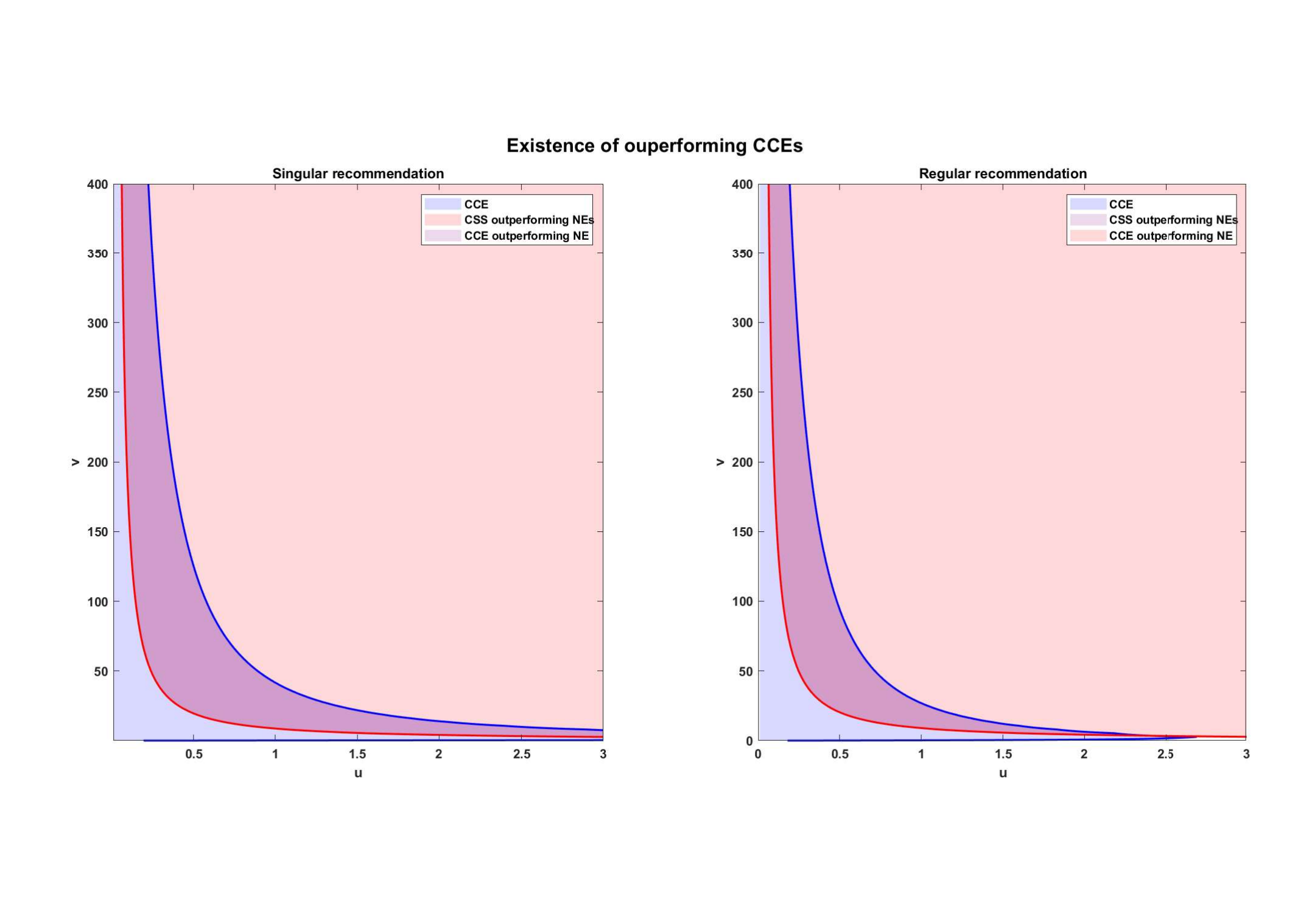}
    \caption{Values of the parameters $(u,v) \in \R_+^2$ so that the $(Z,\lambda^s,\theta_\infty) \in \mathcal{G}^s$ (on the left) and $(Z,\lambda^r,\theta_\infty) \in \mathcal{G}^r$ (on the right), $\theta_\infty \sim \Gamma(u,v)$ is a mean-field CCE outperforming the NE. Here, $\delta = 0.1$, $\sigma = 0.2$, $q = 2$, $\alpha = 0.3$ and $\beta = 0.5$.
    Notice that Assumption \ref{assumption:dissipativity} is satisfied.}
    \label{fig:cce_region}
\end{figure}

\smallskip
For the same choice of parameters as in Figure \ref{fig:cce_region}, Figure \ref{fig:cce_reward} shows the reward associated to those mean-field CCEs in $\mathcal{G}^r$ and $\mathcal{G}^s$ that yield an higher reward than the Nash equilibrium $(\nu^*,\theta^*)$.
The improvement on the Nash equilibrium is $\approx 17\%$ of the reward yielded by the mean-field control solution $\hat{\nu}$ in the singular case, and $\approx 12\%$ in the regular case.

\begin{figure}[!ht]
    \centering
    \includegraphics[trim= 4.5cm 1.5cm 4.5cm 1.5cm, clip, width =0.5\textwidth,angle=-90]{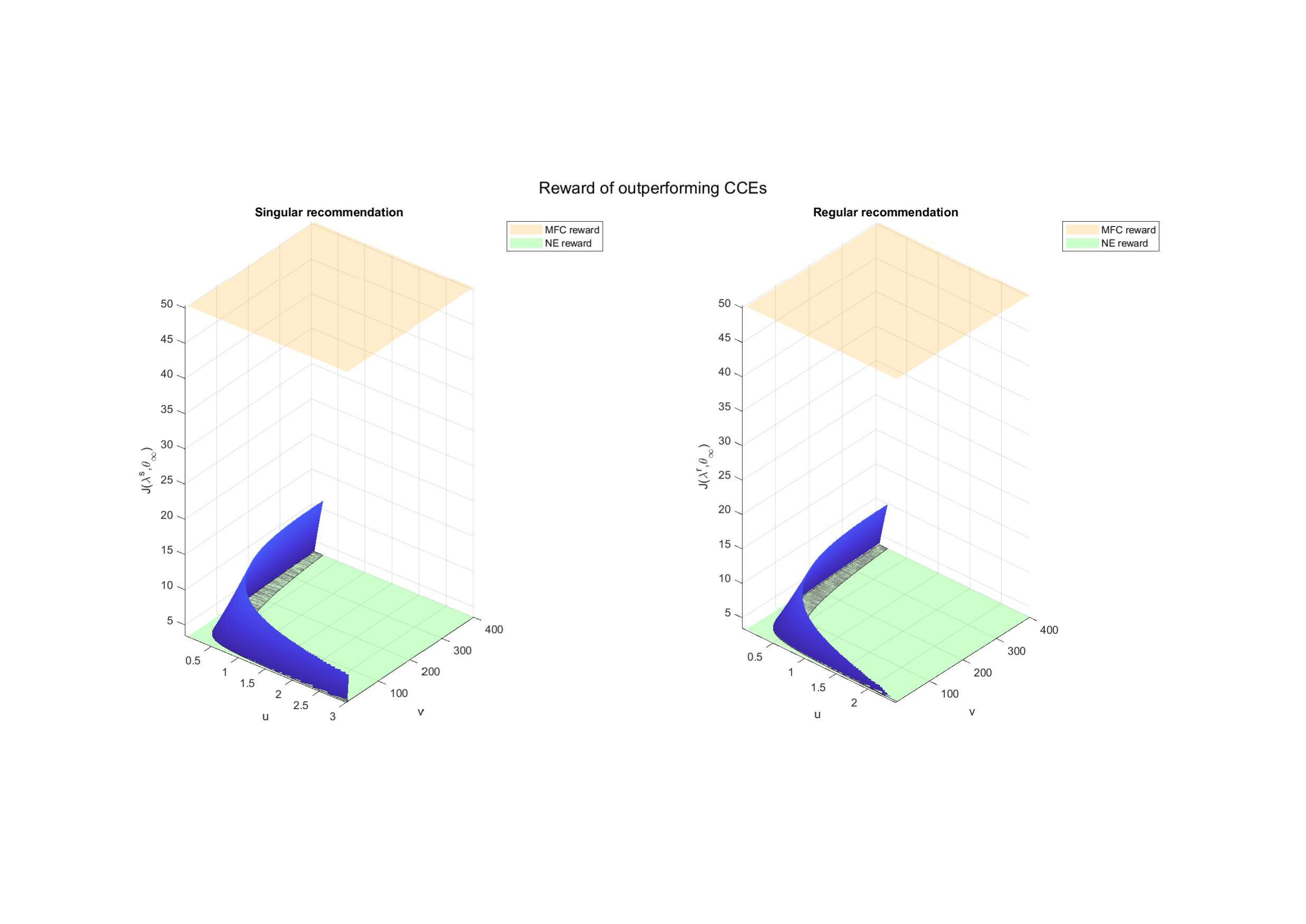}
    \caption{Reward associated to mean-field CCEs $(Z,\lambda^s,\theta_\infty) \in \mathcal{G}^s$ (on the left) and $(Z,\lambda^r,\theta_\infty) \in \mathcal{G}^r$ (on the right) which outperform the reward of the mean-field NE, when $\theta_\infty \sim \Gamma(u,v)$, $(u,v) \in \R_+^2$. Here, $\delta = 0.1$, $\sigma = 0.2$, $q = 2$, $\alpha = 0.3$ and $\beta = 0.5$.
    }
    \label{fig:cce_reward}
\end{figure}

\smallskip
In the case in which there exists a MFC solution (i.e. $\alpha+\beta < 1$ or $\alpha+\beta = 1$ and condition \eqref{eq:mfc:condition_zero_infinity} is verified, see Theorem \ref{mfc:thm:optimal_control}), we are also interested in finding the highest reward yielded by CCEs in the classes $\mathcal{G}^r$ and $\mathcal{G}^s$, for $\theta_\infty \sim \Gamma(u,v)$, i.e., respectively,
\begin{align*}
    & \max_{(u,v) \in \R^2_+} \Big( c_{\alpha+\beta} \frac{\Gamma(\alpha + \beta+u)}{\Gamma(u)} v^{\alpha + \beta} - c_1 u v \Big), \\
    & \max_{(u,v) \in \R^2_+} \Big( \Tilde{c}_{\alpha+\beta} \frac{\Gamma(\alpha + \beta+u)}{\Gamma(u)} v^{\alpha + \beta} - c_1 u v \Big).
\end{align*}
Figure \ref{fig:cce_sigma} shows the maximum payoff in the classes $\mathcal{G}^s$ and $\mathcal{G}^r$, for $\theta_\infty \sim \Gamma(u,v)$, as a function of the volatility parameter $0 < \sigma < \sqrt{2\delta}$, where the last restriction ensures that Assumption \ref{assumption:dissipativity} is satisfied.
We notice that, for any choice of $\sigma$, there always exists a mean-field CCE in $\mathcal{G}^s$ that outperforms the mean-field NE. On the contrary, we notice that, when $\sigma$ is small enough, there does not exist any mean-field CCE in the class $\mathcal{G}^r$ with $\theta_\infty \sim \Gamma(u,v)$.
Moreover, when $\sigma$ is large enough, the best performing mean-field CCE in $\mathcal{G}^r$ does not outperform the mean-field NE.

Figure \ref{fig:cce_beta} shows the maximum payoff in the classes $\mathcal{G}^s$ and $\mathcal{G}^r$, for $\theta_\infty \sim \Gamma(u,v)$, as a function of interaction strength $0 \leq \beta \leq 1-\alpha$, for fixed $\alpha$.
On the right picture, the parameters are chosen so that condition \eqref{eq:mfc:condition_zero_infinity} is satisfied, so that, in the case $\alpha + \beta = 1$, the null control is optimal for the MFC problem.
On the contrary, on the left picture the parameters do not satisfy \eqref{eq:mfc:condition_zero_infinity}, so that, in the case $\alpha + \beta = 1$, the problem is ill-posed.
In both cases, for $\beta$ large enough, the best performing CCEs in the classes $\mathcal{G}^r$ and $\mathcal{G}^s$, for $\theta_\infty \sim \Gamma(u,v)$, yield an higher payoff than the mean-field NE.

\begin{figure}[!ht]
    \centering
    \includegraphics[trim= 1.5cm 8.35cm 1.5cm 9cm, clip, width =0.63\textwidth]{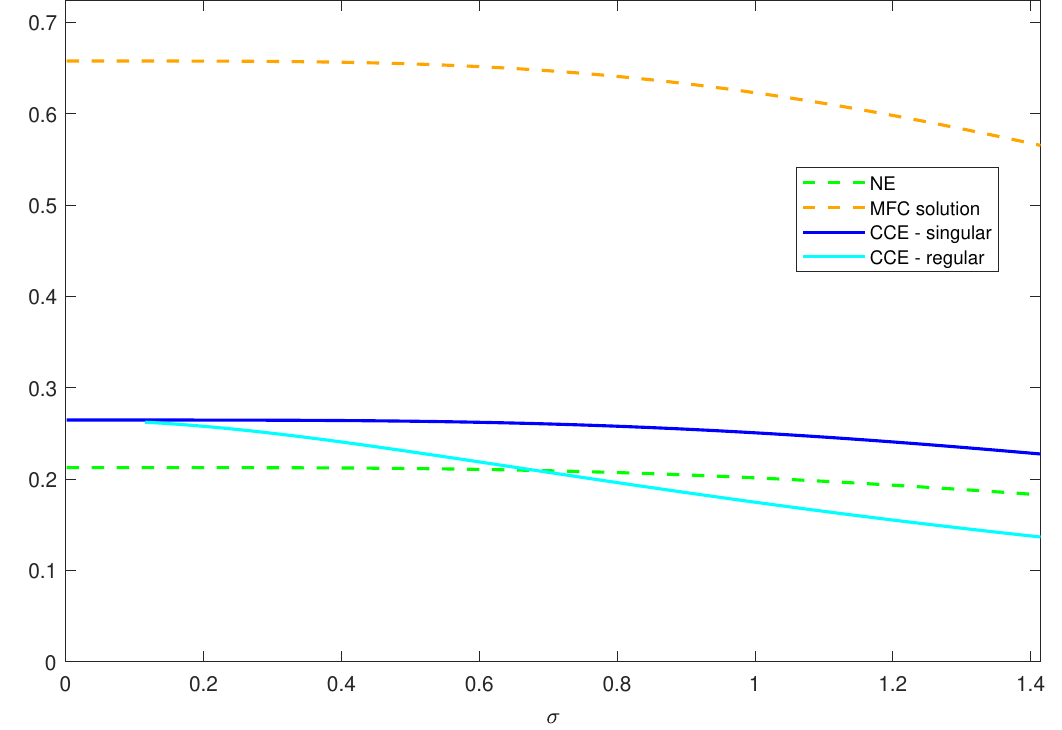}
    \caption{Reward associated to the best performing CCE $(Z,\lambda^s,\theta_\infty) \in \mathcal{G}^s$ (in blue) and $(Z,\lambda^r,\theta_\infty) \in \mathcal{G}^r$ (in cyan), when $\theta_\infty \sim \Gamma(u,v)$, $(u,v) \in \R_+^2$, compared with MFC solution (dashed orange line) and mean-field NE (dashed green line).
    Here, $\delta = 1$, $q=0.5$, $\alpha = 0.3$ and $\beta = 0.4$.
    $\sigma$ varies from $0$ to $\sqrt{2\delta}$.}
    \label{fig:cce_sigma}
\end{figure}

\begin{figure}[!ht]
    \centering
    \includegraphics[trim= 1.5cm 5cm 1.5cm 5cm, clip, width =0.8\textwidth]{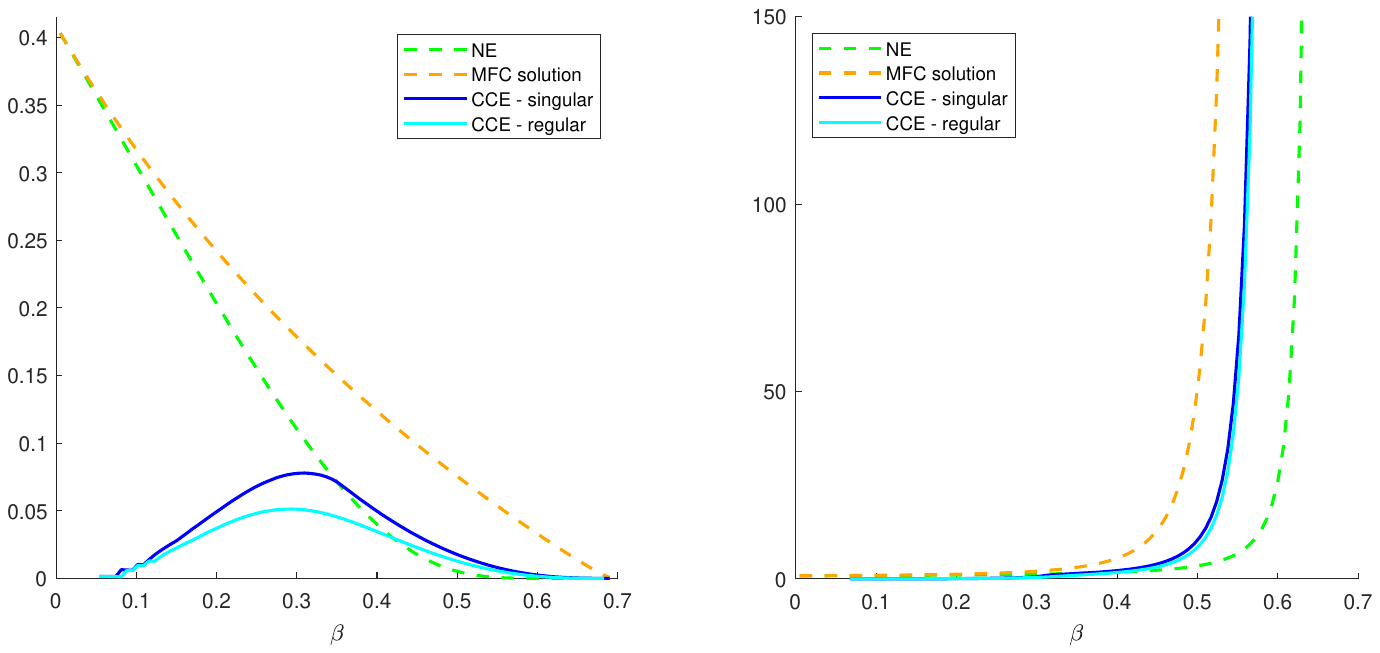}
    \caption{Reward associated to the best performing CCE $(Z,\lambda^s,\theta_\infty) \in \mathcal{G}^s$ (in blue) and $(Z,\lambda^r,\theta_\infty) \in \mathcal{G}^r$ (in cyan), for $\theta_\infty \sim \Gamma(u,v)$, $(u,v) \in \R_+^2$, compared with MFC solution (dashed orange line) and mean-field NE (dashed green line).
    Here, $\alpha=0.3$ and $0 \leq \beta \leq 1-\alpha$.
    On the left, $\delta=1$, $q_u=1$, $\sigma = 1$, so that condition \eqref{eq:mfc:condition_zero_infinity} is satisfied.
    On the right, $\delta=0.1$, $q_u=2$, $\sigma = 0.2$, so that condition \eqref{eq:mfc:condition_zero_infinity} is not satisfied.
    }
    \label{fig:cce_beta}
\end{figure}

We notice that the payoff associated to the Nash equilibrium for the ergodic MFG is strictly less than the reward given by the solution of the MFC problem.
This can be directly deduced from the fact that the the stationary mean $\hat{\theta}$ associated to the MFC solution $\hat{\nu}$ is the unique maximizer of the function $f(\theta)$ given by \eqref{mfc:mean_dependence_function}, which can be equivalently expressed as $f(\theta)=\Tilde{c}_{\alpha+\beta}\theta^{\alpha + \beta} - c_1\theta$.
Since by Proposition \ref{mfg:thm:nash_eq} the value of the ergodic MFG at the Nash equilibrium $(\nu^*,\theta^*)$ can be expressed as $f(\theta^*)$, we deduce
\begin{equation}\label{eq:price_of_anarchy}
    \J(\nu^*,\theta^*) = f(\theta^*) < f(\hat{\theta}) =  \J(\hat{\nu},\E[X^{\hat{\nu}}_\infty]).
\end{equation}
While we limit to empirically observe this phenomenon, we point out that it is widely expected and that it is coherent with the findings of \cite{campi2023LQ} for linear-quadratic mean-field games.

\smallskip
Finally, we consider the case where either either there does not exist any mean-field NE or there exist infinitely many (i.e. $\alpha+\beta = 1$, depending on whether the relation \eqref{mfg:condition:infinitely_many} is satisfied, see Theorem \ref{mfg:thm:nash_eq}).
As noticed in Remark \ref{mfg:rmk_assumptions}, the optimality inequalities \eqref{cce:regular:optimality:condition_moments} and \eqref{cce:singular:optimality_moments} are still valid.
By imposing $\beta = 1 - \alpha$, and noticing that the parameter $v$ is not anymore relevant in the inequalities, equations \eqref{cce:regular:optimality_gamma} and \eqref{cce:singular:optimality_gamma} can be written in terms of $u$ only:
\begin{align}
    & c_{\beta} \left( \frac{\Gamma(1- \alpha +u)}{\Gamma(u)} \right)^{\frac{1}{1-\alpha}} \leq \left( c_{\alpha+\beta}   \Gamma(u) - c_1\right) u ,  \label{cce:regular:optimality_gamma_1} \\
    & c_{\beta} \left( \frac{\Gamma(1-\alpha+u)}{\Gamma(u)} \right)^{\frac{1}{1-\alpha}} \leq \left( \Tilde{c}_{\alpha+\beta} \Gamma(u) - c_1\right) u. \label{cce:singular:optimality_gamma_1}
\end{align}
We observe that there exist  maximal values $u^*_r$ and $u^*_s$, depending on $\alpha$, so that the inequalities \eqref{cce:regular:optimality_gamma_1} and \eqref{cce:singular:optimality_gamma_1} are verified by any $0 < u \leq u^*_r$ and $0 < u \leq u^*_s$, respectively.
Figure \ref{fig:cce_alpha_beta_1} plots such maximal values $u^*_r$ and $u^*_s$ as functions of $\alpha \in (0,1)$.
Therefore, we observe existence of mean-field CCEs even in the case in which there does not exist any mean-field NE.
We notice that, when considering correlated stationary strategies $(Z,\lambda^s,\theta_\infty)$, there exist a value $\bar{\alpha}$ so that that the inequality \eqref{cce:regular:optimality_gamma_1} is verified for any $u > 0$, i.e. $u^* = \infty$.
It can be numerically shown that such value $\bar{\alpha}$ is the unique solution of \eqref{mfg:condition:infinitely_many}, for fixed $\delta$, $\sigma$ and $q$.
We can explain this phenomenon as follows: for $\alpha = \bar{\alpha}$, by Theorem \ref{mfg:thm:nash_eq}, for any $\theta>0$ the pair $(\nu^{a(\theta)},\theta)$ is a mean-field NE, where $\nu^{a(\theta)}$ is the policy that reflects the process $X^{\nu^{a(\theta)}}$ upwards at the level $a(\theta) = \frac{2\delta}{2\delta + \sigma^2} \theta$.
Therefore, any correlated stationary strategy $(Z,\lambda^s,\theta_\infty) \in \mathcal{G}^s$ is just a randomization, or a mixture, of mean-field NE, since $\lambda^s$ reflects the process $X^{\lambda^s}$ at the same barrier $a(\theta_\infty) = \frac{2\delta}{2\delta + \sigma^2} \theta_\infty$.
To put in other terms, the pair $(a(\theta_\infty),\theta_\infty)$ is supported on the set of mean-field NEs.
This implies that the optimality condition is satisfied by any $\theta_\infty$ so that the optimality inequality \eqref{cce:singular:optimality_moments} holds true, and so by any $(Z,\lambda^s,\theta_\infty) \in \mathcal{G}^s$.

\begin{figure}[!ht]
    \centering
    \includegraphics[trim= 4.5cm 1.5cm 4.5cm 1.5cm, clip, width =0.5\textwidth,angle=-90]{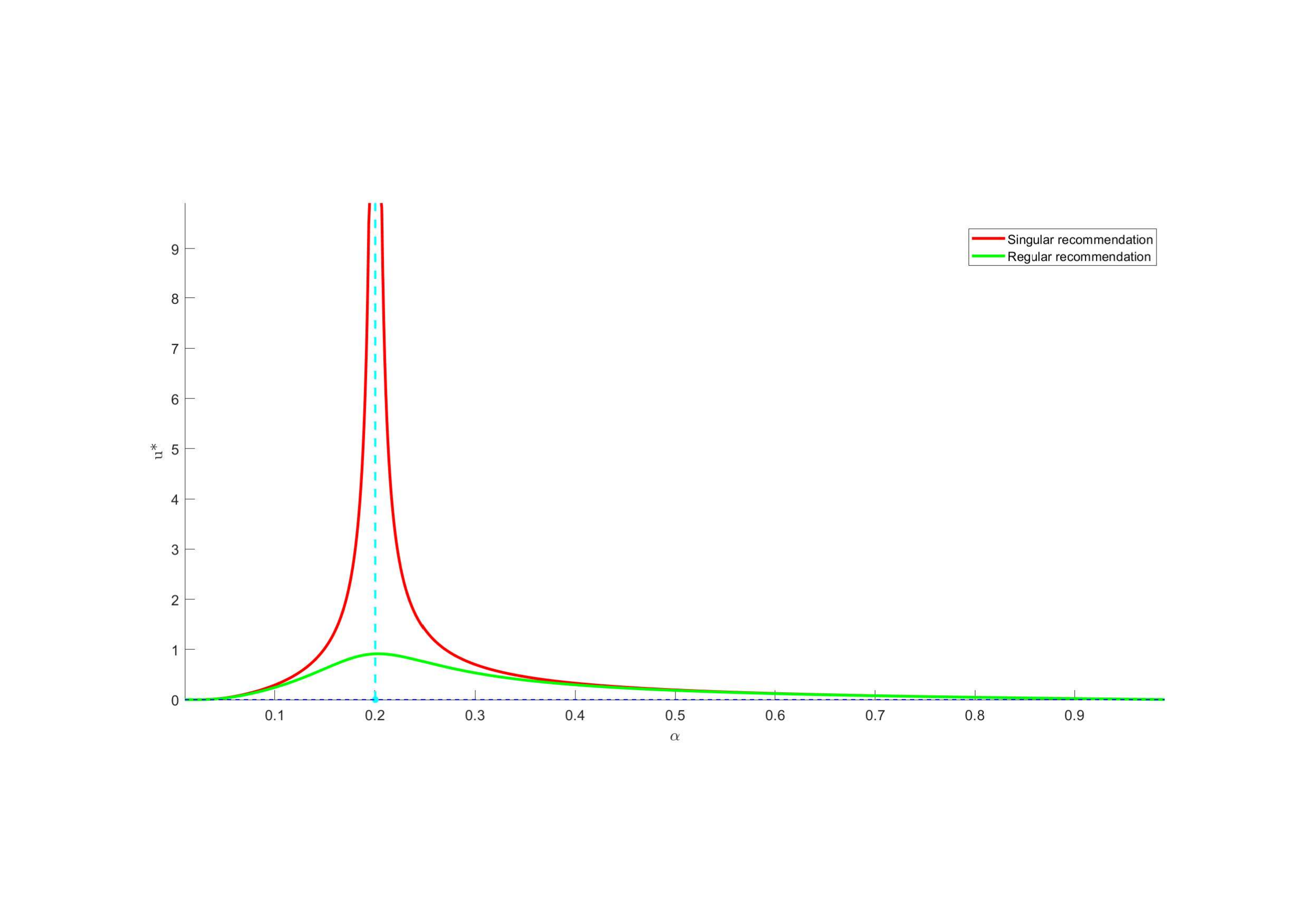}
    \caption{Value of $u^*$ as $\alpha$ varies in $[0,1]$, both for regular and singular recommendations.
    The blue dashed line is located at the value of $\alpha=\bar{\alpha}$ which satisfies \eqref{mfg:condition:infinitely_many}. Here, $\delta = 0.1$, $\sigma = 0.2$, $q = 2$.}
    \label{fig:cce_alpha_beta_1}
\end{figure}

\newpage

\addcontentsline{toc}{section}{Appendix}
\section*{Appendix}\label{sec:appendix}

\subsection*{Control-theoretic Results}

\begin{proof}[Proof of Lemma \ref{geometric_bm:reflected:lemma}]
Point \ref{geometric_bm:reflected:lemma:finite_moments} follows from direct calculations.
As for point \ref{geometric_bm:reflected:lemma:reflection}, the solution of the Skorohod problem follows from standard arguments (see, e.g. \cite[Proposition 3.6.16]{karatzas_shreve}).
As for ergodicity, note that the derivative of the speed measure of the process $X^{\nu^a}$ is given by $m'(x) = \frac{2}{\sigma^2} x^{ - \frac{2\delta}{\sigma^2} }\1_{[a,\infty)}(x)$, which is integrable over $[0,\infty)$ for any $a>0$.
By \cite[Paragraph 36]{borodin2002handbook}, the process $X^{\nu^a}$ is ergodic and admits $\frac{1}{m([a,\infty))}m'(x) dx = p_a(dx)$ given by \eqref{gemetric_bm:reflected:density} as the unique invariant distribution.

As for \ref{geometric_bm:reflected:lemma:integrability}, set $L=1+a$ and observe that, since $a < L$ and the control $\nu^a$ never acts when the process lies in the region $\{x: \; x > a\}$, it holds $ \supp( d\nu^a ) \cap \{X^{\nu^a} \geq L \} = \emptyset$, $\prob$-a.s.
As for \ref{geometric_bm:reflected:lemma:integrability}, set $L=1+a$ and observe that, since $a < L$ and the control $\nu^a$ never acts when the process lies in the region $\{x: \; x > a\}$, it holds $ \supp( d\nu^a ) \cap \{X^{\nu^a} \geq L \} = \emptyset$, $\prob$-a.s.
Reasoning similarly to \cite[Lemma 2]{cao2023stationary}, let $(\tau^i_k, \bar{\tau}^i_k)_{k\geq1}$ be a sequence of stopping times such that $0\leq \tau_1 \leq \bar{\tau}_1 \leq \tau_2 \leq \bar{\tau}_2 \leq \dots $, $\prob$-a.s., and such that $\{(t,\omega) \in [0,T]\times \Omega: \; X^{\nu^a}_t \geq L \} = \bigcup_{k \geq 1} [\tau^i_k,\bar{\tau}^i_k]$.
Observe that $(\tau_k,\bar{\tau}_k)$ can be defined recursively by setting $\tau_0 = \bar{\tau}_0 = 0$ and
\begin{align*}
    &\tau_k = \inf\{ t > \bar{\tau}_{k-1} : \; X^{\nu^a}_t \geq L \}, \quad \bar{\tau}_k = \inf\{ t > \tau_k : \; X^{\nu^a}_t \leq L \},
\end{align*}
so that, in particular, it holds $X^{\nu^a}_{\tau_k}=\xi$ if $\tau_k=0$ or $X^{\nu}_{\tau_k}=L$ if $\tau_k>0$.
Set $Y^a_t = (X^{\nu^a}_t)^2 - L^2$, so that, since $X^{\nu^a}$ is always non-negative, it holds $\{ X^{\nu^a} \geq L \} = \{ Y^a \geq 0 \}$.
By using the sequence $(\tau_k,\bar{\tau}_k)_{k \geq 1}$ and the process $Y^a$, we get
\begin{multline*}
    \E [ (X^{\nu^a}_t)^2 ] = \E \Big[ (X^{\nu^a}_t)^2 \mathds{1}_{ \{ X^{\nu^a}_t \leq L \} } + (X^{\nu^a}_t)^2 \mathds{1}_{ \{ X^{\nu^a}_t \geq L \} } \Big] \leq L ^2 +  \E \bigg[ \sum_{k = 1}^\infty \mathds{1}_{(\tau_k,\bar{\tau}_k ) }(t) (X^{\nu^a}_t)^2 \bigg] \\
    \leq L ^2 +  \sum_{k = 1}^\infty \E \bigg[ \mathds{1}_{ (\tau_k,\bar{\tau}_k ) }(t) ( Y^a_t + L^2) \bigg] \leq 2 L ^2 +  \sum_{k = 1}^\infty \E \bigg[ \mathds{1}_{ \{ \tau_k < t \} }  Y^a_{t \wedge \bar{\tau}_k }  \bigg],
\end{multline*}
by using the identity
\[
\mathds{1}_{ (\tau_k,\bar{\tau}_k ) }(t) Y^a_t = \mathds{1}_{ \{ \tau_k < t \} }(  \mathds{1}_{ \{ \bar{\tau}_k > t \} } Y^a_t + 0 \cdot \mathds{1}_{ \{ \bar{\tau}_k \leq  t \}} )= \mathds{1}_{ \{ \tau_k < t \} } Y^a_{t \wedge \bar{\tau}_k }
\]
since $Y^a_{\bar{\tau_k}} = 0$ for any $k \geq 1$.
By applying Ito's formula on each time interval $[\tau_k, \bar{\tau}_k]$ and exploiting  $2\delta - \sigma^2 > 0$ by Assumption \ref{assumption:dissipativity}, we deduce that
\[
Y^a_{t \wedge \bar{\tau}_k} \leq Y^a_{\tau_k} + 2\sigma \int_{ \tau_k }^{ t \wedge \bar{\tau}_k }  \left(  Y^a_s + L^2 \right) d W_s, \quad k \geq 1.
\]
Plugging this expression in the last inequality above, we get 
\begin{multline*}
    \E [ (X^{\nu^a}_t)^2 ] \leq 2 L ^2 +  \sum_{k = 1}^\infty \E \bigg[ \mathds{1}_{ \{ \tau_k < t \} } \left( Y^a_{\tau_k}  + 2\sigma \int_{\tau_k}^{ t \wedge \bar{\tau}_k }  \left( Y^a_s + L^2 \right) d W_s  \right) \bigg] \\
    \leq 2 L ^2 +  \sum_{k = 1}^\infty \E \bigg[ \mathds{1}_{ \{ \tau_k < t \} } Y^a_{\tau_k}  \bigg]  = 2 L ^2 +  \E \bigg[ \sum_{k = 1}^\infty  \mathds{1}_{ \{ \tau_k < t \} } Y^a_{\tau_k}  \bigg] ,
\end{multline*}
where in the second equality we applied tower property with respect to $\F_{\tau_k}$ to handle the stochastic integral, and we exchanged series and integral thanks to the non-negativity of $\mathds{1}_{ \{ \tau_k < t \} } Y^a_{\tau_k}$.
We conclude by noticing that one has either $Y^a_{\tau_k} = \xi^2 - L^2$ if $\tau_k = 0$ or $(Y^a_{\tau_k})^2 = 0$ $\prob$-a.s if $\tau_k > 0$, for any $k \geq 1$, and that one can have $\tau^k = 0$ only for $k = 1$.
Thus, we conclude
\[
\E [ (X^{\nu^a}_t)^2 ] \leq C( L ^2 + \E[\xi^2]) \leq C(1+a^2).
\]
By definition of $L$ and Assumption \ref{assumption:dissipativity}, the estimate follows.
Finally, we have, for any $T > 0$
\[
\nu^a_T = X^{\nu^a}_T - \xi +\delta \int_0^T X^{\nu^a}_s ds -\int_0^T \sigma X^{\nu^a}_s dW_s.
\]
By taking the expectation, applying Jensen inequality and taking advantage of It\^{o} isometry, we find
\begin{align*}
& \frac{1}{T^2}\E[\vert \nu^a_T \vert^2 ] \leq \frac{c}{T^2}\left( 1 + \E[\vert X^{\nu^a}_T \vert^2] + \E\left[ \left( \int_0^T  X^{\nu^a}_s  ds \right)^2 \right] +  \E\left[ \left( \int_0^T X^{\nu^a}_s dW_s \right)^2 \right] \right) \\
& \leq \frac{c}{T^2}\left( 1 + (T + T^2)\sup_{T \geq 0}\E[(X^{\nu^a}_T)^2]  \right) \leq c(1 + a^2),
\end{align*}
where last inequality follows from previous estimate.
\end{proof}

\begin{proof}[Proof of Lemma \ref{lemma:ergodic_control}.]
We reason as in the proof of \cite[Theorem 2]{cao2023stationary}.
Recall from \eqref{control:auxiliary:running_cost} the definition of the function $g(x,p,\lambda)$.
Let $\mathcal{T}$ be the set of $\mathbb{F}$-stopping times.
Consider the auxiliary optimal stopping problem
\begin{equation}\label{eq optimal stopping problem}
    u(x,p,\lambda):= \inf_{\tau \in \mathcal{T}} \E_x \bigg[ \int_0^\tau e^{-\delta t } g_x(\hat{X}_t, p,\lambda) dt  + q e^{-\delta \tau } \bigg],
\end{equation}
where $\hat{X}=(\hat{X}_t)_{t \geq 0}$ is defined by
\[
d\hat{X}_t=(-\delta + \sigma^2 )\hat{X}_t dt + \sigma \hat{X}_tdW_t,
\]
and $\E_x[\cdot]$ denotes $\E[\cdot \vert \widehat{X}_{0} = x]$, $x \in \R_+$.
Let $\hat{\phi}_0$ is the non-increasing fundamental solution of
\[
\frac{1}{2}\sigma^2 x^2 \phi_{xx}(x) + (-\delta + \sigma^2) x \phi_x(x) -\delta \phi(x) = 0,
\]
and let $\hat{m}'$ be the density of the speed measure $\hat{m}$ of the process $\hat{X}$.
By well-known results (see, e.g., \cite[Theorem 5]{alvarez2001reward}), it can be shown that, if there exists a unique $a^*=a^*(p,\lambda) > 0$ solution to
\begin{equation}\label{lemma:ergodic_control:optimal_stopping}
\int_{a^*}^{+\infty} \hat{\phi}_0(y)\left( \alpha p y^{\alpha - 1} - (q\delta - \lambda) \right) \hat{m}'(y) dy = 0,
\end{equation}
then the value function $u(\cdot,p,\lambda)$ is $C^1(\mathbb{R}_+)$ with $u_{xx}(\cdot,p,\lambda) \in L^{\infty}_{\text{loc}}(\R_+)$, and that the optimal stopping time is given by $\hat{\tau}(x,p,\lambda):=\inf \{t\geq 0 \, | \, \hat{X}_t \leq a^*(p,\lambda) \} $. 
By explicit calculations, we have $\hat{\phi}_0(y) = y^{-1}$ and $\hat{m}'(y) = \frac{2}{\sigma^2} y^{ - \frac{2\delta}{\sigma^2} }$, so that \eqref{lemma:ergodic_control:optimal_stopping} becomes 
\begin{equation}\label{lemma:ergodic_control:optimal_barrier}
\int_{a^*}^{+\infty} \left( \alpha p y^{\alpha - 1} - (q\delta - \lambda) \right)y^{-\frac{2\delta}{\sigma^2} - 1} dy = 0.
\end{equation}
Since, by assumption, $q\delta - \lambda > 0$, there exists a unique solution $a^*(p,\lambda)$ given by \eqref{lemma:ergodic_control:barrier}.
Moreover, relying on optimality of the $\hat{\tau}$, on the dynamic programming principle and on the aforementioned regularity of $u(\cdot,p,\lambda)$, direct computations show that $u(\cdot,p,\lambda)$ solves
\begin{equation*}
\left\{ \, \begin{aligned}
    \frac{\sigma^2}{2}u_{xx} +(-\delta+\sigma^2) x u_x -\delta u +g_x(x,p,\lambda) & = 0, &&  x > a^*(p,\lambda), \\
    u & = q, && x \leq a^*(p,\lambda).
\end{aligned} \right.
\end{equation*}
Let now $(v(\cdot,p,\lambda),\gamma^*(p,\lambda))$ defined by
\begin{equation}
    v(x,p,\lambda) = \int_{a^*(p,\lambda)}^x u(x',p,\lambda)dx', \quad \gamma^* = g(a^*(p,\lambda),p,\lambda) +q\delta a^*(p,\lambda).
\end{equation}
Since $u(\cdot,p,\lambda) \in C^1(\R_+)$, $v \in C^2(\R_+)$.
By explicit computations, one can verify that the pair $(v(\cdot,p,\lambda),\gamma^*(p,\lambda))$ solves the following PDE 
\begin{equation*}
\left\{ \, \begin{aligned}
    \frac{\sigma^2x^2}{2}v_{xx} -\delta x v_x + g(x,p,\lambda) - \gamma^*(p,\lambda) & = 0, &&  x > a^*(p,\lambda), \\
    v_x & = q, && x \leq a^*(p,\lambda).
\end{aligned} \right.
\end{equation*}
Consider the pair $(X^{\nu^{*}},\nu^{*})$ solution to the Skorohod reflection problem at level $a^*(p,\lambda)$.
As $X^{\nu^{*}}_t \geq a^*(p,\lambda)$ $\prob$-a.s. and $\nu^{*}$ increases only on $\{X^{\nu^{*}} = a^*(p,\lambda)\}$, a standard verification theorem (\cite[Theorem 3.2]{jack_zervos2006ergodic}) implies that $\gamma^*(p,\lambda)$ is the value of the ergodic singular stochastic control problem with payoff \eqref{lemma:ergodic_control:payoff} and $\nu^{*}$ is optimal.
\end{proof}

\begin{proof}[Proof of Lemma \ref{lemma:first_order_condition}]
We deal with the case $q'<\infty$.
The case $q'=\infty$ is completely analogous.
We start by proving \ref{lemma:first_order_condition:necessary}.
Let $\hat{\nu}$ be optimal for the control problem with dynamics \eqref{mf:dynamics} and payoff functional $\Tilde{J}(\cdot,p,\lambda)$.
Recall the definition of the function $g(x,p,\lambda)$ in \eqref{control:auxiliary:running_cost}.
For any $\nu \in \A_{mf}$, $\eps \in (0,\frac{1}{2}]$, set $\nu^\eps = \eps\nu + (1-\eps)\hat{\nu}$.
Set
\begin{equation}
\begin{aligned}
    f(\eps,T) & = \frac{1}{\eps}\left(  \frac{1}{T}\E\left[ \int_0^T g(X^{\nu^\eps}_t,p,\lambda) dt - q \nu^\eps_T \right] -  \frac{1}{T}\E\left[ \int_0^T g(X^{\hat{\nu}}_t,p,\lambda)dt - q \hat{\nu}_T\right]  \right).
\end{aligned}
\end{equation}
For any $\eps \in (0,\frac{1}{2}]$ and $T > 0$, it holds
\begin{align*}
    f(\eps,T) &  = \frac{1}{\eps} \frac{1}{T} \left( \E\left[ \int_0^T \Big( \int_0^1 g_x(X^{\hat{\nu}}_t + \tau (X^{\nu^\eps}_t - X^{\hat{\nu}}_t ),p,\lambda)d\tau  \Big) (X^{\nu^\eps}_t - X^{\hat{\nu}}_t ) dt - q (\nu^\eps_T - \hat{\nu}_T) \right] \right) \\
    & = \frac{1}{T} \left( \E\left[ \int_0^T \Big( \int_0^1 g_x(X^{\hat{\nu}}_t + \tau (X^{\nu^\eps}_t - X^{\hat{\nu}}_t ),p,\lambda)d\tau  \Big) (X^{\nu}_t - X^{\hat{\nu}}_t) dt - q ( \nu_T - \hat{\nu}_T) \right] \right).
\end{align*}

We claim
\begin{equation}\label{first_order:dominated} 
    \lim_{\eps \downarrow 0} f(\eps,T) = \frac{1}{T} \E\left[ \int_0^T g_x(X^{\hat{\nu}}_t,p,\lambda) (X^{\nu}_t - X^{\hat{\nu}}_t) dt - q (\nu_T -\hat{\nu}_T ) \right]
\end{equation}
uniformly in $T$.
Indeed,
\begin{align*}
    & \left\vert f(\eps,T ) - \frac{1}{T} \E\left[ \int_0^T g_x(X^{\hat{\nu}}_t,p,\lambda) (X^{\nu}_t - X^{\hat{\nu}}_t) dt - q (\nu_T -\hat{\nu}_T ) \right] \right\vert \\
    & \leq \left\vert  \frac{1}{T} \E\left[ \int_0^T  (X^{\hat{\nu}}_t - X^{\nu}_t ) \int_0^1 \left( g_x(X^{\nu^\eps}_t + \tau (X^{\hat{\nu}}_t - X^{\nu^\eps}_t ),p,\lambda) -  g_x(X^{\hat{\nu}}_t,p,\lambda)  \right) d\tau dt  \right] \right\vert \\
    & \leq   \frac{1}{T} \E\left[ \int_0^T  \vert X^{\hat{\nu}}_t - X^{\nu}_t \vert \int_0^1 \left\vert g_x(X^{\nu^\eps}_t + \tau (X^{\hat{\nu}}_t - X^{\nu^\eps}_t ),p,\lambda) -  g_x(X^{\hat{\nu}}_t,p,\lambda)  \right\vert  d\tau dt  \right].
\end{align*}
By continuity of $g_x(x,p,\lambda)$ in $x$, the inner integral converges to $0$ as $\eps \downarrow 0$ for any $t \geq 0$, $\prob$-a.s.
By linearity of the dynamics \eqref{mf:dynamics}, since $X^\nu_t$ is positive for any $t$, it holds
\begin{align*}
    & X^{\nu^\eps}_t + \tau (X^{\hat{\nu}}_t - X^{\nu^\eps}_t ) = \tau X^{\hat{\nu}}_t + (1-\tau)  X^{\nu^\eps}_t \geq \frac{1}{2}X^{\nu^\eps}_t = \frac{1}{2}(\eps X^{\nu}_t +(1-\eps)X^{\hat{\nu}}_t ) \geq \frac{1}{2} X^{\hat{\nu}}_t \geq \frac{1}{4}X^{\hat{\nu}}_t.
\end{align*}
Since $\vert g_{xx}(y,p,\lambda) \vert = \alpha(1-\alpha) p x^{\alpha -2} \leq c \vert x\vert^{\alpha-2} $ for any $y \geq x$, $g_x(y,p,\lambda)$ is Lipschitz on $[x,+\infty)$ with Lipschitz constant $c \vert x\vert^{\alpha-2}$.
Then, it follows
\begin{equation}\label{lemma:first_order_condition:uniform_estimate}
\begin{aligned}
    & \vert f(\eps,T ) - g(T) \vert \leq  \frac{c}{T} \E\left[ \int_0^T  \vert X^{\hat{\nu}}_t - X^{\nu}_t \vert \vert X^{\hat{\nu}}_t \vert^{\alpha - 2} \int_0^1 \left\vert X^{\nu^\eps}_t + \tau (X^{\hat{\nu}}_t - X^{\nu^\eps}_t ) - X^{\hat{\nu}}_t \right\vert  d\tau dt  \right] \\
    & \leq c \frac{1}{T} \E\left[ \int_0^T \vert X^{\hat{\nu}}_t - X^{\nu}_t \vert \cdot \vert X^{\nu^\eps}_t - X^{\hat{\nu}}_t \vert \vert X^{\hat{\nu}}_t \vert^{\alpha - 2} dt  \right] \\
    & \leq c \left( \sup_{T > 0} \frac{1}{T} \E\left[ \int_0^T \vert X^{\hat{\nu}}_t \vert^{2q} dt  \right] + \sup_{T > 0} \frac{1}{T} \E\left[ \int_0^T \vert X^{\nu}_t \vert^{2q} dt  \right] \right)\sup_{T > 0} \left( \frac{1}{T} \E\left[ \int_0^T \vert (X^{\hat{\nu}}_t)^{\alpha -2} \vert^{q'} dt  \right] \right)^\frac{1}{q'} \eps ,
\end{aligned}
\end{equation}
where we used H\"{o}lder's inequality together with conditions \eqref{lemma:first_order_condition:admissible_controls} and \eqref{lemma:first_order_condition:integrability_condition}.
On the other hand, by taking the limit with respect to $T$, it holds
\begin{align*}
& \liminf_{T \uparrow \infty} f(\eps,T) \leq \frac{1}{\eps}\left(  \liminf_{T \uparrow \infty}\frac{1}{T}\E\left[ \int_0^T g(X^{\nu^\eps}_t,p,\lambda) dt - q \nu^\eps_T \right] \right.\\
& \left. - \liminf_{T \uparrow \infty}\frac{1}{T}\E\left[ \int_0^T g(X^{\hat{\nu}}_t,p,\lambda)dt - q \hat{\nu}_T\right]  \right) = \frac{1}{\eps} ( \Tilde{\J}(\nu^\eps) - \Tilde{\J}(\hat{\nu})) \leq 0,    
\end{align*}
by using the inequality $\liminf_n a_n - \liminf_n b_n \geq \liminf_n (a_n - b_n)$ and by optimality, for any $\eps \in (0,\frac{1}{2}]$.
Lemma \ref{lemma:moore_osgood_liminf} then implies 
\[
\liminf_{T \uparrow \infty} \frac{1}{T} \E\left[ \int_0^T g_x(X^{\hat{\nu}}_t,p,\lambda) (X^{\nu}_t - X^{\hat{\nu}}_t) dt - q (\nu_T -\hat{\nu}_T ) \right] \leq 0. 
\]
This concludes the proof of point \ref{lemma:first_order_condition:necessary}.

\medskip
As for point \ref{lemma:first_order_condition:sufficient}, assume that both \eqref{lemma:first_order_condition:limit} and \ref{lemma:first_order_condition:necessary} hold.
Then, by using the inequality $\limsup_{n} z_n -\liminf_n x_n \geq \limsup_{n}(z_n - y_n)$ (see \cite[Equation (4.25)]{karatzas1984monotone}) and concavity of $g$ jointly in $x$, it holds
\begin{align*}
    & \Tilde{\J}(\hat{\nu},p,\lambda) - \Tilde{\J}(\nu,p,\lambda) \geq \limsup_{T \uparrow \infty} \frac{1}{T}\E\left[ \int_0^T\left( g(X^{\nu}_t,p,\lambda) - g(X^{\hat{\nu}}_t,p,\lambda)\right) dt - q(\hat{\nu}_T - \nu_T) \right] \\
    & \geq \limsup_{T \uparrow \infty} \frac{1}{T}\E\left[ \int_0^T g_x(X^{\nu}_t,p,\lambda)(X^{\hat{\nu}}_t - X^{\nu}_t) dt - q(\hat{\nu}_T - \nu_T) \right] \geq 0,
\end{align*}
which implies that $\hat{\nu}$ is optimal within $\A_{2q}$.
If instead we assume \eqref{lemma:first_order_condition:ineq_liminf}, the claim follows from sublinearity of the inferior limit.
\end{proof}

\begin{proof}[Proof of Lemma \ref{central_planner:lemma:envelope_theorem}]
We verify the identity by explicit calculations.
First notice that
\begin{align*}
    \lambda(\theta) & = q\delta - \tonde{\frac{2\delta+\sigma^2}{2\delta}  }^{1- \alpha} \frac{2\delta\alpha}{2\delta + \sigma^2(1-\alpha)} \theta^{\alpha + \beta - 1} = q\delta - \alpha \liminf_{T \uparrow \infty} \frac{1}{T}\E\left[ \int_0^T (X^{\nu^{a(\theta)}}_t)^\alpha \theta^{\beta - 1} \right].
\end{align*}
On the other hand, we have
\begin{align*}
    f'(\theta) & = -q \delta + (\alpha + \beta )\frac{2\delta + \sigma^2}{2\delta + \sigma^2(1-\alpha)}\left( \frac{2\delta}{2\delta+\sigma^2}\right)^\alpha \theta^{\alpha+\beta-1} \\
    & = -q\delta + (\alpha + \beta )\liminf_{T \uparrow \infty} \frac{1}{T}\E\left[ \int_0^T (X^{\nu^{a(\theta)}}_t)^\alpha \theta^{\beta - 1} \right].
\end{align*}
By summing the two terms, we find
\[
f'(\theta) + \lambda(\theta) = \beta \liminf_{T \uparrow \infty} \frac{1}{T}\E\left[ \int_0^T (X^{\nu^{a(\theta)}}_t)^\alpha \theta^{\beta - 1} \right] = \liminf_{T \uparrow \infty} \frac{1}{T}\E\left[\int_0^T \pi_\theta(X^{\nu^{a(\theta)}}_t,\theta)dt \right].
\]
\end{proof}

\subsection*{Integrability Results}

\begin{lemma}\label{lemma:inv_gamma}
Let $\theta > 0$ and let $m'_{\infty,\theta}(x)$ be given by \eqref{inv_gamma:speed_measure}.
For any $k \geq 0$, the $\int_0^\infty x^k m'_{\infty,\theta}(x)dx$ is finite if and only if $2\delta - (k-1)\sigma^2 > 0$.
If so, it holds
\begin{equation}\label{inv_gamma:k_moment:expression}
    \int_0^\infty x^k m'_{\infty,\theta}(x)dx = \theta^{k - \frac{2\delta}{\sigma^2} -1 } \tonde{\frac{2\delta}{\sigma^2}}^{k -1 - \frac{2\delta}{\sigma^2} } \Gamma\tonde{\frac{2\delta}{\sigma^2} - k + 1}.
\end{equation}
\end{lemma}
\begin{proof}
Let $k \geq 0$ so that $2\delta - (k-1)\sigma^2 > 0$.
By setting $z=\sfrac{x}{\theta}$ in the integral in \eqref{inv_gamma:k_moment:expression}, we have:
\begin{align*}
\int_{\R_+}   x^k x^{ -\frac{2\delta}{\sigma^2}  - 2} \exp\tonde{ -\frac{2\delta}{\sigma^2} \frac{\theta}{x} } dx = \theta^{k - \frac{2\delta}{\sigma^2} -1 } \int_0^\infty z^{k - \frac{2\delta}{\sigma^2} -2 } \exp\tonde{ -\frac{2\delta}{\sigma^2} \frac{1}{z} }dz,
\end{align*}
and, by making the change of variables $t =\sfrac{2\delta}{\sigma^2} \cdot \sfrac{1}{z} $, we have
\begin{align*}
\int_0^\infty z^{k - \frac{2\delta}{\sigma^2} -2 } \exp\tonde{ -\frac{2\delta}{\sigma^2} \frac{1}{z} }dz = \tonde{\frac{2\delta}{\sigma^2}}^{k -1 - \frac{2\delta}{\sigma^2} }\int_0^\infty t^{\frac{2\delta}{\sigma^2} -k} e^{-t}dt = \tonde{\frac{2\delta}{\sigma^2}}^{k - 1 - \frac{2\delta}{\sigma^2} } \Gamma\tonde{\frac{2\delta}{\sigma^2} - k + 1},
\end{align*}
which yields \eqref{inv_gamma:k_moment:expression}.
\end{proof}

\begin{lemma}\label{N_player:regular:lemma_integrability}
For any $\theta > 0$, let $X^\theta = (X^\theta_t)_{t \geq 0}$ be the solution of 
\[
dX^{\theta}_t = \delta(\theta-X^\theta_t)dt + \sigma X^\theta_t dW_t, \quad X^\theta_0 = \xi.
\]
There exists a constant $C$ independent of $\theta$ so that it holds
\begin{equation*}
    \sup_{t \geq 0} \E[(X^\theta_t)^2]  \leq C(1 + \theta^2).
\end{equation*}
\end{lemma}
\begin{proof}
By It\^{o} formula, we have
\begin{align*}
    & d(X^{\theta}_t)^{2} = \left[-(2\delta - \sigma^2)(X^{\theta}_t)^{2} + 2\delta \theta (X^{\theta}_t) \right]dt + 2 \sigma  (X^{\theta}_t)^{2} dW_t, \quad (X_0^\theta)^{2}= \xi^{2},
\end{align*}
so that, by taking the expectation, we get
\[
    \E[(X^{\theta}_t)^{2}] = e^{-(2\delta - \sigma^2) t}\left( \E[\xi^2] + 2 \delta \theta \int_0^t \E[X^\theta_s] e^{-(2\delta - \sigma^2) s} ds \right).
\]
By \eqref{cce:regular:conditional_mean}, we have $\E[X^\theta_s] \leq C(1+\theta)$ for any $s \geq 0$ and, by Assumption \ref{assumption:dissipativity}, $2\delta - \sigma^2>0$.
Thus, it holds 
\begin{align*}
\E[(X^{\theta}_t)^{2}] & \leq e^{-(2\delta - \sigma^2) t}\left( \E[\xi^2] + C(1 + \theta)  2 \delta \theta \int_0^t e^{(2\delta - \sigma^2) s} ds \right) \\
& \leq C(1 + \theta^2) + e^{-(2\delta - \sigma^2) t} \left( \E[\xi^2] - C_1 (1 + \theta^{2}) \right) \leq C(1+\theta^2).
\end{align*}
This completes the proof.
\end{proof}

\subsection*{Auxiliary Results for the Backward Convergence Problem}

\begin{lemma}\label{N_player:lemma:ergodic_limit_q_power}
Suppose $\E[\theta_\infty^2]<\infty$.
Let $(X^i)_{i= 1}^N$ be i.i.d. as $Y$ conditionally to $\theta_\infty$, with either $Y=X^{\lambda^r}$ or $Y=X^{\lambda^s}$, and let $\kappa(dx,\theta)$ be equal to $p^r_\infty(dx,\theta)$ given by \eqref{cce:regular:limit_kernel} or equal to $p^s_\infty(dx,\theta)$ given by \eqref{cce:singular:limit_kernel}.
Let $\theta^{-i,N}_t = \frac{1}{N-1}\sum_{j \neq i}X^j_t$, for $t \geq 0$, $1 \leq i \leq N$.
Then, it holds
\begin{equation}\label{N_player:regular:ergodic_limit}
\begin{aligned}
    \lim_{ T \uparrow \infty} & \frac{1}{T} \int_0^T \E \left[ \left \vert \big( \theta^{-i,N}_t \big)^{\beta} - \theta_\infty^\beta \right \vert^2 \Big \vert \theta_\infty \right] dt  = \int_{\R_+^{N-1}} \Big \vert \big( \frac{1}{N-1}\sum_{j \neq i}x_j \big)^\beta - \theta_\infty^\beta \Big \vert^2 \bigotimes_{j \neq i}\kappa(dx_j,\theta_\infty), \; \prob\text{-a.s.}
\end{aligned}
\end{equation}
\end{lemma}
\begin{proof}
Suppose without loss of generality that $(\Omega,\F,\prob)$ is a Polish space (if it is not, we work on the canonical space).
Consider the regular conditional probability of $\prob$ given $\theta_\infty$. Denote the regular conditional probability of $\prob$ given $\theta_\infty = \theta$ by $\prob^\theta(\cdot) = \prob(\cdot \, \vert \theta_\infty = \theta)$, and by $\E^\theta[\cdot]$ the expectation with respect to the probability measure $\prob^\theta$.
Conditionally to $\theta_\infty = \theta$, we have that $(X^{j})_{j \neq i}$ are i.i.d. as $Y$; thus, in particular, the process $\theta^{-i,N}$ is a positively recurrent regular diffusion with ergodic measure $\bigotimes_{j \neq i} \kappa(dx_j,\theta)$ (see, e.g. \cite[Lemmata 23.17-19]{kallenberg_foundations}).
By the ergodic ratio theorem, it holds 
\begin{equation}\label{N_player:lemma:ergodic_limit_q_power:pointwise}
\lim_{ T \uparrow \infty} \frac{1}{T} \int_0^T \left \vert \big( \theta^{-i,N}_t \big)^{\beta} - \theta^\beta \right \vert^2 dt = \int_{\R_+^{N-1}} \Big \vert \big( \frac{1}{N-1}\sum_{j \neq i}x_j \big)^\beta - \theta^\beta \Big \vert^2 \bigotimes_{j \neq i}\kappa(dx_j,\theta), \quad \prob^\theta\text{-a.s.}.
\end{equation}
Therefore, convergence in probability with respect to the probability measure $\prob^\theta$ holds as well.
In order to get convergence in $L^1$ as well, we show that the family of random variables in the left hand-side of \eqref{N_player:lemma:ergodic_limit_q_power:pointwise} is uniformly integrable. 
By, e.g., \cite[Lemma 4.12]{kallenberg_foundations}, this implies that we can take the expectation with respect to $\prob^\theta$ and exchange the limit and expectation, to get
\begin{equation*}
\begin{aligned}
    \limsup_{ T \uparrow \infty} & \frac{1}{T} \int_0^T \E^\theta \left[ \left \vert \big( \theta^{-i,N}_t \big)^{\beta} - \theta_\infty^\beta \right \vert^2 \right] dt = \int_{\R_+^{N-1}} \Big \vert \big( \frac{1}{N-1}\sum_{j \neq i}x_j \big)^\beta - \theta^\beta \Big \vert^2 \bigotimes_{j \neq i}\kappa(dx_j,\theta),
\end{aligned}
\end{equation*}
holds for $\rho$-a.e. $\theta > 0$, which is equivalent to \eqref{N_player:regular:ergodic_limit}.
To verify uniform integrability, take $r = \sfrac{1}{\beta} > 1$.
By standard estimates, we have
\begin{align*}
    \E^\theta & \left[ \left\vert  \frac{1}{T} \int_0^T \left \vert \big( \theta^{-i,N}_t \big)^{\beta} - \theta^\beta \right \vert^2 dt \right\vert^r \right] \leq \frac{2^{2r -1}}{T} \int_0^T \left( \theta^{2 r \beta} + \E^\theta\left[ \big( \theta^{-i,N}_t \big)^{ 2 r \beta} \right] \right) dt \\
    & \leq C \left(  \theta^{2} + \sup_{t \geq 0} \E^\theta[\vert Y_t \vert^{2}] \right) \leq C (1+\theta^{2}),
\end{align*}
where last inequality holds true by Lemma \ref{N_player:regular:lemma_integrability} if $Y = X^{\lambda^r}$, and by Lemma \ref{geometric_bm:reflected:lemma} if $Y = X^{\lambda^s}$.
This implies that such family is bounded in $L^r$-norm, thus, since $r > 1$, uniformly integrable.
\end{proof}

\begin{lemma}\label{N_player:regular:lemma:convergence_empirical_measures}
Let $\kappa(dx,\theta)$ be either equal to $p^r_\infty(dx,\theta)$ given by \eqref{cce:regular:limit_kernel} or equal to $p^s_\infty(dx,\theta)$ given by \eqref{cce:singular:limit_kernel}.
Then, for any $\theta > 0$, it holds
\begin{equation}\label{N_player:regular:convergence_fixed_m}
    \lim_{N \to \infty} \int_{\R_+^{N-1}} \Big \vert \big( \frac{1}{N-1}\sum_{j \neq i}x_j \big)^\beta - \theta^\beta \Big \vert^{2} \bigotimes_{j \neq i} \kappa (dx_j,\theta) = 0.
\end{equation}
\end{lemma}
\begin{proof}
Let $(Y_i)_{i \geq 1}$ be a sequence of i.i.d random variables with law $\kappa(dx,\theta)$, defined on some probability space $(X,\mathcal{X},\mu)$.
Up to reindexing, the integral in \eqref{N_player:regular:convergence_fixed_m} can be expressed in terms of the expectation of a function of $\bar{Y}^n = \sfrac{1}{n}\sum_{i=1}^n Y_i$.
Since $(Y^i)_{i \geq 1}$ are i.i.d. as $\kappa(dx,\theta)$ and integrable, we have
\[
(\bar{Y}^n)^\beta \to \theta^\beta \quad \mu\text{-a.s.}
\]
by the strong law of large number and continuity of the function $x \mapsto x^\beta$.
Therefore, convergence in probability holds as well.
To conclude, we show that the sequence $( \vert (\bar{Y}^n)^\beta - \theta^\beta \vert^2 )_{n \geq 1}$ is uniformly bounded in $L^r$-norm, for some $r>1$, which implies that the sequence is uniformly integrable, and thus the convergence in $L^2$-norm holds by, e.g., \cite[Proposition 4.12]{kallenberg_foundations}.
Let $r = \sfrac{1}{\beta} > 1$.
By standard estimates, we have
\begin{equation*}
\begin{aligned}
    \E & [ (\vert (\bar{Y}^n)^\beta - \theta^\beta \vert^2)^r ] \leq 2^{2 r - 1 }\left( \theta^{2} + \E[(\bar{Y}^n)^{2}] \right) \leq C\left( \theta^2 +  \E[Y_1^{2}] \right),
\end{aligned}
\end{equation*}
where in the last inequality we used the identical distribution of the sequence $(Y_i)_{i \geq 1}$.
The expectation of $Y_1^2$ is finite by Lemma \ref{lemma:inv_gamma} if $\kappa = p^r_\infty$ and by Lemma \ref{geometric_bm:reflected:lemma} if $\kappa = \Tilde{p}$.
This concludes the proof.
\end{proof}

\subsection*{A Technical Result on the Exchange of Limits}

\begin{lemma}\label{lemma:moore_osgood_liminf}
Let $(a_{n,m})_{n,m \geq 1}$ be a real valued sequence. Suppose that the following holds:
\begin{enumerate}[label = \arabic*.]
    \item \label{lemma:moore_osgood_liminf:assumption_uniform} $\lim_{n \to \infty} a_{n,m} = b_m$ uniformly in $m$, and
    \item $\liminf_{m \to \infty} a_{n,m} = c_n$ for every $n \geq 1$.
\end{enumerate}
Then, it holds
\begin{equation}
    \liminf_{m \to \infty} \lim_{n \to \infty} a_{n,m} \leq \lim_{n \to \infty} \liminf_{m \to \infty} a_{n,m}.
\end{equation}
\end{lemma}
\begin{proof}
For any $n \geq 1$, consider a subsequence $(a_{n,m_k})_{n,k \geq 1}$ so that $c_n = \liminf_{m \to \infty} a_{n,m} = \lim_{k \to \infty} a_{n,m_k}$.
Since \ref{lemma:moore_osgood_liminf:assumption_uniform} is satisfied by the subsequence $(a_{n,m_k})_{n,k \geq 1}$ as well, Moore-Osgood theorem implies that there exists $A \in \R$ so that
\begin{equation}\label{lemma:moore_osgood_liminf:equality_subsequence}
     \lim_{n \to \infty}\lim_{k \to \infty} a_{n,m_k} = A = \lim_{k \to \infty} \lim_{n \to \infty}a_{n,m_k}.
\end{equation}
In particular, \eqref{lemma:moore_osgood_liminf:equality_subsequence} implies that $(b_{m_k})_{k \geq 1}$ is a convergent subsequence of $(b_m)_{m \geq 1}$.
Therefore, we have
\[
    \liminf_{m \to \infty} \lim_{n \to \infty} a_{n,m} = \liminf_{m \to \infty} b_m \leq \lim_{k \to \infty}b_{m_k} = \lim_{k \to \infty} \lim_{n \to \infty} a_{n,m_k} =  \lim_{n \to \infty} \lim_{k \to \infty}a_{n,m_k} = \lim_{n \to \infty} \liminf_{m \to \infty} a_{n,m} ,
\]
where last equality holds by definition of $(a_{n,m_k})_{n,m \geq 1}$. This concludes the proof.
\end{proof}

\vspace{0.25cm}

\subsection*{Acknowledgements}
The authors acknowledge the Deutsche Forschungsgemeinschaft (DFG, German Research Foundation) - Project-ID 317210226 - SFB 1283.

The first author acknowledges financial support by the European Union - NextGenerationEU - Project Title Probabilistic Methods for Energy Transition - CUP G53D23006840001 - Grant Assignment Decree No. 1379 adopted on 01/09/2023 by the Italian Ministry of University and Research (MUR).

\smallskip
We thank the two anonymous referees for their careful reading of the manuscript and their helpful suggestions.

\subsection*{Disclosure statement} The authors declare that none of them has conflict of interests to
mention.

\bibliographystyle{abbrv}
\bibliography{biblio}

\end{document}